\documentclass[11pt]{article}
\usepackage[margin=1.2 in]{geometry}  
\usepackage{graphicx}
\usepackage{amssymb,amsmath}
\usepackage[numbers,sort&compress]{natbib}
\usepackage{array}
\usepackage{multirow}
\usepackage[shortlabels]{enumitem}
\usepackage{color}
\usepackage{amsthm}
\usepackage[final]{showkeys}
\usepackage{dcolumn}    
\newcolumntype{d}{D{.}{.}{-1}}
\usepackage{hyperref}
\usepackage{calc}  

\usepackage[charter]{mathdesign}
\usepackage[mathcal]{eucal}

\usepackage{tikz}
\usetikzlibrary{arrows.meta} 
\tikzset{
    vertex/.style={circle, draw, fill=black!50, inner sep=0pt, minimum width=5pt}
}
\usetikzlibrary{calc}
\usepackage{pgfplots}

\numberwithin{equation}{section}

\newtheorem{theorem}{Theorem}[section]
\newtheorem{corollary}[theorem]{Corollary}
\newtheorem{lemma}[theorem]{Lemma}
\newtheorem{proposition}[theorem]{Proposition}

\newtheorem{remark}[theorem]{Remark}

\newtheorem{mainthm}{Theorem}
\newtheorem{mainprop}{Proposition}

\newtheorem*{definition*}{Definition}

\newcommand{\dd}[2]{\frac{d #1}{d #2}}
\newcommand{\dds}[2]{\frac{d^2 #1}{d #2^2}}

\newcommand{\grad}{\nabla}
\newcommand{\oneover}[1]{\frac{1}{#1}}
\newcommand{\half}{\frac{1}{2}}

\newcommand{\R}{\mathbb{R}}

\DeclareMathOperator*{\argmax}{arg\,max}

\newcommand{\minv}{\min_{|v|=1}}
\newcommand{\maxv}{\max_{|v|=1}}
\newcommand{\weak}[1]{\mbox{Weak}(#1)}
\newcommand{\strong}[1]{\mbox{Strong}(#1)}
\newcommand{\Sball}[1]{\mathbb S^{#1-1}}

\title{Almost-rigidity of frameworks}
\author{Miranda Holmes-Cerfon
\thanks{Courant Institute of Mathematical Sciences, New York University. \texttt{holmes@cims.nyu.edu}}
\and 
Louis Theran%
\thanks{School of Mathematics and Statistics, University of St Andrews. \texttt{lst6@st-and.ac.uk}}
\and 
Steven J. Gortler\thanks{School of Engineering and Applied Sciences, Harvard University. \texttt{sjg@cs.harvard.edu}. 
}}
\date{}                    

\begin{document}

\maketitle

\begin{abstract}
    We extend the mathematical theory of rigidity of frameworks (graphs embedded in $d$-dimensional space) to consider nonlocal rigidity and flexibility properties. We provide conditions on a framework under which (I) as the framework flexes continuously it must remain inside a small ball, a property we call ``almost-rigidity''; (II) any other framework with the same edge lengths must lie outside a much larger ball; (III) if the framework deforms by some given amount, its edge lengths change by a minimum amount; (IV) there is a nearby framework that is prestress stable, and thus rigid. The conditions can be tested efficiently using semidefinite programming. The test  is a slight extension of the test for prestress stability of a framework, and gives analytic expressions for the radii of the balls and the edge length changes.  Examples illustrate how the theory may be applied in practice, and we provide an algorithm to test for rigidity or almost-rigidity. We briefly discuss how the theory may be applied to tensegrities. 
\end{abstract}

\section{Introduction}

Frameworks are graphs with fixed edge lengths embedded in a Euclidean space. Studied for their mathematical properties for decades, they also arise as models in a great many applications: from engineering macroscale structures, such as sculptures \cite{tensegritree}, play structures \cite{Pietroni:2017epa}, or the Kurilpa bridge in Brisbane \cite{kurilpa};  to designing the microstructure of materials with novel properties
\cite{Paulose:2015hd,RayneauKirkhope:2017cy,Borcea:2017jf,bertholdi}, 
to studying arrangements of jammed particles \cite{Liu:2010jx,sticky,henkes}, 
to understanding allostery in biology \cite{Yan:2017cu,Rocks:2017bu}; to studying the properties of molecules \cite{Whiteley:2005hj,Sartbaeva:2006ks,HolmesCerfon:2017hz}; 
to analyzing the structure of proteins \cite{Jacobs:2001we}; 
to studying origami folding \cite{Demaine:2007jh}. 
A question of interest in all of these areas is whether a framework is \emph{locally rigid}, meaning the only continuous motions of the vertices that preserve the lengths of the edges are rigid-body motions. 
Mathematical rigidity theory provides various sufficient
conditions under which a framework is guaranteed to be
locally rigid. For so-called generic embeddings of the graph (roughly, random edge lengths), it is often sufficient to test a combinatorial property of the graph itself, but for specific embeddings the outcome of the tests depends on the particular edge lengths of interest. 

A challenge that arises when applying these tests to the problems  above, is that often one does not only wish to understand rigidity properties for a given, fixed set of edge lengths; 
one may wish to understand nearby configurations, or to allow the edge lengths to change. For example, in molecular dynamics, bonds between particles are never fixed at exactly one length, but rather they vibrate around that length, and this can change the geometry or topology of their configuration space. In engineering applications, one may design a framework on a computer to be rigid, but when one builds it the edges are always slightly different from the designed ones, so it may not be as rigid. 
Even when testing for rigidity on a computer, calculations can only be performed to finite precision, and even small rounding errors can change the outcome of a rigidity test. 


Fortunately, the strongest sufficient condition for local rigidity, infinitesimal rigidity, is robust to small perturbations.
However, a weaker but still important condition, called prestress stability, is not robust to any size of perturbation. Prestress stability 
has nothing to say about the rigidity or flexibility of nearby configurations, nor about how these might change if the edges can change lengths. In fact, as we show later, a prestress stable framework that is perturbed by even the smallest amount can become flexible.

Our main contribution is to introduce a theory aimed at understanding the rigidity properties of frameworks that are close to a framework of interest, and at controlling the edge changes for frameworks that undergo small deformations.  
We call this theory \emph{almost rigidity}, because it allows us to talk about frameworks that are not perfectly rigid, but nevertheless cannot flex very far in configuration space. The theory extends the theory of prestress stability to extract quantitative information about nonlinear deformations of a framework near a rigid configuration. 

Our contribution is laid out in four main theorems, labelled (I)--(IV), which are introduced informally here and  stated more precisely in Section \ref{sec:mainresults}. The theorems are illustrated  schematically in Figure \ref{fig:theorems}. 
The theorems concern a framework, which is a pair $(p,\mathcal E)$ such that $p=(p_1,p_2,\ldots,p_n)^T\in \R^{dn}$ is the \emph{configuration} of the framework, representing the positions of $n$ vertices in $d$-dimensional space, and 
$\mathcal E = \{(i_1,j_1),\ldots, (i_m,j_m)\}$ lists the $m$ edges of a graph connecting the vertices.  
We show that when certain sufficient conditions
are established for the framework, conditions that are similar to but less restrictive than the conditions for prestress stability \cite{Connelly:1996vja},
then:
\begin{enumerate}[(I)]
\item There is a ball of radius $\eta_1$, such that if $p$ flexes continuously with fixed edge lengths to some other configuration $p'$, then $|p'-p|\leq \eta_1$, where $|\cdot|$ is the $l_2$-norm. 
\end{enumerate}
We call a framework satisfying the conditions of (I) \emph{almost-rigid}, because it implies that even though the framework is flexible, it can't flex very much (since $\eta_1$ is usually much smaller than a typical edge length.) So (I) is a form of rigidity in that it constrains a configuration to stay in a small region of configuration space. 
\begin{enumerate}[(I)]\setcounter{enumi}{1}
\item There is a ball of radius $\eta_2>\eta_1$ such that if configuration $q$ has the same edge lengths as $p$, and $|q-p| > \eta_1$, then $|q-p|\geq \eta_2$. 
\end{enumerate}
This theorem is an extension of the concept of global rigidity, which concerns frameworks with exactly one embedding. When this theorem holds, there may be other embeddings of the framework (i.e. different from $p$), but if they are not sufficiently close to $p$, then they must lie sufficiently far away. 
\begin{enumerate}[(I)]\setcounter{enumi}{2}
\item If $p$ is deformed continuously to some configuration $q$ with the same edge lengths as $p$ and such that $|q-p| > \eta_2$, there must be some point $q'$ on the deformation path such that the vector of squared edge lengths, $e(q')$, satisfies $|e(q')-e(p)| \geq e_{\rm min}^*$. 
\end{enumerate}
This theorem gives a lower bound for how much we must stretch the edges of a framework to deform it into another, 
sufficiently distant, embedding with the same edge lengths. 
\begin{enumerate}[(I)]\setcounter{enumi}{3}
\item There is a nearby framework $q$ with $|q-p|\leq \eta_1$, such that $q$ is rigid. 
\end{enumerate}
This last theorem may be used as a numerical test for rigidity, if we are in the unfortunate case when the framework of interest is rigid, but when represented on a computer it is flexible. We give examples of such  frameworks in Section \ref{sec:introexamples}. 

We give analytic expressions for $\eta_1$,$\eta_2$, and $e_{\rm min}^*$, 
that can be evaluated with 
linear algebra and semidefinite programming.
The results in this paper can be applied without either testing exactly 
for zero, or determining in advance a cutoff value below which numerical quantities are treated as zero.  Any observed measurement of a real configuration or numerical approximation to coordinates can be used to obtain a meaningful conclusion. 


We prove Theorems (I)--(IV) by constructing an
energy function, $H(q)$, that depends only on the edge lengths of a configuration $q$.
We derive conditions under which  $H(q) > H(p)$ for $q$ in an annulus, $|q-p| \in (\eta_1,\eta_2)$. This is sufficient to prove Theorems (I),(II). To prove Theorem (III) we calculate a lower bound for $H(q)-H(p)$ along any path across the annulus. For Theorem (IV), which depends on stronger conditions than the other theorems, we show that $H(q)$ has an isolated local minimum at $p_{\rm pss}$ sufficiently close to $p$, which is sufficient for $p_{\rm pss}$ to be locally rigid. 
The function $H(q)$ has a physical interpretation as an energy, since it is constructed by pretending the edges of the framework are springs and some of the springs are under tension or compression. 

The rest of the paper proceeds as follows. First, we review various tests for local rigidity in Section \ref{sec:rigidity}, and give examples to show their limitations. We set up and state our main theorems in Section \ref{sec:mainresults}, and illustrate them on several examples in Section \ref{sec:examples}. The rest of the paper will be devoted to proving these theorems. To do this we develop some results concerning the shape of a general three-times continuously differentiable function
near a point that is nearly a critical point, in Section \ref{sec:energypropositions}. These results, which extend the second derivative test, may be of independent interest. In Section \ref{sec:energyfunction} we  construct a particular energy function $H(q)$ associated with a framework $(p,\mathcal E)$ and discuss its properties. Finally in Section \ref{sec:theoremproofs} we prove our main theorems by applying our general propositions from Section \ref{sec:energypropositions} to $H(q)$. 
We briefly explain how to extend these results to tensegrities in Section \ref{sec:tens}, and we conclude in Section \ref{sec:conclusion}. 

    \definecolor{myblue}{rgb}{0.2, 0.5, 0.91}  
    \definecolor{myblue1}{rgb}{0.4, 0.6, 0.8}  
    \definecolor{myblue2}{rgb}{0.36, 0.54, 0.66} 
    \definecolor{myblue3}{rgb}{0.19, 0.55, 0.91}  
    \definecolor{myblue4}{rgb}{0.0, 0.53, 0.74}  
    \definecolor{myblue5}{rgb}{0.29, 0.59, 0.82} 
    \definecolor{myblue6}{rgb}{0.0, 0.48, 0.65} 

    \definecolor{myred0}{rgb}{0.6, 0.25, 0.16} 
    \definecolor{myred1}{rgb}{0.54, 0.2, 0.14} 
    \definecolor{myred}{rgb}{0.7, 0.11, 0.11} 
    
    \definecolor{mygreen0}{rgb}{0.1, 0.7, 0.2} 
    \definecolor{mygreen1}{rgb}{0.0, 0.42, 0.24} 
    \definecolor{mygreen}{rgb}{0.0, 0.5, 0.0} 
    \definecolor{mygreen}{rgb}{0.2, 0.7, 0.2} 
     	
    \definecolor{mypurple1}{rgb}{0.36, 0.22, 0.33} 
    \definecolor{mypurple2}{rgb}{0.45, 0.31, 0.59} 
    \definecolor{mypurple}{rgb}{0.41, 0.21, 0.61} 

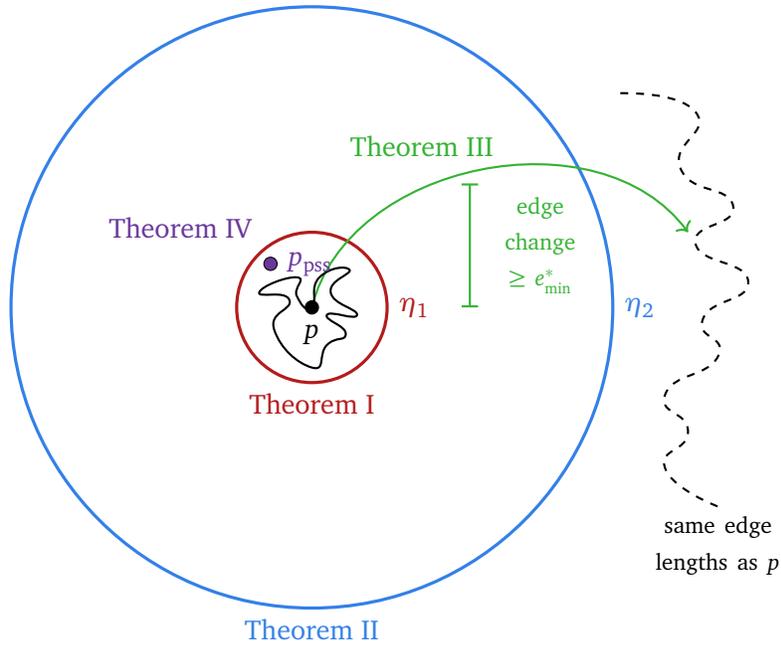
\begin{figure}
    \centering
    \begin{tikzpicture}
   \node[draw,circle,myred,very thick,minimum size=2cm,inner sep=0pt,label={[myred]right:$\eta_1$},label={[myred]below:Theorem \ref{thm:framework}}] at (0,0) {};
   \node[draw,circle,myblue,very thick,minimum size=8cm,inner sep=0pt,label={[myblue]right:$\eta_2$},label={[myblue]below:Theorem \ref{thm:outerbound}}] at (0,0) {};
 \draw[->,mygreen,thick] (0,0) to [bend left=65] node[pos=0.55, anchor=south east] {Theorem \ref{thm:emin}} (5,1) ;  
  \draw[{Bar[]}-{Bar[]},mygreen,thick] (2.1cm,0) -- node[right,text width=1.6cm,align=center] {\footnotesize edge change $\geq e_{\rm min}^*$} (2.1cm,1.65) ;
    \draw[thick] plot [smooth cycle, tension=1] coordinates {(0,0) (0.1,0.4) (0.5,0.5) (0.25,0.2) (0.6,-0.1) (0.2,-0.1) (0.45,-0.35) (0.2,-0.5) (0,-0.8) 
    (-0.5, -0.3) (-0.3,0.1) (-0.7,0.2)  (-0.2,0.4) (-0.2, 0.15)};
    \draw[xshift=4.6cm, yshift=-0.15cm, thick, dashed] plot [smooth, tension=1] coordinates {(-0.5,3) (0.5,2.75) (0.3, 2) (1, 1.5) (0.5, 1) (1.2, 0.5) (0.6, 0) (0.8, -0.5) (0.1, -1) (0.4, -1.5) (0.1, -2) (0.8, -2.5)} node[below,text width=2cm,align=center] {\footnotesize same edge lengths as $p$};
    \node[circle,draw,fill=black, inner sep=0pt, minimum width=5pt,label=below:{$p$}] (p) at (0,0) {};
    \node[circle,draw,fill=mypurple, inner sep=0pt, minimum width=5pt,label={[mypurple]right:{$p_{\rm pss}$}},
    label={[mypurple,label distance=.15cm]120:Theorem \ref{thm:pss}}] (ppss) at (-0.55,0.58) {};
    \end{tikzpicture}
    
  
  

    \caption{Schematic of our Theorems presented in Section \ref{sec:mainresults}, concerning the configuration space $\mathcal C^p\subset\R^{dn}$ near a given framework $p\in\R^{dn}$. 
   Framework $p$ may be flexible (solid black path), but {\color{myred} Theorem \ref{thm:framework}} says that if it flexes continuously with fixed edge lengths, it stays within a ball of radius {\color{myred} $\eta_1$.}
   There may be other configurations with the same edge lengths as $p$ (dashed black line), but {\color{myblue} Theorem \ref{thm:outerbound}} says that if they are not in the ${\color{myred}\eta_1}$-ball, then they lie outside a ball of radius {\color{myblue} $\eta_2$} $ > {\color{myred}\eta_1}$. 
   To deform continuously from $p$ to some configuration $q$ outside the ball of radius {\color{myblue} $\eta_2$}, {\color{mygreen} Theorem \ref{thm:emin}} says that the $l_2$-norm of the list of squared edge lengths must change by at least an amount {\color{mygreen}$e_{\rm min}^*$}. 
   Finally, {\color{mypurple} Theorem \ref{thm:pss}} says there is a configuration {\color{mypurple} $p_{\rm pss}$} within the ball of radius {\color{myred} $\eta_1$} that is prestress stable (and rigid). 
   All of these results depend on verifying conditions on $p$, which generalize the test for prestress stability, and the results of this test gives analytic formulas for ${\color{myred}\eta_1}$, {\color{myblue} $\eta_2$}, {\color{mygreen}$e_{\rm min}^*$}. 
    }
    \label{fig:theorems}
\end{figure}

\section{Review of mathematical rigidity theory}\label{sec:rigidity}

\subsection{First-order rigidity, second-order rigidity, and prestress stability}

\begin{figure}
    \centering

\begin{tikzpicture}[scale=1.25]
    \filldraw[color=yellow!80,fill=yellow!50, very thick] (0.95,1) ellipse (2.9 and 2.5);
    \filldraw[color=cyan!80,fill=cyan!40, very thick] (0.55,0.7) ellipse (2.25 and 1.85);
    \filldraw[color=blue!60,fill=blue!30, very thick] (0.25,0.5) ellipse (1.4 and 1.3);
  \filldraw[color=red!80,fill=red!30, very thick] (-0.15,0.2) ellipse (0.75 and 0.7) node[align=center,color=black] {1$^{\rm st}$ order\\rigid};
  \node[align=center] at (3.1,1.7) {rigid};
  \node[align=center] at (2,1.25) {2$^{\rm nd}$ order\\rigid};
  \node[align=center] at (0.9,0.9) {prestress\\stable};
\end{tikzpicture}

\begin{tikzpicture}[thick,scale=2]
\node at (-0.3, 1.2) {(A)};
\node[vertex,label=below:{1}] (1) at (0,0) {};
\node[vertex,label=below:{2}] (2) at (1,0) {};
\node[vertex,label=above:{3}] (3) at (1,1) {};
\node[vertex,label=above:{4}] (4) at (0,1) {};
\draw[ultra thick] (1) -- (2);
\draw[ultra thick] (2) -- (3);
\draw[ultra thick] (3) -- (4);
\draw[ultra thick] (4) -- (1);
\draw[ultra thick] (1) -- (3);
\end{tikzpicture}
\qquad
\begin{tikzpicture}[thick,scale=2.5]
\newcommand\ys{0}
\node at (-0.0, 0.96) {(B)};
\node[vertex,label=below:{1}] (1a) at (0,0-\ys) {};
\node[vertex,label=below:{2}] (2a) at (1,0-\ys) {};
\node[vertex,label=above:{3}] (3a) at (0.5,1-\ys) {};
\node[vertex,label=below:{4}] (4a) at (0.5,0.5-\ys) {};
\node[vertex,label=below:{5}] (5a) at (1/3,0-\ys) {};
\node[vertex,label=below:{6}] (6a) at (2/3,0-\ys) {};
\draw[ultra thick] (1a) -- (3a);
\draw[ultra thick] (2a) -- (3a);
\draw[ultra thick] (1a) -- (4a);
\draw[ultra thick] (2a) -- (4a);
\draw[ultra thick] (3a) -- (4a);
\draw[ultra thick] (1a) -- (5a);
\draw[ultra thick] (2a) -- (6a);
\draw[ultra thick] (5a) -- (6a);
\end{tikzpicture}
\qquad
\begin{tikzpicture}[thick,scale=2.5]
\node at (-0.0, 0.96) {(C)};
\node[vertex,label=below:{1}] (1) at (0,0) {};
\node[vertex,label=below:{2}] (2) at (1,0) {};
\node[vertex,label=above:{3}] (3) at (0.5,1) {};
\node[vertex,label=below:{4}] (4) at (0.5,0.5) {};
\node[vertex,label=below:{5}] (5) at (1/3,0) {};
\node[vertex,label=below:{6}] (6) at (2/3,0) {};
\draw[ultra thick,blue!25] (1) -- (3);
\draw[ultra thick,blue!25] (2) -- (3);
\draw[ultra thick,red!50] (1) -- (4);
\draw[ultra thick,red!50] (2) -- (4);
\draw[ultra thick,red] (3) -- (4);
\draw[ultra thick,blue!37.5] (1) -- (5);
\draw[ultra thick,blue!37.5] (2) -- (6);
\draw[ultra thick,blue!37.5] (5) -- (6);
\draw[->,black!30!yellow] (1/3+0.025,0) -- (1/3+0.025,0.2);
\draw[->,black!30!yellow] (6) -- (2/3,-0.2);
\draw[->,black!30!green] (5) -- (1/3,0.2);
\draw[->,black!30!green] (6) -- (2/3,0.2);
\end{tikzpicture}
\qquad
\begin{tikzpicture}[thick,scale=2.5]
\newcommand\ys{0}
\node at (-0.0, 0.96) {(D)};
\node[vertex,label=below:{1}] (1b) at (0,0-\ys) {};
\node[vertex,label=below:{2}] (2b) at (1,0-\ys) {};
\node[vertex,label=above:{3}] (3b) at (0.5,1-\ys) {};
\node[vertex,label=below:{4}] (4b) at (0.5,0.5-\ys) {};
\node[vertex,label=below:{5}] (5b) at (1/3,0.05-\ys) {};
\node[vertex,label=below:{6}] (6b) at (2/3,-0.05-\ys) {};
\draw[ultra thick] (1b) -- (3b);
\draw[ultra thick] (2b) -- (3b);
\draw[ultra thick] (1b) -- (4b);
\draw[ultra thick] (2b) -- (4b);
\draw[ultra thick] (3b) -- (4b);
\draw[ultra thick] (1b) -- (5b);
\draw[ultra thick] (2b) -- (6b);
\draw[ultra thick] (5b) -- (6b);
\end{tikzpicture}
    \caption{Schematic illustrating different kinds of local rigidity and how they are related. Examples show frameworks with different kinds of rigidity. (A) First-order rigid in $\R^2$ (but flexible in $\R^3$.) (B) Prestress stable in $\R^2$. (C) Same as B, with two infinitesimal flexes plotted as arrows (green, yellow) 
    and colors on the bars are proportional to values of the self-stress. Bright blue is the most positive (like a cable, that wants to shrink), and bright red the most negative (like a strut, that wants to expand.) (D) Same as B, but vertices 5, 6 have been perturbed vertically. 
    This framework is flexible in $\R^2$, because vertices 5 and 6 can move up and down, but they can't move by very much.}
    \label{fig:rigidity}
\end{figure}
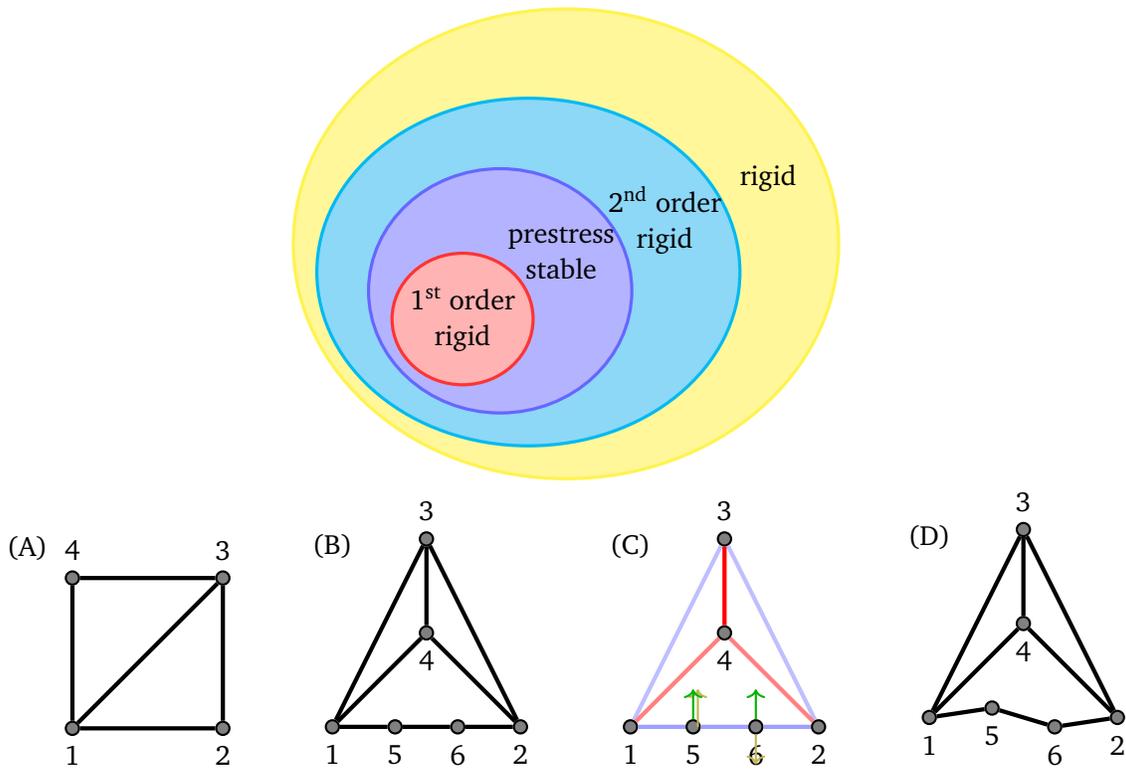

In this section we review the mathematical theory of local rigidity. We start from the definition that a framework $(p,\mathcal E)$ with edge lengths $d_{ij}$ 
is \emph{locally rigid} if it is an isolated solution (modulo translations and rotations) to the following system of algebraic equations: 
\begin{equation}\label{fij}
    f_{ij}(p) := |p_i-p_j|^2 - d_{ij}^2 = 0\,\qquad (i,j)\in \mathcal E\,.
\end{equation}
Testing for local rigidity is co-NP-hard in general \cite{Demaine:2007jh,Abbott:2008tj}. 
This leads to the study of sufficient conditions for rigidity that are easier to test.

To motivate these conditions it is helpful to consider a continuous, analytic deformation $p(t)$ with $p(0) = p$ (we are abusing notation here), and trying to solve \eqref{fij} locally.  Taking $\dd{}{t}$ of \eqref{fij} and evaluating at $t=0$ gives 
\begin{equation}\label{fijt}
    (p_i-p_j)\cdot (v_i-v_j) = 0\,\qquad (i,j)\in \mathcal E\,,
\end{equation}
where $v = p'(0)$ and $v_i\in \R^d$ are the components of $v$ corresponding to the $i$th vertex. A vector $v$ which solves \eqref{fijt} is called an \emph{infinitesimal flex}. Physically, a vector $v\in \R^{dn}$ corresponds to a set of velocities for the vertices. It is an infinitesimal flex if, when the framework is deformed as $x\to x+\delta v$ for some small $\delta$, the edge lengths do not change at $O(\delta)$. 
An infinitesimal flex corresponding to an infinitesimal rigid-body motion is called a \emph{trivial flex}; we write the subspace of trivial flexes $\mathcal T = \mathcal T(p)$. This space depends on the particular configuration, but we sometimes omit this dependence in the notation. 

Another way to write \eqref{fijt} is 
\begin{equation}\label{Rt}
    R(p)v = 0\,,
\end{equation}
where $R(p) \in \R^{m\times dn}$ is the \emph{rigidity matrix}. If the $k$th edge is $(i_k,j_k)$, then the $k$th row of the rigidity matrix is constructed as $R_{k,(d-1)i_k+1:di_k} = p_{i_k}-p_{j_k}$, $R_{k,(d-1)j_k+1:dj_k} = -(p_{i_k}-p_{j_k})$, with zeros everywhere else. 

If there are no  solutions $v$ to \eqref{Rt} except trivial flexes, then we say that $(p,\mathcal E)$ is \emph{first-order rigid} (also called \emph{infinitesimally rigid}.) 
An example of a framework that is first-order rigid is given in Figure \ref{fig:rigidity} (A). 
A theorem from rigidity theory  says that if $(p,\mathcal E)$ is first-order rigid, then it is rigid \cite[e.g.][]{Asimow:1978ena}. 
Clearly, testing for first-order rigidity is efficient whenever computing the null space of a matrix is efficient.  

A technical but important point is that in practice, to remove trivial flexes from consideration, one typically defines a subspace $\mathcal C$ to be a subspace complementary to $\mathcal T$, i.e. such that $\mbox{Span}\{\mathcal C,\mathcal T\} = \R^{dn}$ and $\mathcal C\cap \mathcal T = \{0\}$. 
Some possibilities include setting $\mathcal C = \mathcal T^\perp$, or pinning (freezing the values of) $d(d+1)/2$ vertex coordinates \cite{whitewhiteley,anchored}, or letting $\mathcal C$ be a random $dn-d(d+1)/2$-dimensional subspace. Then a framework is first-order rigid if $\mbox{Null}(R(p))\cap\mathcal C$ is empty.

If $(p,\mathcal E)$ is not first-order rigid, then it could be either rigid or flexible. To go further, take $\dds{}{t}$ of \eqref{fij} and evaluate at $t=0$: 
\begin{equation}\label{fijtt}
    (p_i-p_j)\cdot (a_i - a_j) + (v_i-v_j)\cdot (v_i-v_j) = 0\,\qquad (i,j)\in \mathcal E\,.
\end{equation}
Here $a = p''(0)$ has the physical interpretation of the acceleration of the vertices. 
Another way to write \eqref{fijtt} is 
\begin{equation}\label{Rtt}
    R(p)a = -R(v)v\,.
\end{equation}
If there are no solutions $(v,a)$ with non-trivial flexes $v$ to the system \eqref{Rt},\eqref{Rtt}, then the framework $(p,\mathcal E)$ is \emph{second-order rigid}. A theorem from rigidity theory says that if $(p,\mathcal E)$ is second-order rigid, then it is rigid \cite{Connelly:1980joa}. Neither this theorem, nor the one for first-order rigidity, are trivial statements, because it could be that $p$ lies at a ``cusp'' in configuration space, so that there is a continuous deformation $p(t)$ but its first $k$ derivatives are zero, i.e. $p'(0) = p''(0) = \cdots = p^{(k)}(0) = 0$, $p^{(k+1)} \neq 0$ \cite{Connelly:1994fz}.

Testing for second-order rigidity seems to be difficult.
To show second-order rigidity, one has to show that for each $v\in \mbox{Null}(R(p))\cap\mathcal C$, there is no solution $a$ to \eqref{Rtt}. 
If we fix $v$, then  \eqref{Rtt} becomes a linear system in $a$.  Since the column span of $R(p)$ is orthogonal to $\mbox{Null}(R^T(p))$, either
there is a solution or 
\begin{equation}\label{fredholm}
\exists \omega \in \mbox{Null}(R^T(p)) : \quad\; \omega^TR(v)v > 0\,,
\end{equation}
which witnesses that $R(v)v$ is not in the column span of $R(p)$.  (This is the Fredholm alternative.)

In principle, one could show that there is no solution to \eqref{Rtt} for \emph{any} $v\in \mbox{Null}(R(p))\cap \mathcal C$, by finding an $\omega$ solving \eqref{fredholm} for each such $v$, but 
without fixing $v$,
this is a nonlinear problem with no known efficient
algorithm to solve it.

Therefore one usually considers a version of rigidity that is stronger than
second-order rigidity, but still weaker than first-order rigidity.
Suppose that
\begin{equation}\label{pss0}
  \exists \omega \in \mbox{Null}(R^T(p)): \qquad  \omega^TR(v)v>0 \qquad\forall v \in \mbox{Null}(R(p))\cap \mathcal C\,.
\end{equation}
That is, a single $\omega$ works for all $v$ in \eqref{fredholm}, so clearly \eqref{fredholm} holds for each $v$. So, $(p,\mathcal E)$ is second-order rigid, and hence rigid. A framework satisfying \eqref{pss0} is said to be \emph{prestress stable}. An example of a framework that is prestress stable is given in Figure \ref{fig:rigidity} (B). This figure also shows how the various flavours of local rigidity are nested within each other. 

An important element in \eqref{pss0} is the vector $\omega$. Any vector $\omega\in \R^m$ is called a \emph{stress}, and if in addition $\omega \in \mbox{Null}(R^T(p))$ it is called a \emph{self-stress} (also called an \emph{equilibrium stress}.)
Physically, a stress is an assignment of tensions to the edges, so the edges behave like springs which are under tension or compression. It is a self-stress if the forces on the edges are balanced at each node, so the framework is local force equilibrium with that assignment of tensions. To see the connection between forces and the left null space of the rigidity matrix, write $(\omega^TR(p))_i = \sum_{j:(i,j)\in \mathcal E}\omega_{ij}(p_i-p_j)$, where $\omega_{ij}$ is the component of $\omega$ associated with edge $(i,j)$. Since $\omega_{ij}(p_i-p_j)$ is the force applied at vertex $i$ by the edge $(i,j)$, if $\omega$ is a self-stress, the sum of the forces at each vertex is zero.

One major advantage of prestress stability is that it is possible to test
this property 
using semidefinite programming, which is efficient for many problems. 
To see how, observe that the term $\omega^TR(v)v$ is quadratic in velocities and linear in stresses. Therefore, we can write this term as $v^T\Omega(\omega)v$, where $\Omega(\omega)\in \R^{dn\times dn}$ is the \emph{stress matrix} associated with a stress $\omega$, defined so that it acts on flexes as $u^T\Omega v = \sum_{(i,j)\in\mathcal E} \omega_{ij}(u_i-u_j)\cdot(v_i-v_j)$. Clearly the stress matrix is symmetric. 
Then, \eqref{pss0} can be written as 
\begin{equation}\label{pss}
    \exists \omega \in \mathcal W^0 : \quad\; v^T\Omega(\omega)v > 0 \qquad\forall v \in \mathcal V^0\,,
\end{equation}
where we defined the linear spaces of nontrivial infinitesimal flexes, and self-stresses respectively, as 
\begin{equation}\label{VW0}
 \mathcal V^0 = \mbox{Null}(R(p))\cap \mathcal C\,,\qquad    \mathcal W^0 = \mbox{Null}(R^T(p))\,.
\end{equation} 
That is, we must find a matrix $\Omega$ within a linear space of symmetric matrices, that is positive definite on a linear subspace. This is a convex problem, so may be solved using algorithms that have polynomial complexity and are often efficient in practice. 

Here is one algorithm to solve it. 
Let $n_v = \mbox{dim}(\mathcal V^0)$, $n_w = \mbox{dim}(\mathcal W^0)$; we assume  $n_v,n_w\geq 1$. Let $I_{n_v}\in \R^{n_v\times n_v}$ be the identity matrix. 
Let $V$ be a matrix whose columns form an orthonormal basis of $\mathcal V^0$, and let $\{w_i\}_{i=1}^{n_w}$ be an orthonormal basis of $\mathcal W^0$. Let $M_i = V^T\Omega(w_i)V$.\footnote{In practice it helps to symmetrize this matrix, as $M_i \to (M_i+M_i^T)/2$, to get rid of imaginary eigenvalues that can arise because of numerical rounding errors.}
Solve the following problem for unknown variables $t\in \R$, $a\in \R^{n_w}$, $X\in \R^{n_v\times n_v}$: 
\begin{equation}\label{pssalg}
\begin{aligned}
& \underset{t,a,X}{\text{maximize}}
& & t \\
& \text{subject to}
& & X = \sum_{i=1}^{n_w}a_iM_i - t I_{n_v}, \\
&&& X \succeq 0, \\
&&& \left( \sum_{i=1}^{n_w}a_i^2\right) ^{1/2} \leq 1 . 
\end{aligned}
\end{equation}
The notation $A\succeq\lambda$ ($A\succ\lambda$) for a symmetric matrix $A$ means that $A-\lambda I$ is positive semidefinite (positive definite), where $I$ is the identity matrix with the same dimensions as $A$. Equivalently, $\lambda_{\rm min}(A) \geq \lambda$, where $\lambda_{\rm min}(A)$ is the smallest eigenvalue of a matrix $A$.

The above optimization program can be input nearly verbatim into standard convex optimization software; we use CVX for the examples in this paper \cite{cvx,gb08}. The optimal stress and eigenvalue are extracted from the optimal solution $(t_{\rm opt}, a_{\rm opt}, X_{\rm opt})$ as 
\[
\omega = \sum_{i=1}^{n_w}(a_{\rm opt})_iw_i\,,\qquad
\lambda_0 = t_{\rm opt}\,.
\]
The test is successful if $t_{\rm opt} > 0$. 

There are two special cases for which the convex optimization program \eqref{pssalg} need not be run:
\begin{itemize}
\item If $n_w = 1$, so there is only one self- stress $w_1$, then we merely check the eigenvalues of $M_1$. If they are all of the same sign, say with absolute values $0 < \lambda_1 \leq \lambda_2 \leq \cdots$, then we set $\lambda_0 = \lambda_1$, and $\omega = \mbox{sgn}(\lambda_1) w_1$. If the eigenvalues are not all of the same sign, or if some are zero, the test is unsuccessful. 

\item If $n_v = 1$, so there is only one infinitesimal flex, then each $M_i$ is a scalar. If one of the $M_i$ is nonezro, set $\lambda_0 = \max_i|M_i|$, and set $\omega = \mbox{sgn}(M_i) w_i$, where $i=\argmax_i|M_i|$. If all the $M_i$ are zero, the test is unsuccessful. 

\end{itemize}

\subsection{A challenge in testing for prestress stability}\label{sec:introexamples}

\begin{figure}
    \centering
    \includegraphics[trim={3cm 0 3cm 0},clip,width=0.3\textwidth]{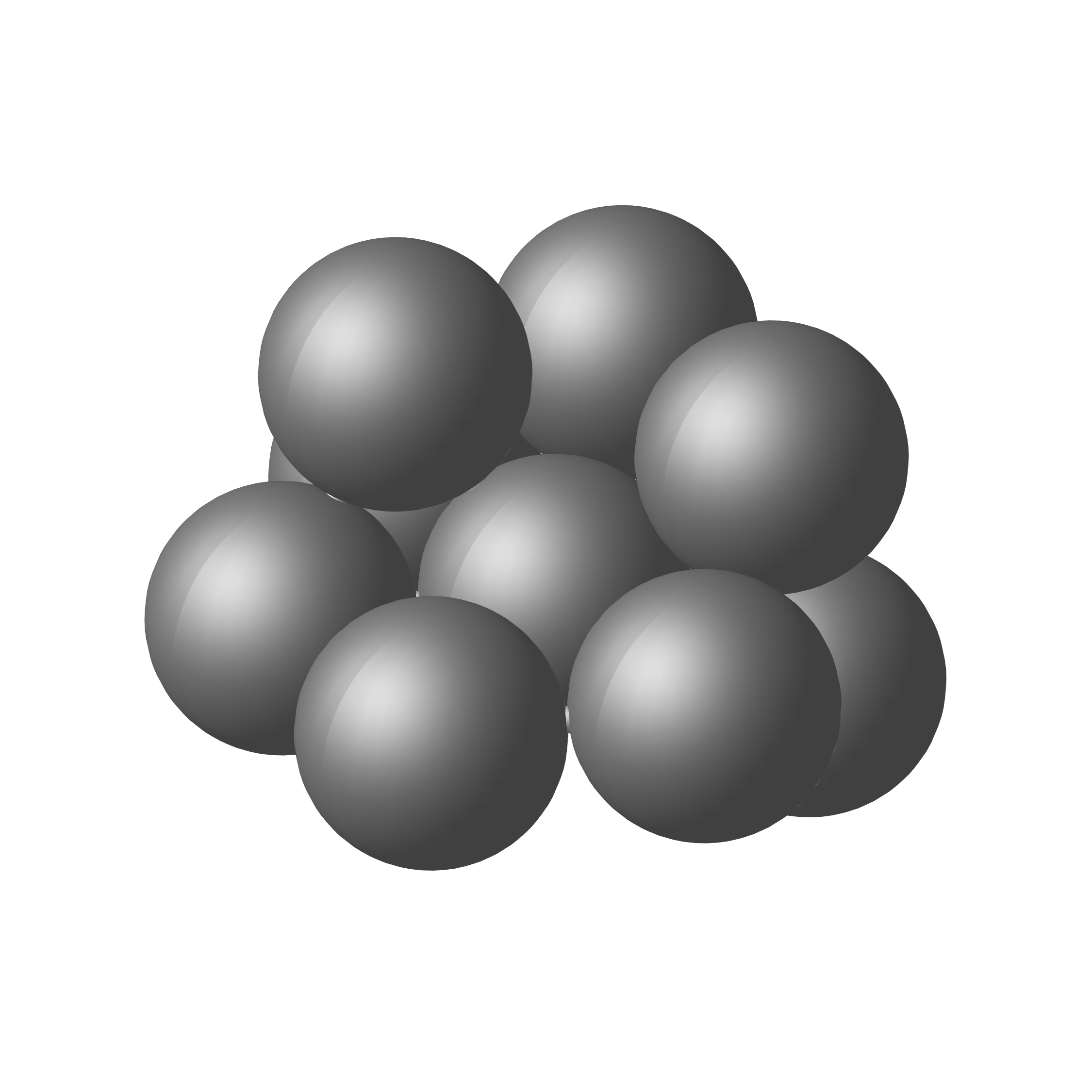}\hfill
    \includegraphics[trim={4cm 0 4cm 0},clip,width=0.3\textwidth]{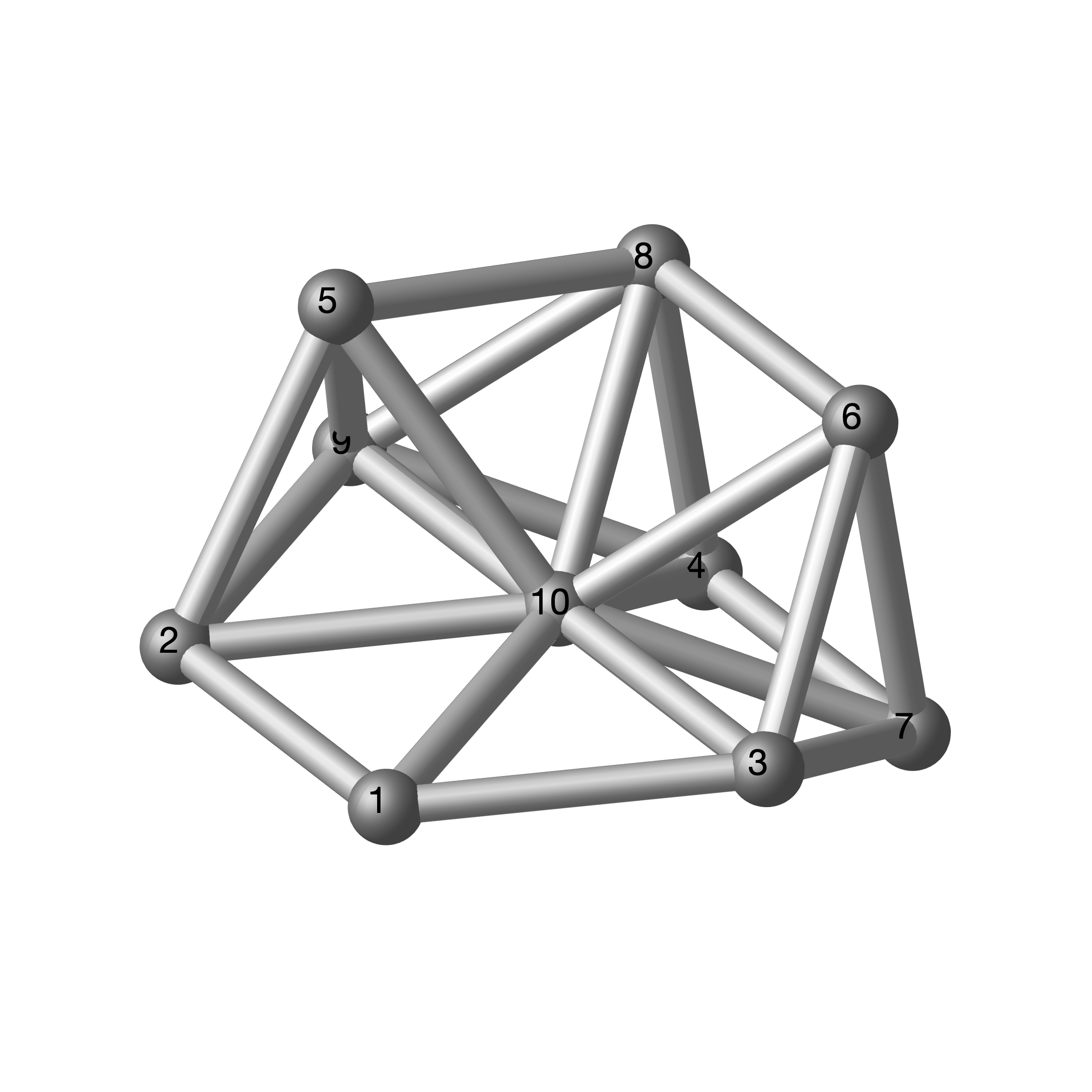}\hfill
    \includegraphics[trim={4cm 0 4cm 0},clip,width=0.3\textwidth]{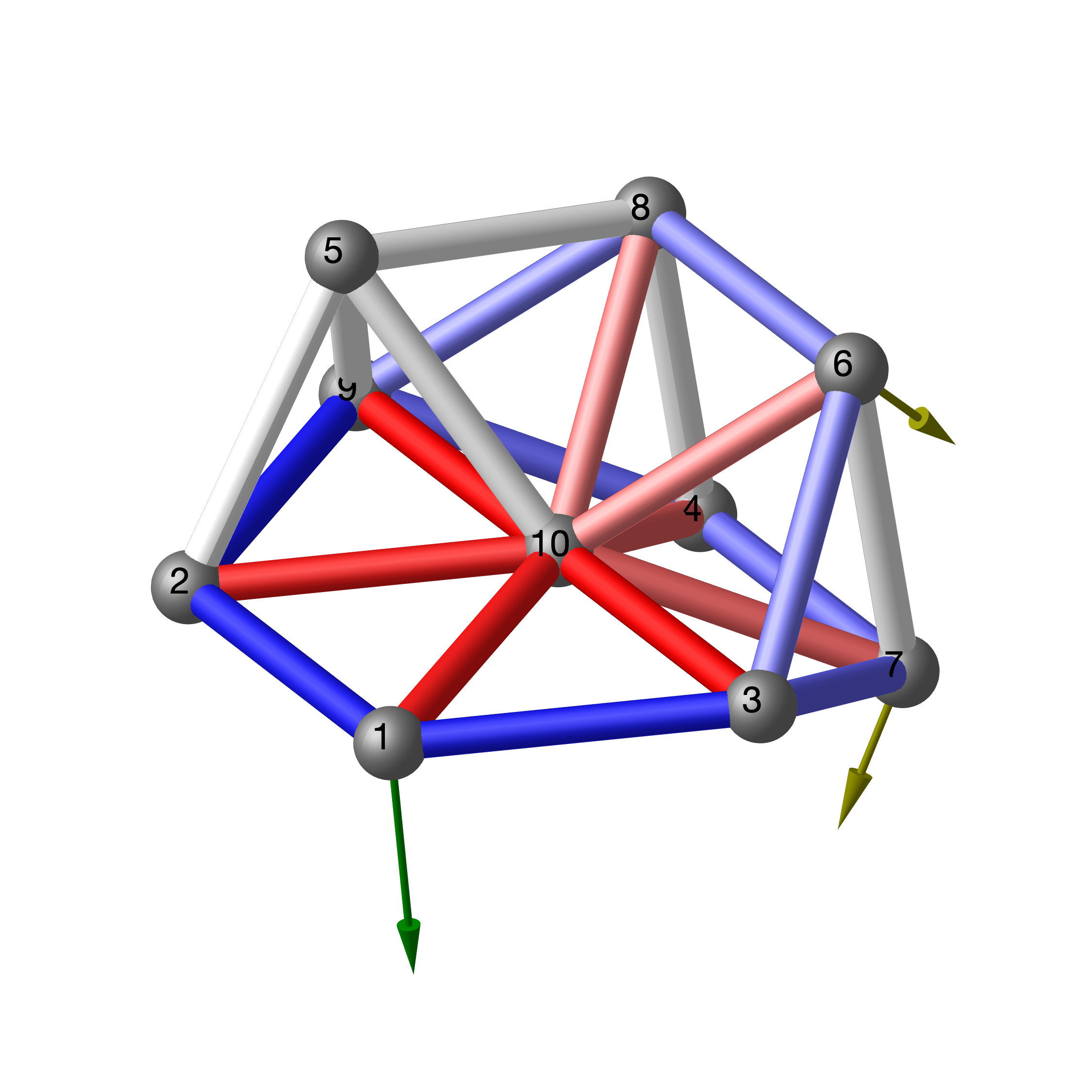}
    \caption{Left: A packing of $N=10$ unit spheres that has $3N-7=23$ contacts, fewer than the $3N-6$ needed generically to be rigid. Nevertheless, this packing is rigid. Middle: a representation of the packing as a framework, with vertices at the sphere centers and edges between spheres in contact. Right: the self-stress, as colors on bars, (blue is positive, like a cable that wants to shrink, and red is negative, like a strut that wants to expand), and the two infinitesimal flexes (yellow and green arrows.) The infinitesimal flexes are calculated assuming spheres 2,10,3 are pinned.}
    \label{fig:n10hypo}
\end{figure}

In this section we introduce two examples to illustrate a challenge that can occur in testing for rigidity when a prestress stable framework is perturbed. 

\subsubsection{Example 1}\label{sec:example1}
Consider framework (B) in Figure \ref{fig:rigidity}. This framework has 6 vertices, 9 edges, and 3 rigid-body degrees of freedom. When the vertices are placed 
generically (e.g. randomly), it has $6\times 2 - 3 - 8 = 1$ nontrivial degree of freedom, making it flexible. However, for the particular embedding in the figure, where vertices $1,5,6,2$ are colinear, the framework acquires one self-stress and a second non-trivial
infinitesimal-flex (Figure \ref{fig:rigidity}(C)), and the corresponding stress matrix is positive definite on the flex space as in \eqref{pss0}. Therefore, this framework is prestress stable, and hence rigid. 

If we perturb this framework, however, it is no longer rigid. For example, in Figure \ref{fig:rigidity}(D), vertices 5,6 have been perturbed vertically by 5\% of the length of edge 1-2, so they are no longer colinear with vertices 1,2. We calculated the left and right null spaces of the rigidity matrix $R(p)$ on a computer, and found the framework has a one-dimensional space of nontrivial flexes, and a zero-dimensional space of self-stresses. Therefore it cannot be prestress stable; indeed it 
flexible. This is true no matter how little vertices 5,6 are perturbed vertically;
even a perturbation as small as $10^{-14}$ destroys the extra flex and stress (results from \cite{whitewhiteley} give a formal proof.)

Notably, in this example, it is also clear that if vertices 1,2 are pinned (to remove the trivial Euclidean motions), then vertices 5,6 \emph{can} move, \emph{but not very far}. As the framework deforms, it remains ``close to'' the prestress stable configuration B, even though it never adopts this configuration.  Yet, the rigidity test above is blind to the fact that the framework is constrained to move in a small region, nor can it detect there is a nearby prestress stable configuration.

This example is somewhat contrived, so we turn to a more physically motivated example with the same basic properties. 

\subsubsection{Example 2} \label{sec:example2}
Consider the cluster formed from $N=10$ unit spheres in 3 dimensions shown in Figure \ref{fig:n10hypo}.  Clusters such as these are of interest in materials science, since the arrangements of small groups of particles can explain certain macroscopic properties of materials, such as how they form gels and glassy phases \cite{PatrickRoyall:2008fz,Manoharan:2015ko,Robinson:2019et}. This particular sphere cluster was discovered in \cite{HolmesCerfon:2016wa}, which presented a nearly complete dataset of rigid packings of $N\leq 14$ unit spheres. 

The sphere cluster can be represented as a framework by putting a vertex at the center of each sphere, and an edge between spheres in contact. The induced framework has 10 vertices, so requires 30 variables to describe the positions of the vertices. It has 23 edges, and when the vertices are placed 
generically, this framework  has 30-23=7 degrees of freedom. Of these,  6 correspond to rigid-body motions (three rotations and three translations), so generically this framework is flexible, with 1 nontrivial degree of freedom. 
Nevertheless, when this framework is formed from edges that are all unit length, it can be argued to be rigid (see \cite{HolmesCerfon:2016wa}.)


The numerical configuration of the framework that is plotted was found on a computer by solving \eqref{fij} numerically 
with $d_{ij} = 1$. 
Because computers have finite precision, the solver did not find the exact solution $p_{\rm true}$. Rather, it found an approximate solution $p$ such that 
\begin{equation}\label{approxedgeeqn}
|f_{ij}(p)| =  \big||p_i-p_j|^2 - d_{ij}^2\big| \leq \delta\,,\qquad (i,j)\in \mathcal E\,,
\end{equation} 
where $\delta>0$ is a numerical tolerance parameter, chosen to be $\delta = 4\varepsilon \approx 9\times 10^{-16}$ where $\varepsilon\approx 2.2204\times 10^{-16}$ is machine precision. The numerical solution $p$ is a perturbation of $p_{\rm true}$, so it should behave like a generic embedding, and should be flexible. Indeed, we calculated the rigidity matrix $R(p)$ and found it has a one-dimensional right null space, and a zero-dimensional left null space, so $p$ cannot be prestress stable. 

However, if we look at the singular values of $R(p)$, they tell a more nuanced story. The singular values in increasing order 
are
\begin{gather*}
6.6\times 10^{-9},\; 
    0.47,\;
    0.54,\;
    0.76,\;
    0.82,\;
    0.90,\;
    1.01,\;
    1.15,\;
    1.18,\;
    1.19,\;
    1.28,\;
    1.30,\;
    1.35,\; \\
    1.45,\;
    1.48,\;
    1.59,\;
    1.66,\;
    1.72,\;
    1.77,\;
    1.86,\;
    1.90,\;
    2.22,\;
    2.33\,.
\end{gather*}
The smallest singular value, $\sigma_0=6.6\times 10^{-9}$, is many orders of magnitude smaller than the second-smallest, $0.47$. 
The right and left singular vectors corresponding to $\sigma_0$, call them $v$ and $w$ respectively, are ``almost'' a true flex and stress. That is, if we flex the framework in direction $v$, the edge lengths change at a rate proportional to $\sigma_0$, hence very little. Similarly, if we put the edges under tension according to stress $w$, the imbalance of forces at each node is proportional to $\sigma_0$, hence very small. Can the almost-flex and almost-stress be used to gain more information about how framework $p$ can actually flex? 

A natural idea is to include the almost-flex $v$ and almost-stress $w$ in the test  \eqref{pss} for prestress stability. 
To this end we formed the  ``almost flex space'' $\mathcal V=(\mbox{Null}(R(p)) \oplus v)\ \cap \mathcal C$, and the ``almost self-stress space'' $\mathcal W=\mbox{Null}(R^T(p))\oplus w$. 
Then $\mbox{dim}(\mathcal V) = 2$, $\mbox{dim}(\mathcal W) = 1$. 
We performed test \eqref{pss} using $\mathcal V,\mathcal W$ in the place of $\mathcal V^0,\mathcal W^0$, by forming the ``almost'' self-stress matrix $\Omega = \Omega(w)$ (setting $|w|=1$), and computing the eigenvalues of $V^T\Omega V$, where $V$ is an orthonormal basis of $\mathcal V$. We obtained
\[
\mbox{eig}(V^T\Omega V) = \{0.33, 0.46\}\,.
\]
The almost-self-stress matrix $\Omega$ is positive definite on the almost-flex space $\mathcal V$. 

\medskip

Unfortunately, 
this test doesn't show that $p$ is rigid, since $\mathcal V$, $\mathcal W$ are not true flexes and self-stresses. What then does it show? 

One approach to answering this question (an approach we will not follow in this paper) would be to assume that $|p-p_{\rm true}|$ is sufficiently small, and in addition, that any sufficiently small singular values of $R(p)$ are perturbations of zero singular values of $R(p_{\rm true})$. Then, $v$ and $w$ are simply perturbations of a true flex $v_{\rm true}$ and a true stress $w_{\rm true}$ of $p_{\rm true}$. One can then use perturbation bounds from linear algebra to show that $|v-v_{\rm true}|$, $|w-w_{\rm true}|$ are both small, so the eigenvalues of $V^T\Omega V$ are close to those of $V^T_{\rm true}\Omega_{\rm true} V_{\rm true}$. Therefore, if $V^T\Omega V$ is sufficiently positive definite, then  $V^T_{\rm true}\Omega_{\rm true} V_{\rm true}$ is also positive definite. Hence, $p_{\rm true}$ is prestress stable. 

 We will not pursue this approach here, because it requires making the unsatisfying assumption that all small singular values are actually perturbations from zero, an assumption which could never be verified in practice. Instead, we will take a more global approach, and ask what information we can extract from the set of small singular values of $R(p)$ about the amount by which a framework can deform, without making any assumptions about where these small singular values came from. 
 This approach will allow us to prove statements about the rigidity and flexibility of nearby frameworks, and to make more global statements about how much the edges can deform while remaining close to $p$. We will introduce a new notion of rigidity that applies to frameworks that are only slightly flexible, such as the ones in our examples, a notion we call \emph{almost-rigidity}.

\section{Main results}\label{sec:mainresults}

In this section we present our main theorems. They are illustrated schematically in Figure \ref{fig:theorems}. 

\subsection{Setup}\label{sec:setup}

We start with a framework $(p,\mathcal E)$ and the corresponding rigidity matrix $R(p)$. 
We identify a subspace $\mathcal C \subset \R^{dn}$ that is complementary to the space  $\mathcal T$ of trivial motions at $p$. We will almost always use $\mathcal C = \mathcal T^\perp$ in our examples, but the theory holds with more general choices of $\mathcal C$  (for example, by fixing coordinates of some of the points.)
%
To remove trivial degrees of freedom, we restrict the framework's configuration space to the affine space $\mathcal C^p$ defined by 
\begin{equation}\label{Cdef}
    \mathcal C^p = \{q\in \R^{dn}: q=p+v, v\in \mathcal C\}\,.
\end{equation}  
We would like $\mathcal C^p$ to be homeomorphic to the quotient space $\R^{dn}$ modulo the space of $d$-dimensional rotations and translations of the framework, which is a manifold almost everywhere. 
Although this cannot hold globally, 
we expect  $\mathcal C^p$ to have this property
in a sufficiently large neighborhood of $p$.\footnote{For example, if we construct $\mathcal C^p$ for a two-dimensional framework by pinning one vertex to the origin, and another vertex to the $x$-axis, then when these two vertices are coincident, $\mathcal C^p$ contains all the rotations of this configuration about the origin. Conversely, for some choices
of $\mathcal C^p$, this space could fail to include some configurations and its Euclidean transformations entirely.}

We aim to generalize the test \eqref{pss} for prestress stability. Recall that this test required as input subspaces 
$\mathcal V^0,\mathcal W^0$ containing the infinitesimal flexes and  self-stresses, see \eqref{VW0}. 
Suppose we enlargen these subspaces to $\mathcal V\subset \mathcal C$, $\mathcal W$, in such a way that $\mathcal V$ contains the infinitesimal flexes, $ \mathcal V^0\subset \mathcal V$. We impose no condition on $\mathcal W$. 
We call $\mathcal V$ the \emph{almost-flex space} and $\mathcal W$ the \emph{almost-self-stress space}, because we have in mind that these spaces contain velocities and stresses that are almost, but not quite, infinitesimal flexes and self-stresses. That is, if $v\in \mathcal V$, $w\in \mathcal W$, we want $|R(p)v|$, $|w^TR(p)|$ to be ``small'' in some sense.
One way to construct these subspaces is to  choose a singular value $\sigma$ of $R(p)$, 
then build the spaces out of singular vectors corresponding to singular values less than or equal to $\sigma$, as 
\begin{equation}\label{VWsig}
    \mathcal V^{\sigma} = \{v\in \mathcal C : \frac{|R(p)v|}{|v|}\leq \sigma\}\,,\qquad
    \mathcal W^{\sigma} = \{w\in \R^m : \frac{|w^TR(p)|}{|w|}\leq \sigma\}\,.
\end{equation}

Now, choose $\omega\in \mathcal W$ with stress matrix $\Omega=\Omega(\omega)$, such that $\Omega$ is positive definite on $\mathcal V$. That is, if $V$ is a matrix whose columns are an orthonormal basis of $\mathcal V$, we require that 
\begin{equation}\label{OmegaV}
V^T\Omega V \succ 0\,.
\end{equation}
It is possible to find such an $\omega$, if it  exists, by solving \eqref{pss} using $\mathcal V$, $\mathcal W$ in place of $\mathcal V^0$, $\mathcal W^0$, with the algorithm given in \eqref{pssalg}.  

We remark that we don't actually need a subspace $\mathcal W$ for the theorems to hold in theory -- we only need a stress solving \eqref{OmegaV}. We introduce $\mathcal W$ because this lets us find $\omega$ in practice.

Let $\lambda_0$ be the minimum eigenvalue of $\Omega$ on $\mathcal V$, and let $\mu_0$ be the (possibly negative) minimum eigenvalue of $\Omega$ over all of $\mathcal C$:
\begin{equation}\label{lam0mu0}
\lambda_0 = \lambda_{\rm min}(V^T\Omega V)\,, \qquad
\mu_0 = \lambda_{\rm min}(C^T\Omega C) \,.
\end{equation}
Here $C$ is a matrix whose columns form an orthonormal basis of $\mathcal C$. 
Choose $\lambda\in(0,\lambda_0)$, and let $\kappa$ be the solution to the following problem: 
\begin{equation}\label{kappa}
\min \kappa : \quad\text{s.t. }\quad C^T \Omega C + 2\kappa C^TR(p)^TR(p)C \succeq \lambda\,.
\end{equation}
The left-hand side of \eqref{kappa} will turn out to be half the Hessian at $p$ of the function $H:\R^{dn}\to \R$ that we use to prove our theorems, and the quantity $\kappa$ is the smallest number that makes this Hessian sufficiently positive definite.
The solution to \eqref{kappa} exists and is nonnegative by Lemma \ref{lem:AB} in Section \ref{sec:theoremproofs}, applied with $A=\Omega$, $B=2R(p)^TR(p)$. 
It can be found efficiently by solving a convex optimization problem:
\begin{equation*}
\begin{aligned}
& \underset{\kappa,X}{\text{minimize}}
& & \kappa \\
& \text{subject to}
& & X = 2\kappa C^TR^T(p)R(p)C + C^T\Omega C - \lambda I, 
\quad X \succeq 0.
\end{aligned}
\end{equation*}
Here $I\in \R^{(nd-d(d+1)/2) \times (nd-d(d+1)/2)}$ is the identity matrix.

To summarize the ingredients necessary to state our theorems: we start with $(p,\mathcal E)$, and choose: (i) a subspace $\mathcal C$ complementary to the trivials (ii) an almost flex space $\mathcal V$, and (iii) an almost self-stress space $\mathcal W$. Then we find $\omega$ such that $\Omega$ solves \eqref{OmegaV}, which in turn determines  $\lambda_0,\mu_0$.  We choose $\lambda\in(0,\lambda_0)$ and solve \eqref{kappa} for $\kappa$. 

\medskip 

Two combinations of these numbers that will prove useful are
\begin{equation}\label{Lmu}
L=\left(\frac{\lambda}{8 z \kappa}\right)^{1/2},\qquad
\bar\mu = 1-\frac{\mu_0}{\lambda}\,.
\end{equation}
Here $z$ is the maximum adjacency of the graph, i.e. the maximum number of edges coming out of any vertex. The quantity $L$ will turn out to have the dimensions of length, as we discuss in Section \ref{sec:energyfunction}, while $\bar\mu$ is dimensionless. 

Our theorems will be proved by constructing a function $H:\R^{nd}\to R$ that depends on the ingredients above and is physically motivated. 
Writing $q_{ij}^2 = |q_i-q_j|^2$ for the squared edge length between vertices $i,j$ of configuration $q$, the function is 
\begin{equation}
H(q) = \sum_{(i,j)\in\mathcal E} \half\kappa(q_{ij}^2-p_{ij}^2)^2 + \omega_{ij}q_{ij}^2\,.
\end{equation}
(See equation \eqref{H} and surrounding discussion for more details.)
We call $H$ an \emph{energy function}, because 
it is constructed by pretending the edges of the framework are spring-like with spring constant $\kappa$, and the edges are under tension $\omega$. 
The tensions are such that the associated stress matrix $\Omega$ is not positive definite on all of $\mathcal C$, so these tensions induce forces that cause the energy to decrease along some flexes $v\in \mathcal C$; in other words $\grad H(p) \neq 0$.
The  minimization problem \eqref{kappa} that gives $\kappa$ asks us to increase the stiffness of the springs just enough so that even if the energy does decrease in some directions initially, the springs stretch enough so that after a short distance the energy rapidly increases and becomes positive again; in other words $\grad\grad H(p)$ is sufficiently positive definite.

\subsection{Theorems}

We are now ready to state our main theorems. They are proved in Section \ref{sec:theoremproofs}. 

\begin{mainthm}\label{thm:framework}
Given a framework $(p,\mathcal E)$ and $\mathcal C, \mathcal V, \omega, \lambda, L, \bar\mu$ as in the setup above. 
Define a radius
\begin{equation}\label{eta1}
\eta_1=\frac{4|\omega^TR(p)|}{\lambda}
\end{equation}
and define
\begin{equation}\label{D}
D = \frac{\eta_1}{L}\left(\bar\mu^{1/2} + \frac{\eta_1}{L} \right)\,.
\end{equation}
Suppose 
\begin{equation}\label{Framework1}
 D< \frac{1}{2}\,.
\end{equation}
Then any continuous path $q(t)\in\mathcal C^p$ with $q(0)=p$ which preserves the edge lengths of $p$ remains within a distance $\eta_1$ of $p$: $|q(t)-p|\leq\eta_1$. 
\end{mainthm}

When Theorem \ref{thm:framework} holds, $(p, \mathcal{E})$ may be flexible, but as it deforms it has to stay in a small ball of radius $\eta_1$. We say that  a $p$ satisfying the conditions of Theorem \ref{thm:framework} is \emph{almost-rigid}. 

\begin{remark}
In the case that $\omega$ is actually a
self-stress, then 
$\omega^TR(p)=0$, 
$\eta_1=0$, and this ball shrinks down to a single point. The stress solves \eqref{pss} so the framework is prestress stable.
In this sense, our notion of almost-rigidity is a generalization
of prestress-stability. 
\end{remark}

In this theorem, $\eta_1$ shows that $\omega$ should be
an almost-stress in order for $\eta_1$ to be small.
Additionally, if we do not include almost-flexes in
our chosen space
$\mathcal V$, then $\kappa$, and thus $D$ will be large.

The proof of this theorem
will be based on studying the landscape of a certain 
fictitious energy function. A critical point of this energy
that passes the second-derivative test will be prestress-stable.
A point that is not critical, but has sufficiently well
behaved second and third derivatives, will still behave
almost like an energy minimum.

\medskip

Our second theorem studies configurations that have the same edge lengths as $p$. 
For example, if $(p,\mathcal E)$ is the framework in Figure \ref{fig:rigidity} (A), and $q$ is the same configuration but where vertex 4 is on top of vertex 2, then $q$ has the same edge lengths as $p$. Clearly, $q$ is not achievable as a continuous, edge-length preserving deformation of $p$.

\begin{mainthm}\label{thm:outerbound}
Given the setup in Theorem \ref{thm:framework} and suppose that \eqref{Framework1} holds. 
 Define
\begin{equation}\label{eta2}
\eta_2 = \begin{cases}
L\,\frac{(\bar\mu+4)^{1/2}-\bar\mu^{1/2}}{2} - \eta_1 & \text{if }  D < \frac{1}{4} + \frac{1}{8}\left( \bar\mu^{1/2}(\bar\mu+4)^{1/2} - \bar\mu\right)\\
L\,\frac{(\bar\mu+2)^{1/2}-\bar\mu^{1/2}}{2} & \text{otherwise.} 
\end{cases}
\end{equation}
Then $\eta_2>\eta_1$, and any configuration $q\in\mathcal C^p$ with the same edge lengths as $p$ such that $|q-p|>\eta_1$, must additionally satisfy $|q-p|\geq\eta_2$. 
\end{mainthm}

The upper bound for $D$ in the first case of \eqref{eta2} is an increasing function of $\bar\mu$, which equals $1/4$ for $\bar\mu=0$, and asymptotes to $1/2$ as $\bar\mu\to\infty$. 

When Theorem \ref{thm:outerbound} holds, any configuration $q$ with the same edge lengths as $p$ that is not within $\eta_1$ of $p$, must be sufficiently far away from it, at least a distance of $\eta_2$ away.
Therefore, there is an annulus around $p$ with radii $(\eta_1,\eta_2)$ where all configurations have edge lengths distinct from those of $p$. 

\begin{remark}
It is tempting to ask whether $\eta_2=\infty$, which would imply global rigidity if $\eta_1=0$ or perhaps ``almost global rigidity'' if $\eta_1>0$. Unfortunately, we doubt this will happen, and if it does it wouldn't be useful. We can only have $\eta_2=\infty$ when $\kappa=0$, which is only possible when $\Omega(\omega)$ is positive definite everywhere on $\mathcal C$. If $\omega$ is a self-stress then $\Omega(\omega)$ is never positive definite since it gives a zero inner product along scaling transformations, which are not ruled out by $\mathcal C$. If $\omega$ is only an almost-stress, then $\eta_1$ should be large, because in this case $H$ is quadratic and has a global minimum at $q=0$. Therefore, $\eta_1$ must be at least $2|p|$ to ensure the $\eta_1$-ball contains the level set $H(q)=H(p)$, and this is likely too large to be useful. Another way to see this, is that the reflection of $p$ has the same edge lengths but is not a continuous edge-preserving deformation of it. Therefore if $\eta_2=\infty$, then the $\eta_1$-ball must be large enough to contain this reflection.
\end{remark}

In what follows, we write $q_{ij} := q_i-q_j\in\R^{d}$ for the vector between vertices $(i,j)$, and write $q_{ij}^2 := q_{ij}\cdot q_{ij}$ for the squared length of that vector.

\begin{mainthm}\label{thm:emin}
Given the setup in Theorem \ref{thm:framework}, and suppose that 
\begin{equation}\label{Framework3}
    D \quad < \quad
    \frac{1}{9}\left(\frac{8}{3} + \bar \mu^{1/2} \Big(\bar\mu+\frac{8}{3}\Big)^{1/2} - \bar \mu \right)\,.
\end{equation}
Define
\begin{equation}\label{eta3}
\eta_3 = \frac{L}{2}\Big(\Big(\bar\mu+\frac{8}{3}\Big)^{1/2}-\bar\mu^{1/2}\Big)\,.
\end{equation}
Then $\eta_3 > (3/2)\eta_1$. 
Suppose $q\in \mathcal C^p$, let $|q-p| = \eta$, and suppose $\frac{3}{2}\eta_1< \eta \leq \eta_3$. 
Let $e(q) = (q_{ij}^2)_{(i,j)\in\mathcal E}$ 
be the vector of squared edge lengths at configuration $q$. 
Then $|e(q)-e(p)| \geq e_{\rm min}(\eta)$, where  $|\cdot|$ is the $l_2$-norm, and 
\begin{equation}\label{emin}
 e_{\rm min}(\eta) = 
 \begin{cases}
    \sqrt{\big(\frac{|\omega|}{\kappa}\big)^2 + \frac{2}{3}\frac{\lambda}{\kappa}  \eta^2\big( 1-\frac{3}{2}\frac{\eta_1}{\eta}\big)}  \;-\; \frac{|\omega|}{\kappa} & \text{ if } \kappa > 0\\
   \frac{1}{3}\frac{\lambda}{|\omega|}\eta^2\big( 1-\frac{3}{2}\frac{\eta_1}{\eta}\big) & \text{ if }  \kappa = 0\,.
    \end{cases}
\end{equation}
Furthermore, if $p$ is deformed continuously to some configuration $q\in \mathcal C^p$ with the same edge lengths as $p$ but such that $|q-p| > (3/2)\eta_1$, so that $|q-p| > \eta_2$ (by Theorem \ref{thm:outerbound}), 
then there is some point $q'$ on the deformation path at which $|e(q')-e(p)| \geq e_{\rm min}^*$, given by 
\begin{equation}\label{estar}
    e_{\rm min}^* = e_{\rm min}(\eta_3)\,.
\end{equation}
\end{mainthm}


Condition \eqref{Framework3} 
is similar to \eqref{Framework1}, but with a different right-hand side. The right-hand side of \eqref{Framework3} is an increasing function of $\bar\mu$, equal to $8/27$ when $\bar\mu=0$ and asymptoting to 
$12/27$ as $\bar \mu\to \infty$, ever so slightly smaller than the constant $1/2$ on the right-hand side of \eqref{Framework1}. 

The first part of Theorem \ref{thm:emin} gives a specific lower bound on the amount the edges of $p$ have to deform, in order to adopt a specific configuration within a distance   $\eta \in ((3/2)\eta_1,\eta_3)$ from $p$. 
It holds under a slightly stronger condition than Theorems \ref{thm:framework} and \ref{thm:outerbound}, since the upper bound for $D$ in \eqref{Framework3} is never quite 1/2. Usually $D$ will be much smaller than its upper bound, and then comparing \eqref{eta2}, \eqref{eta3}, we see that $\eta_2>\eta_3$.  There could be small regions of parameter space where $\eta_2 < \eta_3$, since the theorems were proved using slightly different assumptions, however we have not come across these regions in our examples. 

If instead we want to know how far $q$ is from $p$, given that $e(q)$ is close to $e(p)$, an upper bound on this distance is given by the following corollary.

\begin{corollary}\label{cor:emin}
Suppose the conditions of Theorem \ref{thm:emin} hold, and suppose $\epsilon=|e(q)-e(p)| < e_{\rm min}^*$, where $e_{\rm min}^*$ is given in  \eqref{estar}. Define
\begin{equation}\label{etamax}
    \eta_{\rm max}(\epsilon) = \frac{3}{4}\eta_1 + \sqrt{\Big(\frac{3}{4}\eta_1\Big)^2 + 3\frac{|\omega|}{\lambda}\epsilon + \frac{3}{2}\frac{\kappa}{\lambda}\epsilon^2} \,.
\end{equation}
Then either $|q-p| < \eta_{\rm max}(\epsilon)$, or $|q-p| > \eta_3$, where $\eta_3$ is given in \eqref{eta3}. 
\end{corollary}

\begin{proof}
Since $\epsilon=|e(q)-e(p)| < e_{\rm min}(\eta_3)$, and $e_{\rm min}(\eta)$ is a continuous increasing function of $\eta$ such that $e_{\rm min}(\eta)=0$ at $\eta=(3/2)\eta_1$, there is some value of $\eta>(3/2)\eta_1$ such that $e_{\rm min}(\eta) = \epsilon$; one can verify that $\eta_{\rm max}(\epsilon)$ is this value, for both $\kappa>0$ and $\kappa=0$; furthermore $\eta_{\rm max}(\epsilon) > (3/2)\eta_1$. 
Now suppose that $|q-p| \in (\eta_{\rm max}(\epsilon),\eta_3]$. Then by Theorem \ref{thm:emin}, $|e(q)-e(p)| \geq e_{\rm min}(\epsilon)$, a contradiction. 
\end{proof}
This corollary describes an even more robust form of
almost-rigidity.
In this situation, we know that even if the edge lengths are allowed to
change a little, the configuration can still not change very much.

\medskip

If a framework is almost rigid, it is natural to wonder whether there is a nearby framework that is rigid. 
Our final theorem gives conditions when this is the case. 


\begin{mainthm}\label{thm:pss}
Given the setup in Theorem \ref{thm:framework}, and let
\begin{equation}\label{Dpss}
    D_{\rm pss} = \frac{\eta_1}{L} \left( \bar\mu^{1/2} + \frac{3}{2}\frac{\eta_1}{L} +   \frac{1}{2}\frac{|p|}{L}\right) \,.
\end{equation}
Suppose that
\begin{equation}\label{PssBound}
D_{\rm pss}  < \frac{1}{2}\,.
\end{equation}
Then there exists a framework $(p_{\rm pss},\mathcal E)\in \mathcal C^p$ with the same edge lengths as $p$, such that $|p_{\rm pss}-p|<\eta_1$ and $p_{\rm pss}$ is prestress stable.
\end{mainthm}

Condition \eqref{PssBound} is similar to condition \eqref{Framework1} required to show the existence of $\eta_1$, but it is stronger. 
The main difference is the third term in $D_{\rm pss}$, proportional to $|p|$. One can minimize this term by translating $p$ so its center of mass lies at the origin, however, this term could still pose a problem for large frameworks. It will require $\eta_1$ to be small, typically much smaller than needed in Theorem \ref{thm:framework}, in order to show the existence of a nearby prestress stable framework. 

We end with two corollaries that apply the theorems to specific cases. Our first corollary handles first-order rigid clusters. 

\begin{corollary}\label{cor:firstorder}
Suppose that $(p,\mathcal E)$ is first-order rigid. Let $\sigma_0$ be the smallest nonzero singular value of 
$R(p)C$, where $C$ is a matrix whose columns form an orthonormal basis of $\mathcal C$. 
Then any other configuration $q\in\mathcal C^p$ with the same edge lengths as $p$ must remain at least a distance of 
\begin{equation}\label{eta2firstorder}
\eta_2 = \frac{\sigma_0}{2z^{1/2}}\frac{\sqrt{5}-1}{2} \;\;\approx\;\; 0.618\frac{\sigma_0}{2z^{1/2}} 
\end{equation}
from $p$. Furthermore, to deform the framework continuously from $p$ to any other configuration $q\in\mathcal C^p$ with the same edge 
lengths as $p$, there must be some point $q$ on the deformation path at which the vector of squared edge lengths, $e(q)$, has changed by at least
\begin{equation}\label{lengthsfirstorder}
|e(q)-e(p)|\;\;\geq\;\;  e_{\rm min}^* = \sigma_0^2 \frac{((11/3)^{1/2} - 1)}{2\sqrt{3}z^{1/2}}\;\;\approx\;\; 0.528\frac{\sigma_0^2}{2z^{1/2}} \,.
\end{equation}
\end{corollary}


It is possible that $R(p)C$ has zero singular values but no null space, for example if the framework is overconstrained. This is why we require $\sigma_0$ to be the smallest nonzero singular value, and not just the smallest singular value. 

\begin{proof}
Choose 
\begin{equation}\label{firtorderconstants}
    \omega= 0\,,\qquad 
    \kappa=1\,,\qquad \lambda = 2\sigma_0^2\,.
\end{equation}
 Then \eqref{kappa} holds with these choices. 
 Since $\eta_1 = 0$,  \eqref{Framework1} holds. 
Therefore we may apply Theorem \ref{thm:outerbound}, calculating $\eta_2$ using the first case in \eqref{eta2}. 
We have, by \eqref{Lmu} and because $\mu_0 = 0$,
\begin{equation}
    L = \frac{\sigma_0}{2z^{1/2}}\,,\qquad \bar \mu = 1\,,
\end{equation}
so we obtain \eqref{eta2firstorder} for the outer radius. 

For second statement, notice that \eqref{Framework3} holds, and therefore we may apply Theorem \ref{thm:emin} to say that if  $|q-p|=\eta_3$, where 
\[
\eta_3 = \frac{\sigma_0}{2z^{1/2}}\frac{((11/3)^{1/2} - 1)}{2}\;\;\approx\;\; 0.457\frac{\sigma_0}{2z^{1/2}}\,,
\]
then 
\[
|e(q)-e(p)| \geq e_{\rm min}(\eta_3) = \frac{2}{\sqrt{3}}\sigma_0\eta_3   \,.
\]
Since  $\eta_3 < \eta_2$, any continuous path from $p$ to some configuration $q\in\mathcal C$ such that $|q-p| \geq \eta_2$, must have some point $q'$ on the path such that $|q'-p| = \eta_3$, at which point $|e(q')-e(p)| \geq \frac{2}{\sqrt{3}}\sigma_0\eta_3$, the lower bound given in \eqref{lengthsfirstorder}. 
\end{proof}

Our next corollary gives the asymptotic behaviour of $\eta_{\rm max}(\epsilon)$ for small $\epsilon$. The asymptotic relation $f(\epsilon) \sim g(\epsilon) + O(\epsilon^m)$ as $\epsilon \to 0$ means $\limsup_{\epsilon\to 0}\frac{|f(\epsilon) - g(\epsilon)|}{\epsilon^m} < \infty$.

\begin{corollary}
Suppose that $p$ satisfies the conditions of Theorem \ref{thm:emin}, and that $q\in\mathcal C^p$ is such that
$\epsilon = |e(q)-e(p)| < e_{\rm min}^*$, where $e_{\rm min}^*$ is defined in \eqref{estar}. 
If $|q-p| < \eta_3$, where $\eta_3$ is defined in \eqref{eta3}, then $|q-p| < \eta_{\rm max}(\epsilon)$, where $\eta_{\rm max}(\epsilon)$ has the following behaviour as $\epsilon \to 0$: 
\begin{itemize}
    \item If $p$ is first-order rigid, then  
    \[
    \eta_{\rm max}(\epsilon) = \frac{\sqrt{3}}{2\sigma_0}\epsilon\,,
    \]
    where $\sigma_0$ is the smallest nonzero singular value of $R(p)C$;
    \item If $p$ is prestress stable but not first-order rigid, then 
    \[
    \eta_{\rm max}(\epsilon) =\sqrt{\frac{3|\omega|}{\lambda}}\epsilon^{1/2}\left(1 + \frac{\kappa}{2|\omega|}\epsilon\right)^{1/2} \sim \sqrt{\frac{3|\omega|}{\lambda}}\epsilon^{1/2}  + O(\epsilon^{3/2}) \qquad \text{as }\quad  \epsilon \to 0\,.
    \]
\end{itemize}
\end{corollary}



This corollary is useful when solving for a rigid framework 
$(p,\mathcal E)$ on a computer, given the desired edge lengths. 
One does this by numerically solving \eqref{fij} for a given a set of edge lengths $\{d_{ij}\}_{(i,j)\in\mathcal E}$. 
Since \eqref{fij} is nonlinear, we cannot typically find the exact solution $p_{\rm true}$, assuming this exists, but rather we find an approximate solution $p$ that satisfies \eqref{approxedgeeqn}, where  $\delta$ is a desired numerical tolerance. 
The edge lengths can be made very close to the desired edge lengths by controlling $\delta$, but how close can we expect $p$ to be to the true solution $p_{\rm true}$, as a function of $\delta$? 

The difference in squared edge lengths between the approximate and true solutions satisfies $|e(p)-e(p_{\rm true})| < m\delta \sim  O(\delta)$, where $m$ is the number of edges. Therefore, if  $p_{\rm true}$ 
is first-order rigid, and we believe we are close enough to it, then $|p-p_{\rm true}| \sim O(\delta)$. However, if $p_{\rm true}$
is only prestress stable, but not first-order rigid, then we expect 
$|p-p_{\rm true}| \sim O(\delta^{1/2})$ in general. If $\delta \ll 1$, the error is much bigger for a framework that is not first-order rigid. For example, if $\delta = 10^{-8}$, then we may expect about 8 digits of precision for the configuration of a first-order rigid framework, but only 4 digits of precision for that of a prestress stable framework. 

\begin{proof}
Straightforward application of Corollary \ref{cor:emin} gives that  $|q-p| < \eta_{\rm max}(\epsilon)$. Consider the specific form of $\eta_{\rm max}(\epsilon)$ for each of the cases. 
If $p$ is first-order rigid, then $\omega=0$, $\kappa=1$, $\lambda=2\sigma_0^2$, $\eta_1=0$, as in the proof of Corollary \ref{cor:firstorder}. Directly evaluating $\eta_{\rm max}(\epsilon)$ leads to the given result. 

If $p$ is prestress stable but not first-order rigid, then $\eta_1 = 0$ but $\omega \neq 0$. Again we may directly evaluate $\eta_{\rm max}(\epsilon)$ under this condition. 
%
\end{proof}

\section{Examples}\label{sec:examples}

Now we give some examples to show how the theorems in Section \ref{sec:mainresults} may be applied. 

In all of these examples we chose $\mathcal C = \mathcal T^\perp$, unless otherwise stated.  The choice of $\mathcal V, \mathcal W$ depends on the example, but we generally construct them using singular vector subspaces as in \eqref{VWsig}. With these choices, we find $\omega$ using the algorithm given in \eqref{pssalg}. 
This sets $\lambda_0, \mu_0$; we then must choose $\lambda \in (0,\lambda_0)$. While we could optimize over $\lambda$ to choose the value that maximizes the lengthscale $L$ for each example (see \eqref{Lmu}), we instead fix its value, and address the sensitivity of our results to $\lambda$ in Section \ref{sec:LvsLam}. 
Finally, we solve \eqref{kappa} for $\kappa$. 

\subsection{Revisiting the introductory examples}

We now apply our theory to the examples from Section \ref{sec:introexamples}. 

\subsubsection{Example 2}
Consider the underconstrained cluster of $N=10$ unit spheres from Section \ref{sec:example2}. We found that $V^T\Omega V \geq 0.33=\lambda_0$, i.e. the almost-stress was positive definite on the almost-flex space. To apply our theory we chose $\lambda = \lambda_0 / 2$ and solved \eqref{kappa} to find $\kappa = 3.26$, giving a lengthscale of $L = 0.0265$. The minimum eigenvalue of the stress matrix on the whole space is $\mu_0 = \lambda_{\rm min}(\Omega) = -2.25$, so $\bar \mu = 14.7$. 

From these numbers we can compute the various radii. The inner radius is $\eta_1 = 1.6\times 10^{-7}$. The left-hand side of \eqref{Framework1} is $D=2.3\times 10^{-7} < 0.5$, so Theorem \ref{thm:framework} applies.  Therefore, as the numerically-computed framework $p$ flexes continuously along some path $p(t)\in\mathcal C^p$, it must remain within a distance $\eta_1$ from $p$: 
\[
|p(t) - p| \leq 1.6\times 10^{-7}\,.
\]
This is an extremely small radius of containment, much smaller than either of the lengthscales $|p| = 2.97$, or the edges, which are all unit length. Even though $p$ is not locally rigid, it is still rigid in a broader sense, in that it is confined to a very small region in configuration space. 

Theorem \ref{thm:outerbound} also applies, and gives an outer radius $\eta_2 = 0.0065$. This means that any other framework $q\in\mathcal C$ with the same edge lengths as $p$, that is not within a distance of $\eta_1$ of $p$, must be at least this distance from $p$, 
\[
|q-p| \geq 0.0065\,.
\]
This is significantly further than  the distance $\eta_1$ that $p$ can naturally flex. 

We found $\eta_3 = 0.0044159$, and $e_{\rm min}^* = 1.07\times 10^{-6}$. This means that to continuously deform the framework from $p$ to a configuration that is a distance of at least $\eta_3$ from $p$, such as a configuration with the same edge lengths, there must be some point $q$ on the path at which the squared edge lengths deform by at least $\Big(\sum_{(i,j)\in \mathcal E} (q_{ij}^2 - p_{ij}^2)^2\Big)^{1/2} \geq 1.07\times 10^{-6}$. Writing $q_{ij}^2 - p_{ij}^2 = (|q_{ij}| - |p_{ij}|)(|q_{ij}| + |p_{ij}|)\approx 2(|q_{ij}| - |p_{ij}|)\approx 2\Delta l$, where $\Delta l$ is a typical change in edge length, shows that each edge must change on average by at least 
\[
\Delta l \approx \frac{e_{\rm min}^*}{\sqrt{4\cdot23}} \approx 10^{-7}\,.
\]
This barrier is not particularly large compared to the actual edge lengths; our theory gives a minimum bound for the change in edge lengths, but does not say how close this bound is to the true barrier. 

What do these calculations tell us about the actual framework we are interested in? We calculated $D_{\rm pss}=0.00036 < 0.5$ in \eqref{PssBound}, so Theorem \ref{thm:pss} applies and says there is a prestress stable framework $p_{\rm pss}$ that is very close to $p$: $|p_{\rm pss} - p| \leq 1.6\times 10^{-7}$. Of course, we don't know that $p_{\rm pss}$ has the desired edge lengths, although we do know they will all be very close to unit length. The existence of a nearby prestress stable framework follows without making any assumptions on whether the small singular values of $R(p)$ are perturbations of a zero singular value or not.

\medskip

What happens when we solve for the cluster with a less stringent edge length tolerance? 
We perturbed our starting cluster by a random amount, and then solved \eqref{approxedgeeqn} using different values of $\delta$. For each perturbation and each $\delta$, we computed $\eta_1,\eta_2,e_{\rm min}^*$. Figure \ref{fig:hypoperturbations} shows the average of each quantity at each fixed $\delta$ over 20 random perturbations, as well as the estimated error bars. The inner radius is a straight line with slope $1/2$ on a log-log plot; a best-fit line to the data gives $\eta_1 = e^{\delta / 2 + 2}$. The outer radius $\eta_2$ and edge length barrier $e_{\rm min}^*$ are nearly constant with $\delta$, until $\delta \gtrapprox 10^{-8}$, after which they decrease rapidly (data points where they drop below zero are not shown.) 

We also calculated  $D$, 
$D_{\rm pss}$ in \eqref{Framework1}, 
\eqref{PssBound}, also shown in Figure \ref{fig:hypoperturbations}. Since $D$ crosses $0.5$ at about $\delta=10^{-3}$, the various radii are only meaningful up to this point. 
The constant $D_{\rm pss}$ required to show there is a nearby prestress stable framework crosses two orders of magnitude earlier, at about $\delta=10^{-9}$. 
These numbers give a sense of how accurately one must solve \eqref{approxedgeeqn}, in order for the theorems to apply. 

\subsubsection{Example 1} 

We also applied the theory to the example in Section \ref{sec:example1},  Figure \ref{fig:rigidity} (B), and obtained radii of a similar order of magnitude (See Table \ref{tbl:intro}.) 
When we perturbed the framework, by moving vertices 5,6 vertically by $\pm h$ respectively, we obtained a small enough $D$ for the theorems to hold when the vertical perturbation was around $h=5\times 10^{-4}$ or less. This is a limitation of our theory, since geometrical arguments would show that the framework is constrained to a small ball under much larger perturbations.  

For this example, we can calculate the maximum distance the framework can flex analytically, assuming that vertex 1 is pinned to the origin, and the horizontal coordinate of vertex 4 is pinned. (This effectively pins vertices 1,2,3,4.) 
For $h=10^{-4}$, we calculated $\eta_1 = 1.4\times 10^{-3}$ for the pinned framework, and the maximum distance it can actually flex when pinned is $10^{-4}$. Note that $\eta_1$ is larger than this distance, as it must be, but not too much larger. 


\begin{figure}
\centering
\includegraphics[width=0.4\linewidth]{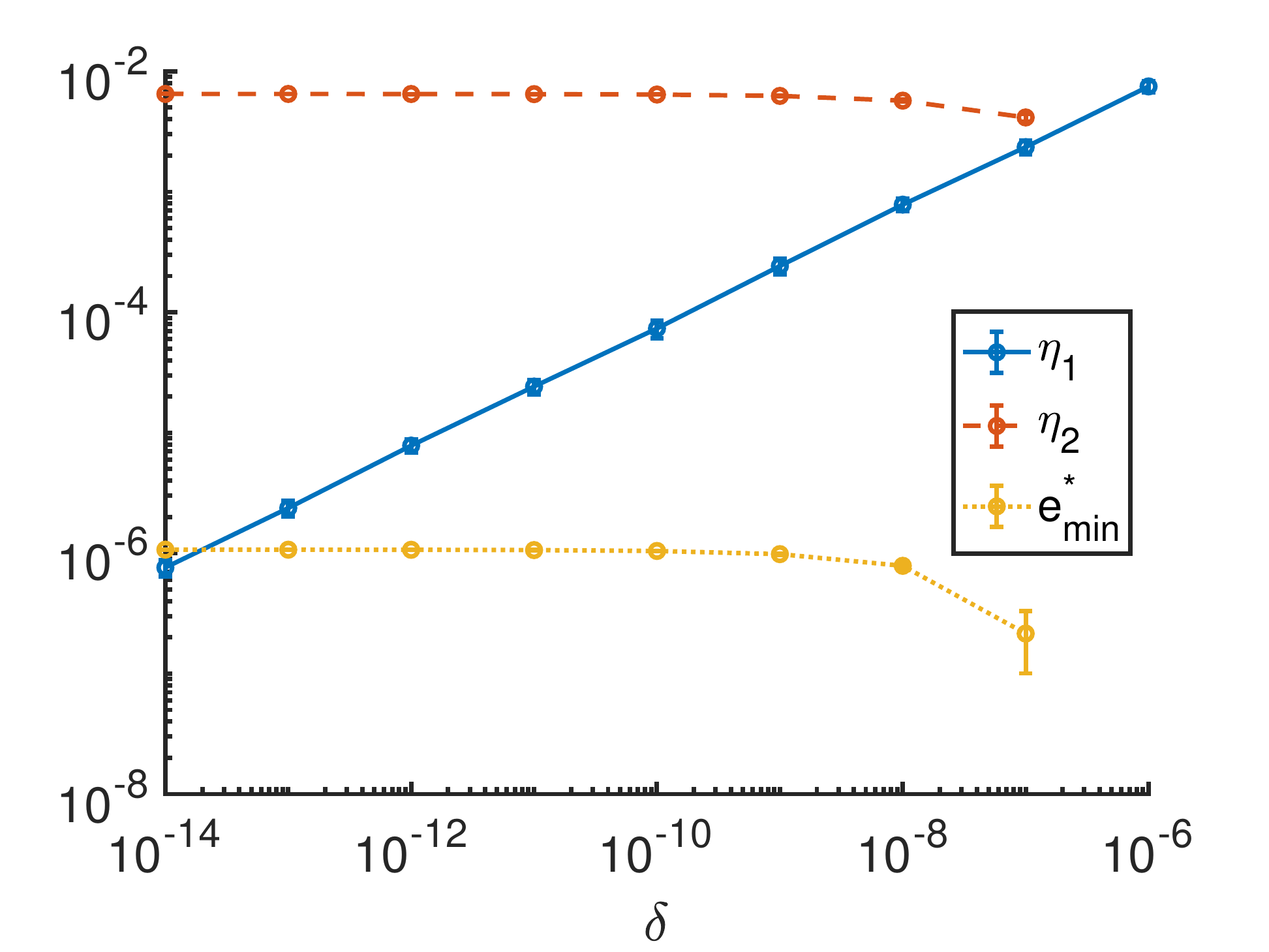}\quad
\includegraphics[width=0.4\linewidth]{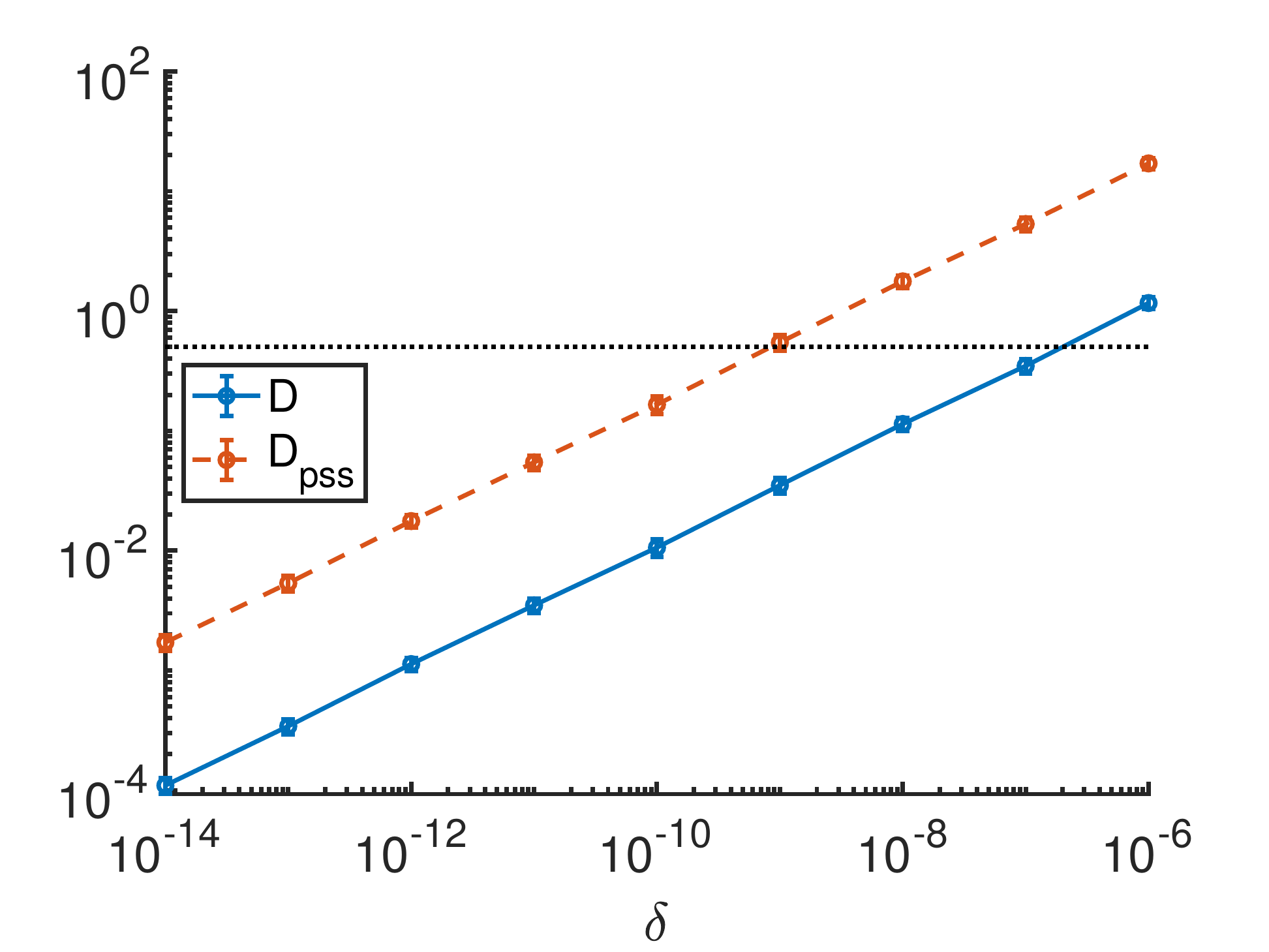}
\caption{
Characteristic lengthscales and theorem constants for the cluster of $N=10$ spheres shown in Figure \ref{fig:n10hypo} and discussed in Section \ref{sec:introexamples}. The cluster was randomly perturbed and then \eqref{approxedgeeqn} was solved with varying tolerance parameters $\delta$ for the squared edge lengths. 
The graphs show the average of each quantity  and estimated error bars (one standard deviation) over 20 independent perturbations at each value of $\delta$. 
}\label{fig:hypoperturbations}
\end{figure}

\begin{table}
    \centering
    \begin{tabular}{c p{7.5cm} p{4cm}}
    image   &  comments & theorem lengthscales \\\hline\\
\raisebox{2ex - \height}{  
\begin{tikzpicture}[thick,scale=1.25]
\newcommand\ys{0}
\node[vertex] (1b) at (0,0-\ys) {};
\node[vertex] (2b) at (1,0-\ys) {};
\node[vertex] (3b) at (0.5,1-\ys) {};
\node[vertex] (4b) at (0.5,0.5-\ys) {};
\node[vertex] (5b) at (1/3,0.0-\ys) {};
\node[vertex] (6b) at (2/3,-0.0-\ys) {};
\draw[ultra thick] (1b) -- (3b);
\draw[ultra thick] (2b) -- (3b);
\draw[ultra thick] (1b) -- (4b);
\draw[ultra thick] (2b) -- (4b);
\draw[ultra thick] (3b) -- (4b);
\draw[ultra thick] (1b) -- (5b);
\draw[ultra thick] (2b) -- (6b);
\draw[ultra thick] (5b) -- (6b);
\end{tikzpicture}
}
& 
(a) 6 vertices, 8 edges ($<$ isostatic)\newline
(see Figure \ref{fig:rigidity} (B))
& 
$\eta_1=0$\newline 
$\eta_2 = 0.014$,  $\eta_3 = 0.0093$ \newline
$e_{\rm min}^* = 5.7\times 10^{-6}$\newline
\\
\raisebox{2ex - \height}{\includegraphics[trim={4cm 0 4cm 7cm},clip,width=2.2cm]{cluster_n10_framework} }
& 
(b) 10 vertices, 23 bonds ($<$ isostatic), unit edges \newline
solved on computer to tolerance $\delta = 9\times 10^{-16}$ \newline
1 flex, 1 almost-flex, 1 almost-stress\newline
(see Figure \ref{fig:n10hypo})
& 
$\eta_1=1.6\times 10^{-7}$\newline
$\eta_2 = 0.0065$,  $\eta_3 = 0.0044$ \newline
$e_{\rm min}^* = 1.1\times 10^{-6}$\newline
\\\hline
    \end{tabular}
    \caption{Lengthscales from Theorems \ref{thm:framework}-\ref{thm:emin} calculated for our introductory examples (Section \ref{sec:introexamples}), using $\lambda = \lambda_0/2$. }
    \label{tbl:intro}
\end{table}

\subsection{A collection of examples} 

\begin{table}
    \centering
    \begin{tabular}{c p{7.5cm} p{4cm}}
    image   &  comments & theorem lengthscales \\\hline\\
    \raisebox{2ex - \height}{ 
\begin{tikzpicture}[thick,scale=1]
\node[vertex] (1) at (0,0) {};
\node[vertex] (2) at (1,0) {};
\node[vertex] (3) at (1,1) {};
\node[vertex] (4) at (0,1) {};
\draw (1) -- (2);
\draw (2) -- (3);
\draw (3) -- (4);
\draw (4) -- (1);
\draw (1) -- (3);
\end{tikzpicture}   }
& 
(c) square, first-order rigid, unit height \newline
distance to other embeddings =\newline \phantom{1.222} 1.22 (fold), 2 (reflect)
& 
$\eta_1=0$\newline
$\eta_2 = 0.16$, $\eta_3 = 0.12$ \newline
$e_{\rm min}^* = 0.12$\newline
\\
\raisebox{2ex - \height}{\includegraphics[trim={10cm 0 0 10cm},clip, width=1.5cm]{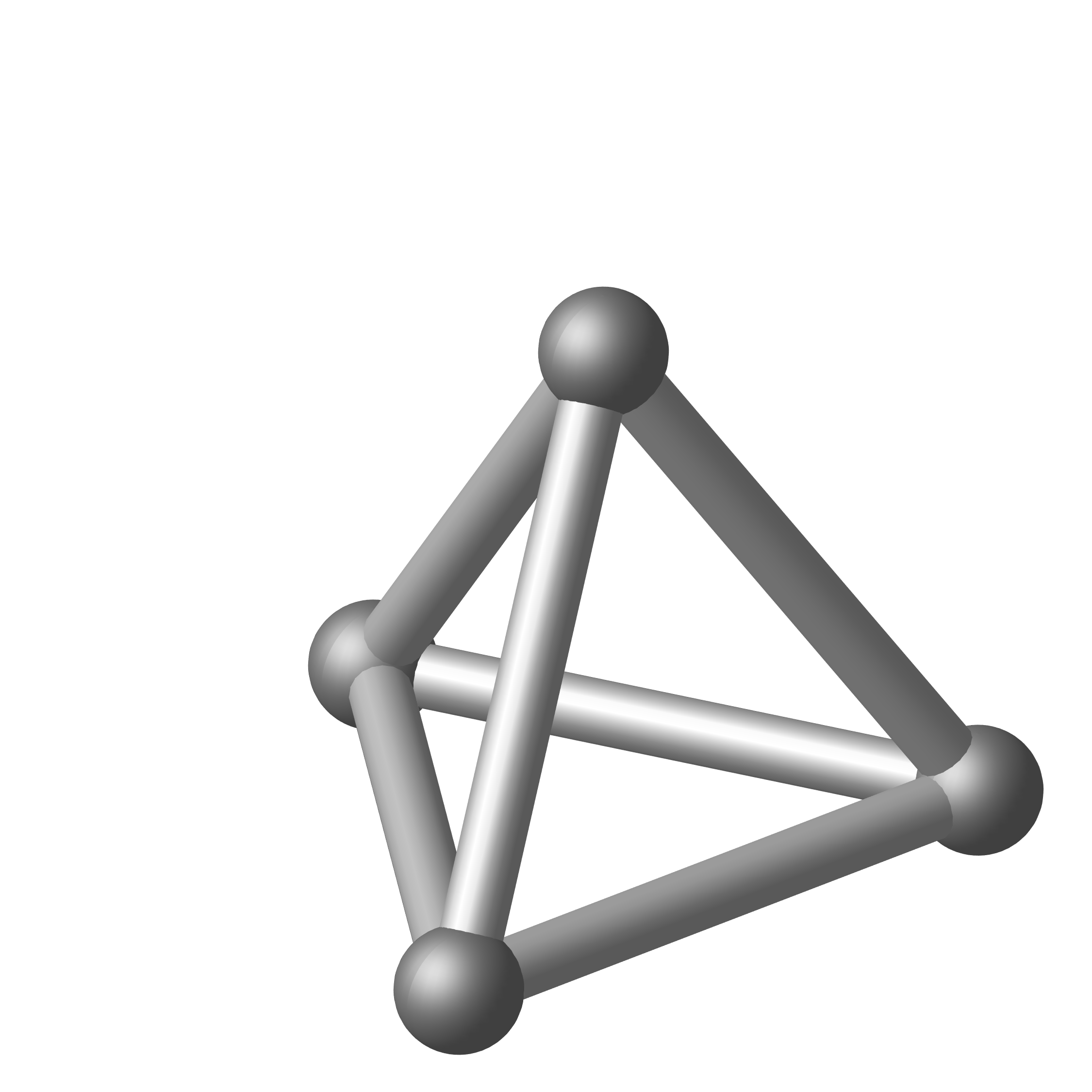} }
& 
(d) tetrahedron, unit edges \newline
first-order rigid \newline
distance to reflection = 1.06
& 
$\eta_1=0$\newline
$\eta_2 = 0.18$,  $\eta_3 = 0.13$\newline
$e_{\rm min}^* = 0.15$\newline
\\
\raisebox{2ex - \height}{\includegraphics[trim={4cm 4cm 4cm 4cm},clip, width=1.6cm]{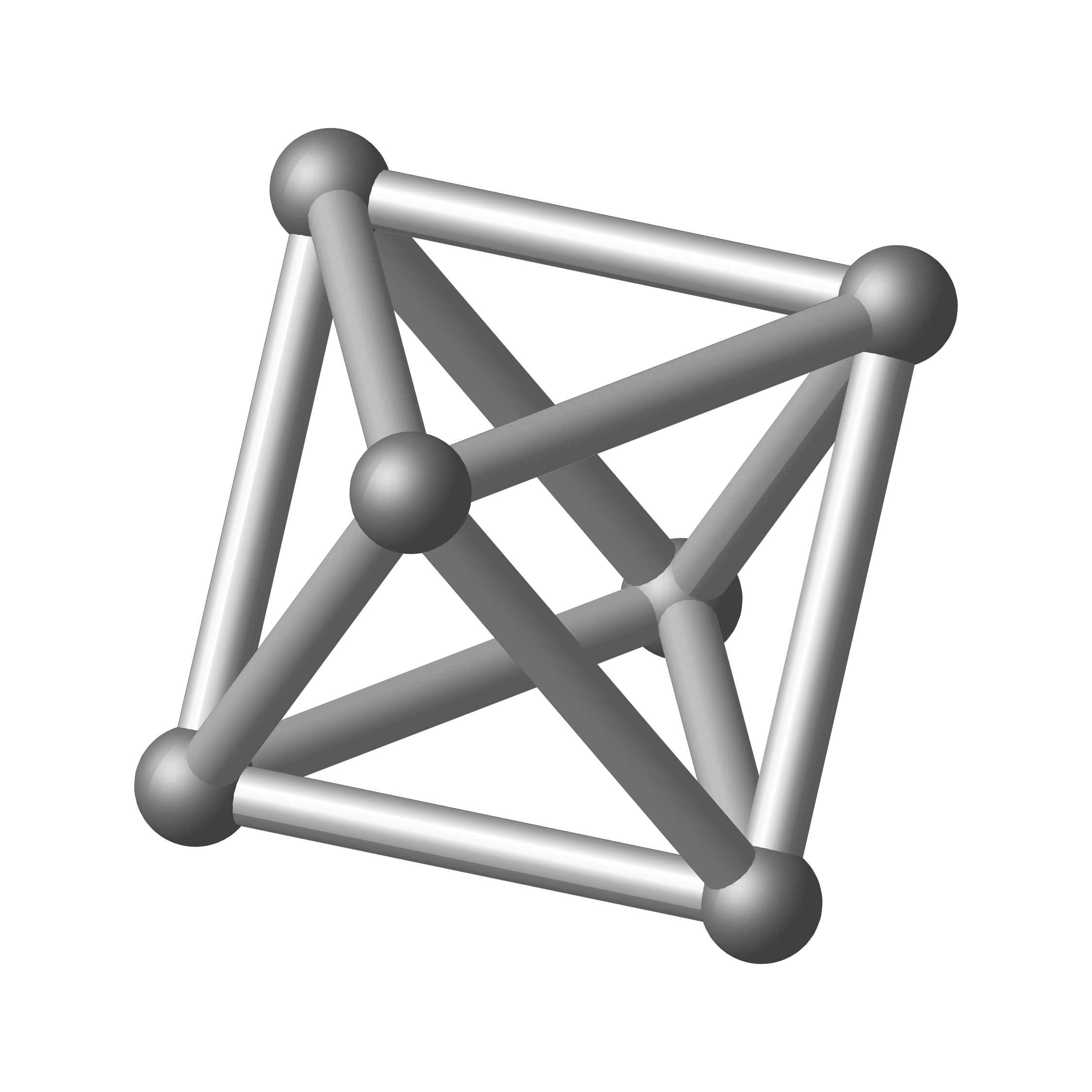}}
& 
(e) octahedron, unit edges \newline
first-order rigid 
& 
$\eta_1=0$\newline
$\eta_2 = 0.15$,  $\eta_3 = 0.11$\newline
$e_{\rm min}^* = 0.13$\newline
\\
\raisebox{2ex - \height}{\includegraphics[trim={1cm 9cm 7cm 6cm},clip, width=2.1cm]{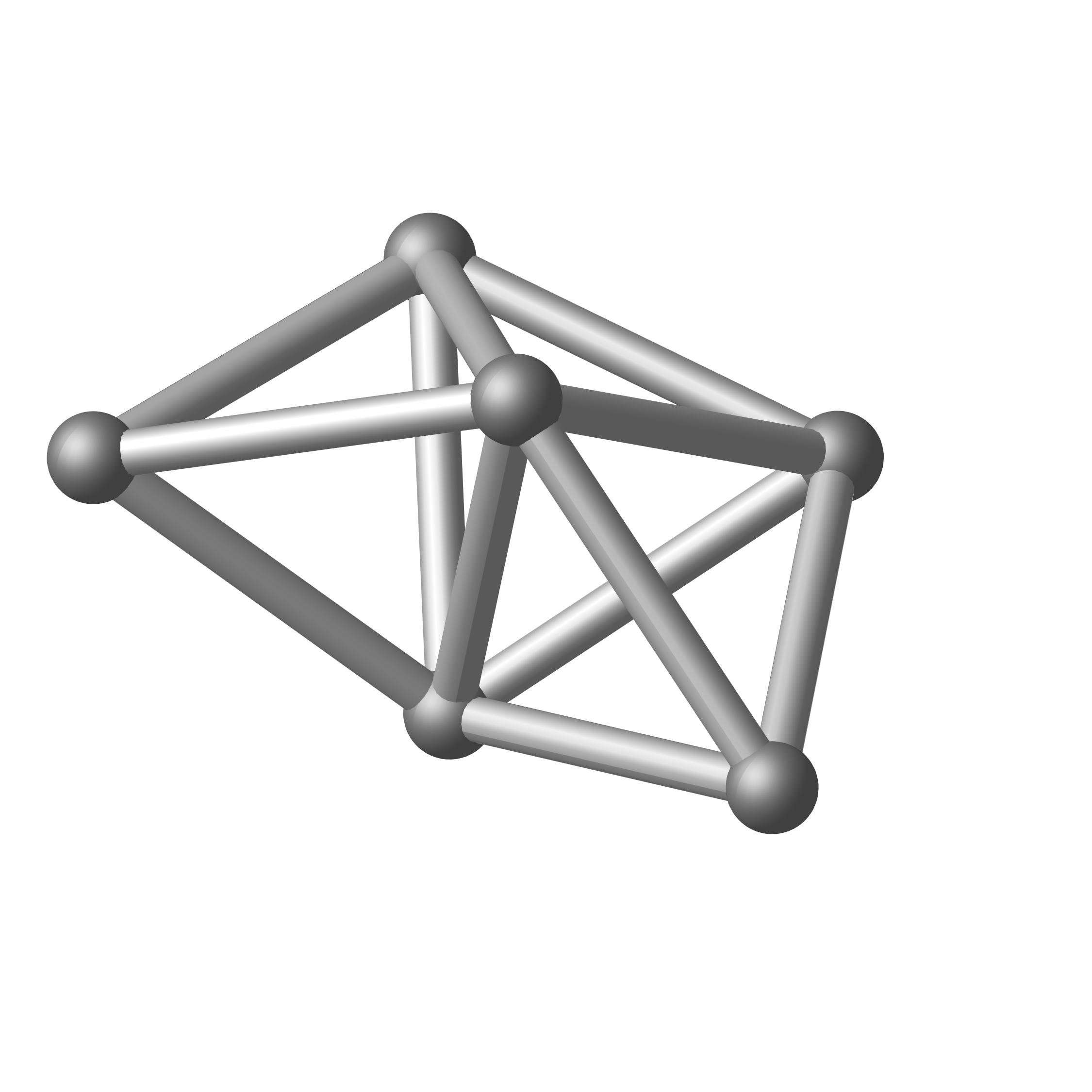}}
& 
(f) polytetrahedron, unit edges \newline
first-order rigid 
& 
$\eta_1=0$\newline
$\eta_2 = 0.078$,  $\eta_3 = 0.057$\newline
$e_{\rm min}^* = 0.037$\newline
\\
\raisebox{2ex - \height}{\includegraphics[trim={1cm 6cm 3cm 6.5cm},clip,width=2.2cm]{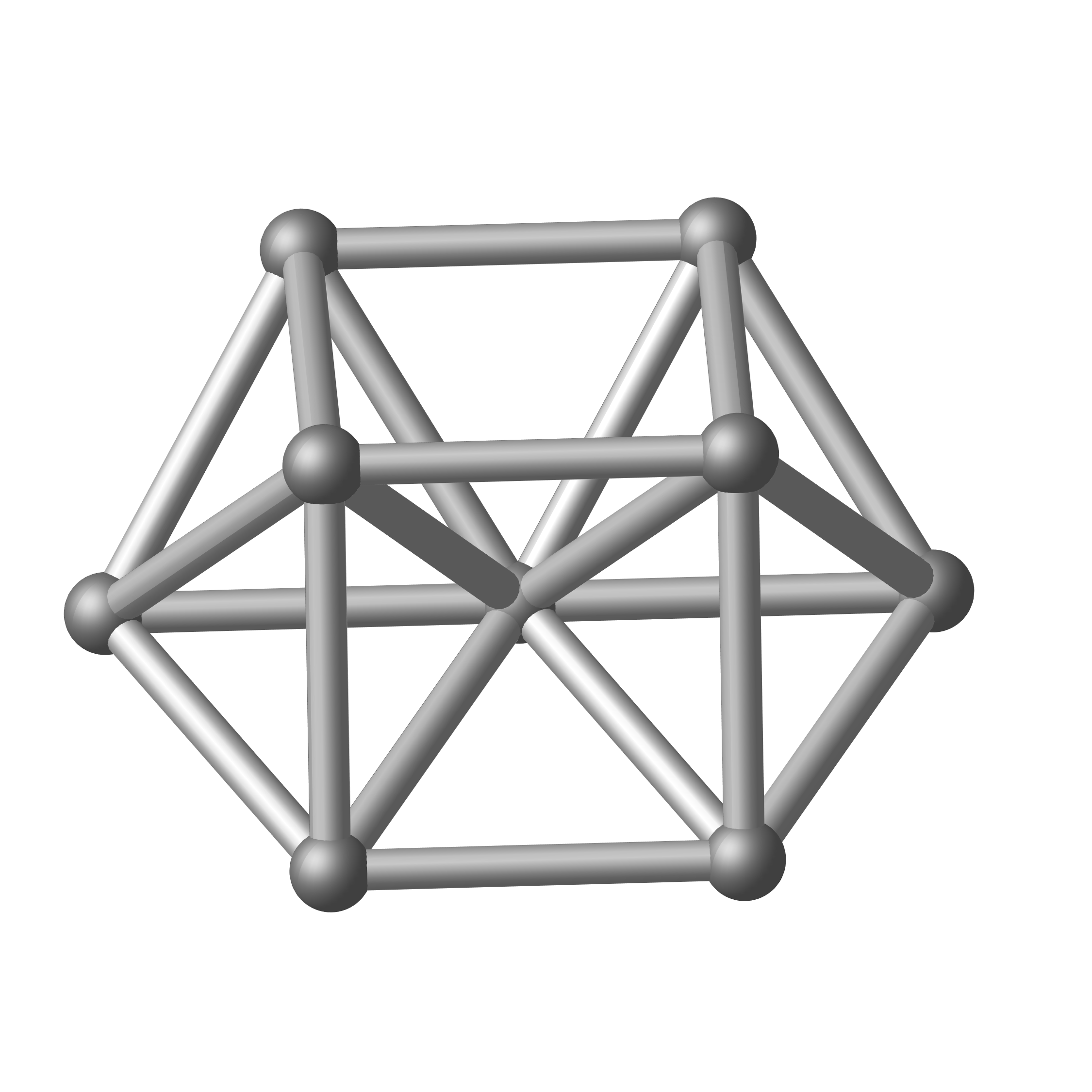}}
& 
(g) isostatic (9 vertices, 21 edges), unit edges\newline
solved numerically to tolerance $\delta = 9\times10^{-16}$\newline
1 almost-flex, 1 almost-stress 
& 
$\eta_1=1.7\times 10^{-7}$\newline
$\eta_2 = 0.0145$,  $\eta_3 = 0.0099$ \newline
$e_{\rm min}^* = 4.3\times 10^{-6}$\newline
\\
\raisebox{2ex - \height}{
\begin{tikzpicture}[thick,scale=0.5]
\node[vertex] (1) at (0,0) {};
\node[vertex] (2) at (0,2) {};
\node[vertex] (3) at (1,1) {};
\node[vertex] (4) at (3,1) {};
\node[vertex] (5) at (4,2) {};
\node[vertex] (6) at (4,0) {};
\draw (1) -- (2);
\draw (1) -- (3);
\draw (2) -- (3);
\draw (3) -- (4);
\draw (2) -- (5);
\draw (4) -- (5);
\draw (1) -- (6);
\draw (4) -- (6);
\draw (5) -- (6);
\end{tikzpicture}
}
& 
(h) isostatic (6 vertices, 9 edges)\newline 
prestress stable, not first-order rigid\newline
unit height
& 
$\eta_1=0$\newline
$\eta_2 = 0.13$, $\eta_3 = 0.096$ \newline
$e_{\rm min}^* = 8.9\times 10^{-4}$\newline
\\
\raisebox{2ex - \height}{\includegraphics[trim={3cm 8cm 7cm 7.5cm},clip,width=2.4cm]{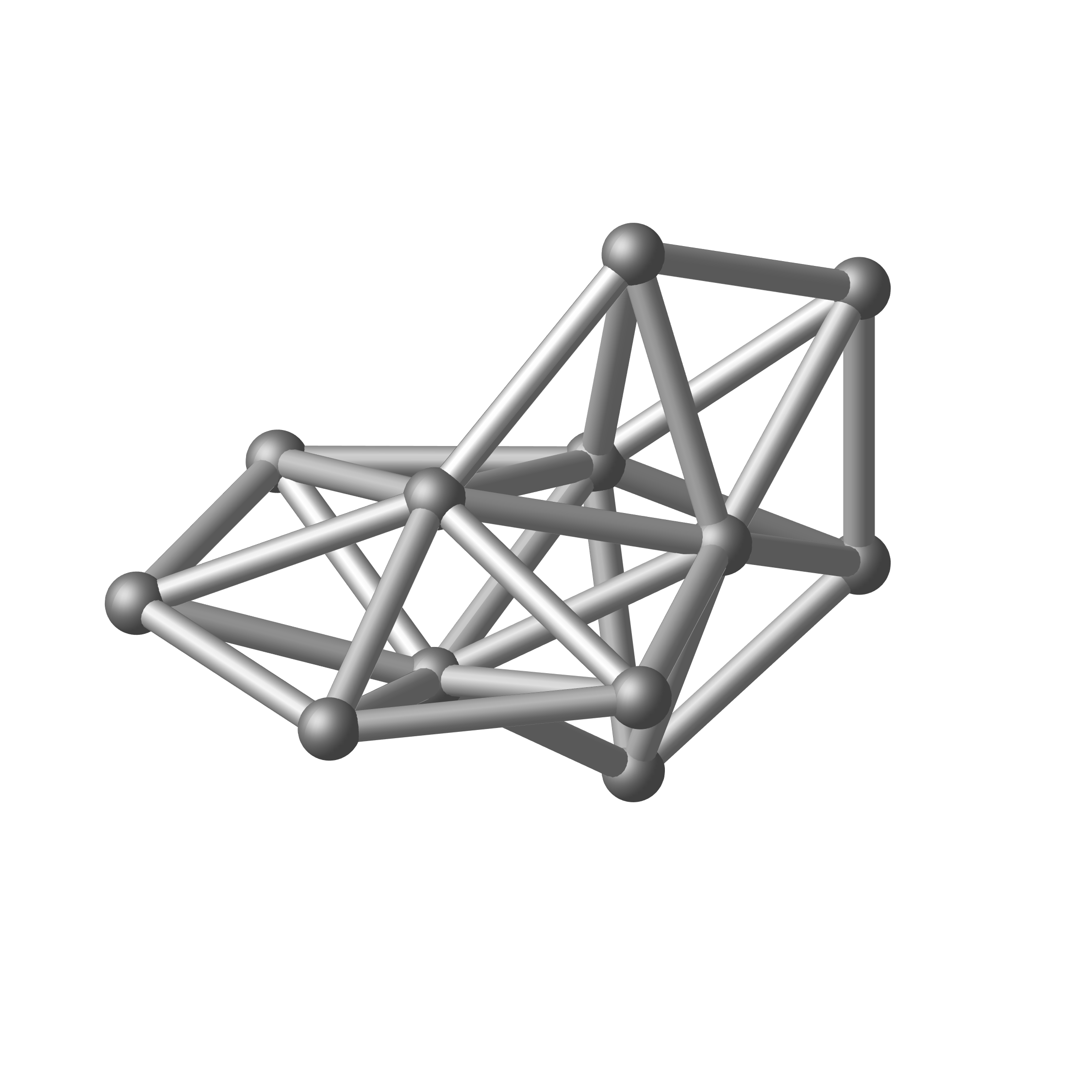}}
& 
(i) Siamese dipyramid, first embedding. \newline
12 vertices, unit edges
& 
$\eta_1=0$\newline
$\eta_2 = 0.0064$,  $\eta_3 = 0.0047$ \newline
$e_{\rm min}^* = 3.0\times 10^{-4}$\newline
\\
\raisebox{2ex - \height}{\includegraphics[trim={3.5cm 8cm 7cm 7cm},clip,width=2.4cm]{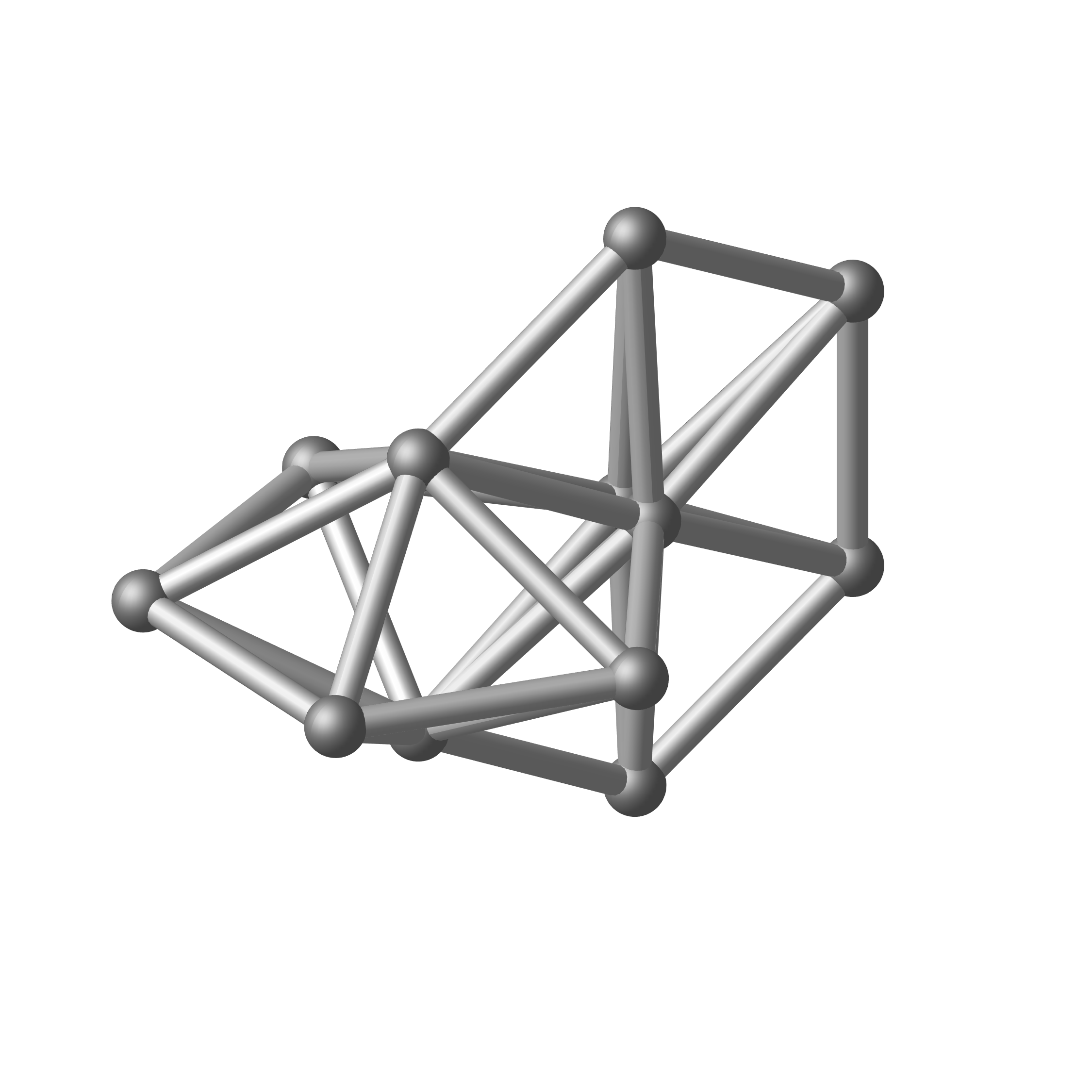}}
& 
(j) Siamese dipyramid, second embedding. \newline
Distance to first embedding = 0.49.
& 
$\eta_1=0$\newline
$\eta_2 = 0.0068$, $\eta_3 = 0.0050$ \newline
$e_{\rm min}^* = 3.3\times10^{-4}$\newline
\\
\raisebox{2ex - \height}{
\begin{tikzpicture}[thick,scale=0.75]
\node[vertex] (1) at (1,0) {};
\node[vertex] (2) at (0.6235,0.7818) {};
\node[vertex] (3) at (-0.2225,0.9749) {};
\node[vertex] (4) at (-0.9010, 0.4339) {};
\node[vertex] (5) at (-0.9010,  -0.4339) {};
\node[vertex] (6) at (-0.2225, -0.9749) {};
\node[vertex] (7) at (0.6235, -0.7818) {};
\draw (1) -- (4);
\draw (1) -- (5);
\draw (1) -- (6);
\draw (1) -- (7);
\draw (2) -- (4);
\draw (2) -- (5);
\draw (2) -- (6);
\draw (2) -- (7);
\draw (3) -- (4);
\draw (3) -- (5);
\draw (3) -- (6);
\draw (3) -- (7);
\end{tikzpicture}
}
& 
(k) $K_{3,4}$ (7 vertices, 12 edges), vertices on a circle\newline
globally rigid for generic $p$\newline
this $p$ is prestress stable, not first-order rigid
& 
$\eta_1=0$\newline
$\eta_2 = 0.027$, $\eta_3 = 0.019$ \newline
$e_{\rm min}^* = 1.3\times 10^{-5}$\newline
\\
\hline
    \end{tabular}
    \caption{Some frameworks and their lengthscales given by Theorems \ref{thm:framework}-\ref{thm:emin}, calculated using $\lambda = \lambda_0/2$. }
    \label{tbl:example}
\end{table}

Table \ref{tbl:example} shows $\eta_1,\eta_2,\eta_3,e_{\rm min}^*$ for a collection of frameworks with different properties, that we discuss in turn. 

\subsubsection{Small first-order rigid frameworks}

Examples (c)-(f) in Table \ref{tbl:example} are all first-order rigid, so $\eta_1=0$. They have significantly larger values of $\eta_2$, $e_{\rm min}^*$ than for our introductory examples, which were not first-order rigid. For Examples (c) and (d), a two-dimensional square with a diagonal edge and a three-dimensional tetrahedron, we calculated the distance to the other embeddings of each framework (i.e. those with the same edge lengths), using the Kabsch algorithm to find the minimum distance over rigid body motions \cite{Kabsch}. The distance is larger than $\eta_2$, as it must be, about 7-10 times larger for these examples.

\subsubsection{An isostatic framework that is not first-order rigid}

\begin{figure}
    \centering
   \raisebox{-0.8cm}{\includegraphics[trim={7cm 6cm 6cm 7cm},clip, width=0.28\linewidth]{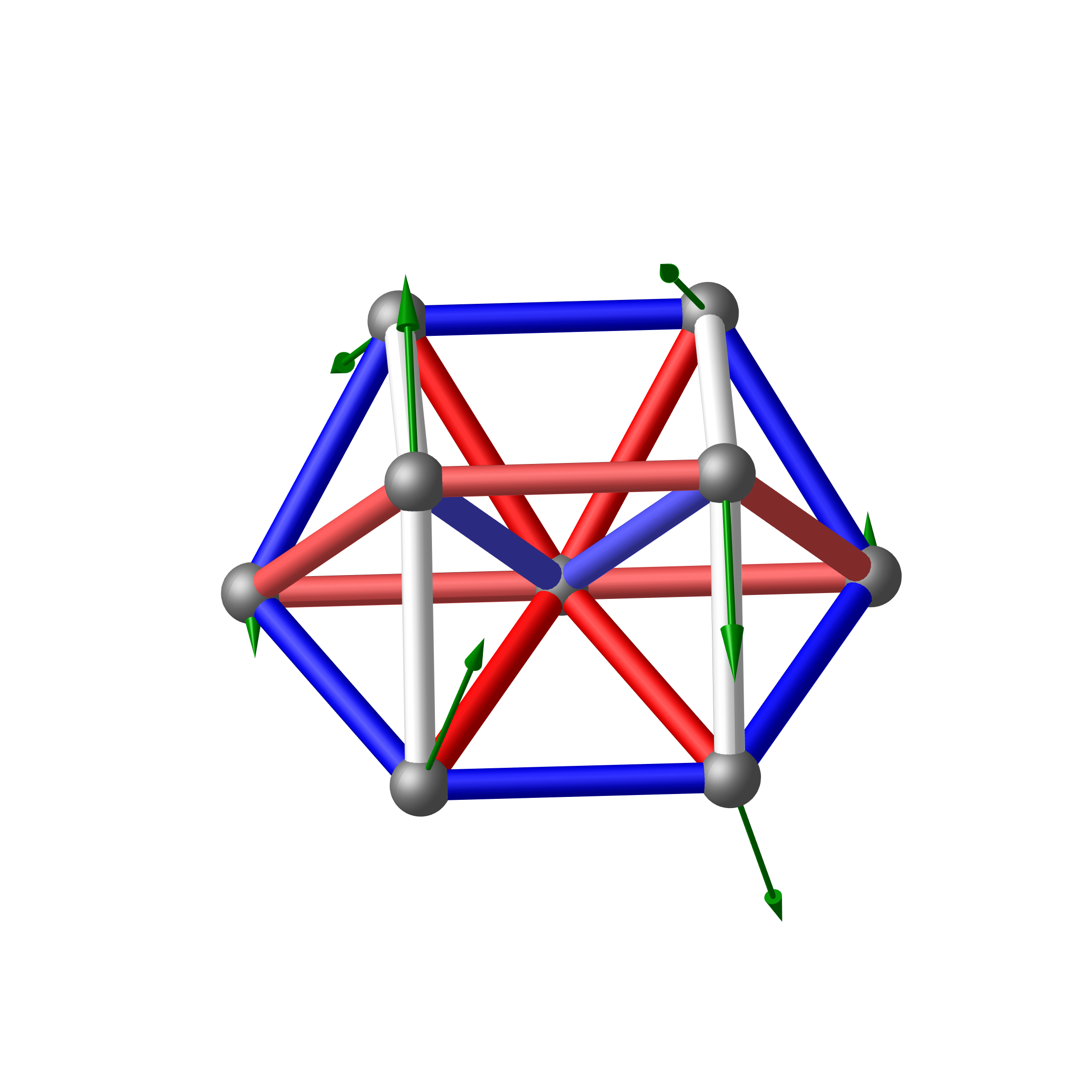}}
\quad
%
%
    \begin{tikzpicture}[thick,scale=1]
\node[vertex] (1) at (0,0) {};
\node[vertex] (2) at (0,2) {};
\node[vertex] (3) at (1,1) {};
\node[vertex] (4) at (3,1) {};
\node[vertex] (5) at (4,2) {};
\node[vertex] (6) at (4,0) {};
\draw[ultra thick,red!50] (1) -- (2);
\draw[ultra thick,blue!100] (1) -- (3);
\draw[ultra thick,blue!100] (2) -- (3);
\draw[ultra thick,blue!100] (3) -- (4);
\draw[ultra thick,red!25] (2) -- (5);
\draw[ultra thick,blue!100] (4) -- (5);
\draw[ultra thick,red!25] (1) -- (6);
\draw[ultra thick,blue!100] (4) -- (6);
\draw[ultra thick,red!50] (5) -- (6);
\draw[->,black!30!green] (1) -- ($ (1) + 1.5*(-0.3769,-0.0754) $);
\draw[->,black!30!green] (2) -- ($ (2) + 1.5*(0.3769,-0.0754) $);
\draw[->,black!30!green] (3) -- ($ (3) + 1.5*(0,-0.4523) $);
\draw[->,black!30!green] (4) -- ($ (4) + 1.5*(0,0.4523) $);
\draw[->,black!30!green] (5) -- ($ (5) + 1.5*(0.3769,0.0754) $);
\draw[->,black!30!green] (6) -- ($ (6) + 1.5*(-0.3769,0.0754) $);
\end{tikzpicture}
%
%
\quad
\raisebox{-0.25cm}{
    \begin{tikzpicture}[ultra thick,scale=1.5]
\node[vertex] (1) at (1,0) {};
\node[vertex] (2) at (0.6235,0.7818) {};
\node[vertex] (3) at (-0.2225,0.9749) {};
\node[vertex] (4) at (-0.9010, 0.4339) {};
\node[vertex] (5) at (-0.9010,  -0.4339) {};
\node[vertex] (6) at (-0.2225, -0.9749) {};
\node[vertex] (7) at (0.6235, -0.7818) {};
\draw[blue!44] (1) -- (4);
\draw[red!100] (2) -- (4);
\draw[blue!100] (3) -- (4);
\draw[red!40] (1) -- (5);
\draw[blue!72] (2) -- (5);
\draw[red!50] (3) -- (5);
\draw[red!34] (1) -- (6);
\draw[blue!50] (2) -- (6);
\draw[red!28] (3) -- (6);
\draw[blue!90] (1) -- (7);
\draw[red!90] (2) -- (7);
\draw[blue!40] (3) -- (7);
\draw[->,black!30!green,thick] (1) -- ($ (1) + 1.1*(0.2956,-0.2474) $);
\draw[->,black!30!green,thick] (2) -- ($ (2) + 1.1*(0.11,0.1380) $);
\draw[->,black!30!green,thick] (3) -- ($ (3) + 1.1*(-0.3070,0.2332) $);
\draw[->,black!30!green,thick] (4) -- ($ (4) + 1.1*(0.2468,-0.4613) $);
\draw[->,black!30!green,thick] (5) -- ($ (5) + 1.1*(0.2468,-0.0335) $);
\draw[->,black!30!green,thick] (6) -- ($ (6) + 1.1*(-0.0876,0.2332) $);
\draw[->,black!30!green,thick] (7) -- ($ (7) + 1.1*(-0.5047,0.1380) $);
\end{tikzpicture}
}

    \caption{Left, Middle: Isostatic frameworks that are prestress stable but not first-order rigid. Each has a single infinitesimal flex (green arrows) and a single self-stress (colors on bars, where blue is positive and red is negative.) 
    Right: $K_{3,4}$ with vertices on a circle. 
    For generic configurtations, it has one self-stress 
    and is globally rigid.
 When the vertices lie on a circle, it acquires an infinitesimal flex (green arrows) and extra self-stress. The self-stress plotted is the one that makes the framework prestress stable. }
    \label{fig:isostatic}
\end{figure}
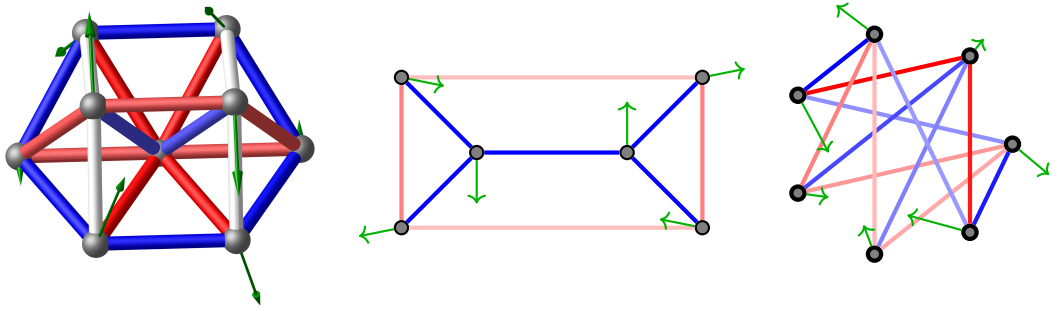

Next consider Example (g) in Table \ref{tbl:example}. This has unit edges and can be formed by putting spheres with unit diameters in contact. 
It has 9 vertices and 21 edges, so  generically it should be first-order rigid, since the number of variables required to describe the configuration, $9\times 3 = 27$, equals the number of edges plus the number of rigid-body degrees of freedom. We say such a configuration is \emph{isostatic}. In three dimensions a graph with $N$ vertices is isostatic when it has $3N-6$ edges, and in two dimensions it is isostatic when it has $2N-3$ edges. 

However, for this particular configuration, there is an additional infinitesimal flex and self-stress; the infinitesimal flex corresponds to ``twisting'' the two halves of the framework relative to each other (Figure \ref{fig:isostatic}; see also \cite{Kallus:2017hi}.) One can argue that the framework with perfect unit edges is prestress stable; in fact, it is the smallest packing of identical spheres that is prestress stable but not first-order rigid \cite{HolmesCerfon:2016wa}. 

When we represent the framework on a computer, we don't obtain perfect unit edge lengths, so the numerically represented framework $p$ should behave like a generic configuration, and hence should be  first-order rigid. Indeed, we solved \eqref{approxedgeeqn} for $p$ using a tolerance of $\delta = 4\epsilon\approx 9 \times 10^{-16}$ and a random initial condition. The singular values of $R(p)$ were
\begin{gather*}
 1.62 \times 10^{-9},\;
0.61,\;0.67,\;0.78,\;0.87,\;1.02,\;1.06,\;1.14,\;1.19,\;1.30,\;1.31,\; 1.35,\;1.41,\;\\
1.59,\;1.61,\;1.75,\;1.75,\;1.79,\;1.99,\;2.04,\;2.26.
\end{gather*}
The smallest singular value is $\sigma_0=1.62 \times 10^{-9} >0$, so $p$ is first-order rigid, i.e. $\eta_1=0$. We applied Corollary \ref{cor:firstorder} to calculate the other radii using this singular value, and found the outer radius to be $\eta_2 = 2.8\times 10^{-10}$, and the edge length barrier to be $e_{\rm min}^* = 6\times 10^{-19}$, both of which are minuscule. 

Can we do better, by using the almost-flex $v$ and almost-self-stress $w$ in our test? 
We repeated the test, choosing $\mathcal V = \mathcal V^{\sigma_0}$, $\mathcal W = \mathcal W^{\sigma_0}$ (see \eqref{VWsig}.)
With these choices, $\eta_1 > 0$ (see Table \ref{tbl:example} (g)), so the cluster is not provably rigid using this test. However, $D_{\rm pss} = 8.6\times 10^{-5}<0.5$, so there is a nearby prestress stable cluster.
The other lengthscales are $\eta_2 = 0.0145$, $e_{\rm min}^* = 4.3\times 10^{-6}$,  much bigger than they were using the first-order test above. 
In this case, although we know that the perturbed $p$ is rigid, from our first test, the utility of $\eta_1$ in the second test is to constrain the distance to nearby configurations with the same edge lengths. 
It could be the case -- and probably is -- that there is a nearby rigid configuration with the same edge lengths as the perturbed $p$, however, this nearby configuration must lie within $\eta_1$ of $p$; any other configuration is at least a distance of $\eta_2$ away, hence, much further. 
This example illustrates how one can choose $\mathcal V$, $\omega$ in different ways to prove different kinds of statements. 


Example (h) in Table \ref{tbl:example} is an isostatic cluster in two dimensions which is also prestress stable, but not first-order rigid.

\subsubsection{Siamese dipyramid}

Next we consider a polyhedron that is provably rigid (as a polyhedron), but, such that physical models of it ``feel'' flexible: they can be deformed by some finite amount, without any noticeable stretching or bending of the faces. 
This family of examples, called Siamese dipyramids, was introduced to the mathematics literature by Goldberg \cite{Goldberg:1978df}, and partially analyzed recently by Gorkavyy \& Fesenko \cite{Gorkavyy:2018gp}. 
The authors built our own physical model of a Siamese dipyramid using cardboard for the faces and scotch tape for the hinges, and it did not feel at all rigid; we could deform it significantly without noticeably bending or creasing the cardboard or ripping the hinges. 

We consider Siamese dipyramids with $N=12$ nodes and unit edge lengths. There are two such examples, shown in Table \ref{tbl:example} (i),(j), where they are plotted as  frameworks, with nodes at the vertices of the polyhedron and edges at the boundaries of the faces. The two examples differ in the widths of the pac-man-like discs that make them up: in example (i), the discs are the same width, so the framework can be superimposed on itself by a reflection and a rotation by $\pi/2$. In example (j), one disc has been flattened and one has been thickened. See \cite{Gorkavyy:2018gp} for a description of how to construct these examples. 


We were curious if our theory could give insight into why these examples feel so flexible. We treated each example as a framework and calculated the various lengthscales. For example (i), call its configuration $p_1$, the smallest singular value of $R(p_1)$ was $\sigma_0=0.0546$. Since $\sigma_0>0$, the framework is first-order rigid, however the singular value is significantly smaller than for a typical first-order rigid cluster. 
The outer radius and energy barrier are correspondingly small, $\eta_2 = 0.00638$, $e_{\rm min}^* = 2.98\times 10^{-4}$; the outer radius in particular is much smaller than for the other first-order rigid 
examples
we have considered, and smaller even than for some of the prestress stable sphere clusters. The constants for example (j), call its configuration $p_2$, were of a similar magnitude, $\sigma_0 = 0.0579$, $\eta_2 = 0.00676$, $e_{\rm min}^* = 3.35\times 10^{-4}$. 
Such small outer radii and edge length barriers suggest the two embeddings could be close to each other in configuration space, and that one can deform these examples without changing the edge lengths by much.

We calculated the actual distance between the configurations (using the Kabsch algorithm to find the minimum distance over rigid body motions) and found it to be about $|p_1-p_2|\approx 0.49$, significantly larger than either value of $\eta_2$. Therefore, although $\eta_2$ may be correlated with how rigid or flexible a structure feels in practice, it does not quantitatively predict the location of other embeddings.

\subsubsection{$K_{3,4}$}

The example above showed that the radii in our theory can correlate with how flexible a framework ``feels'', at least qualitatively. To show that feeling flexible does not always correlate with having another solution nearby, consider the graph  $K_{3,4}$ in $\R^2$, See Table \ref{tbl:example} (k).
This graph is generically globally rigid: for a generic configuration\footnote{
Generic means there is no algebraic relation with integer coefficients between coordinates.}, there is only one embedding of the graph in
the given dimension.  That $K_{3,4}$ is generically globally rigid in dimension $2$ follows from the classification of such graphs in \cite{jacksonjordan}.

However, when the vertices lie on a circle, the the $K_{3,4}$ framework acquires an extra infinitesimal flex and self-stress (Figure \ref{fig:isostatic}.) We took the vertices on a regular heptagon on the unit circle, and calculated the various lengthscales, see Table \ref{tbl:example}. When we perturb the vertices by small random amounts, these numbers don't change much, however for a random perturbation the resulting 
framework is globally rigid. Therefore, even though its outer radius is only $\eta_2 = 0.027$, there is actually no other embedding with the same edge 
lengths at all. 



\subsection{Clusters of unit spheres}

\begin{figure}
    \centering
    \includegraphics[width=0.48\linewidth]{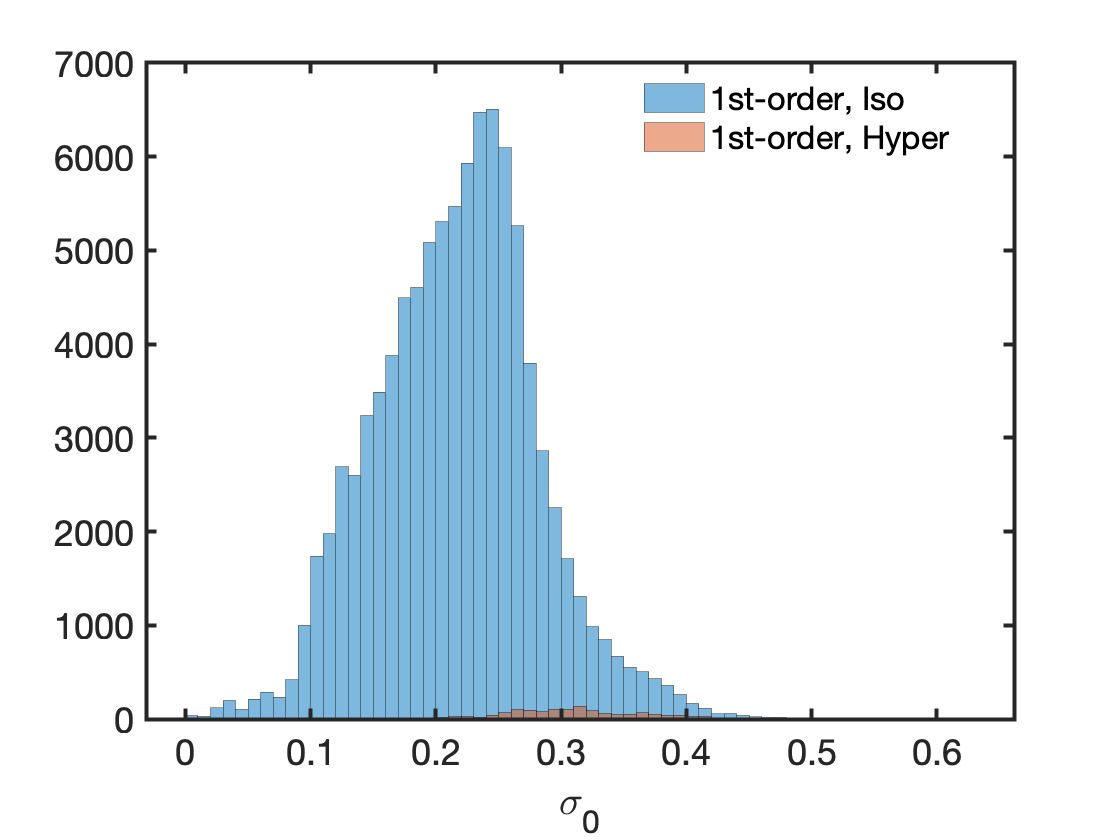}\hfill
    \includegraphics[width=0.48\linewidth]{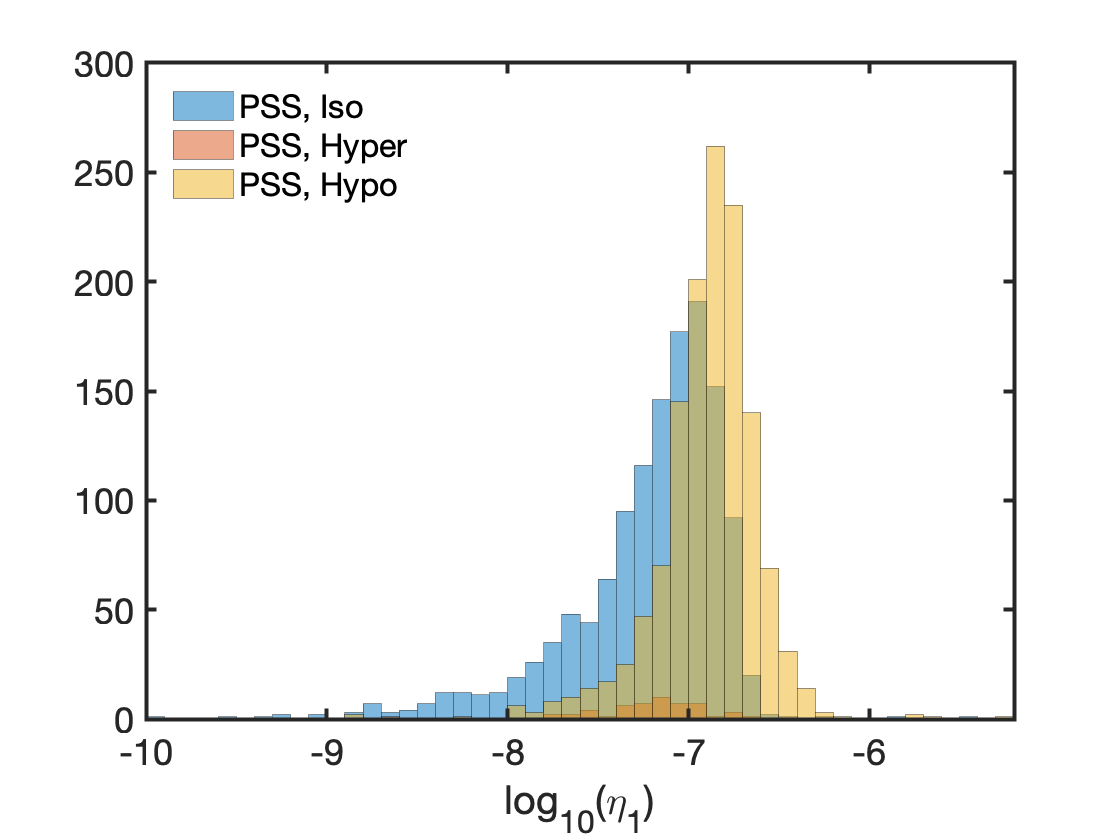}
    \\
    \includegraphics[width=0.48\linewidth]{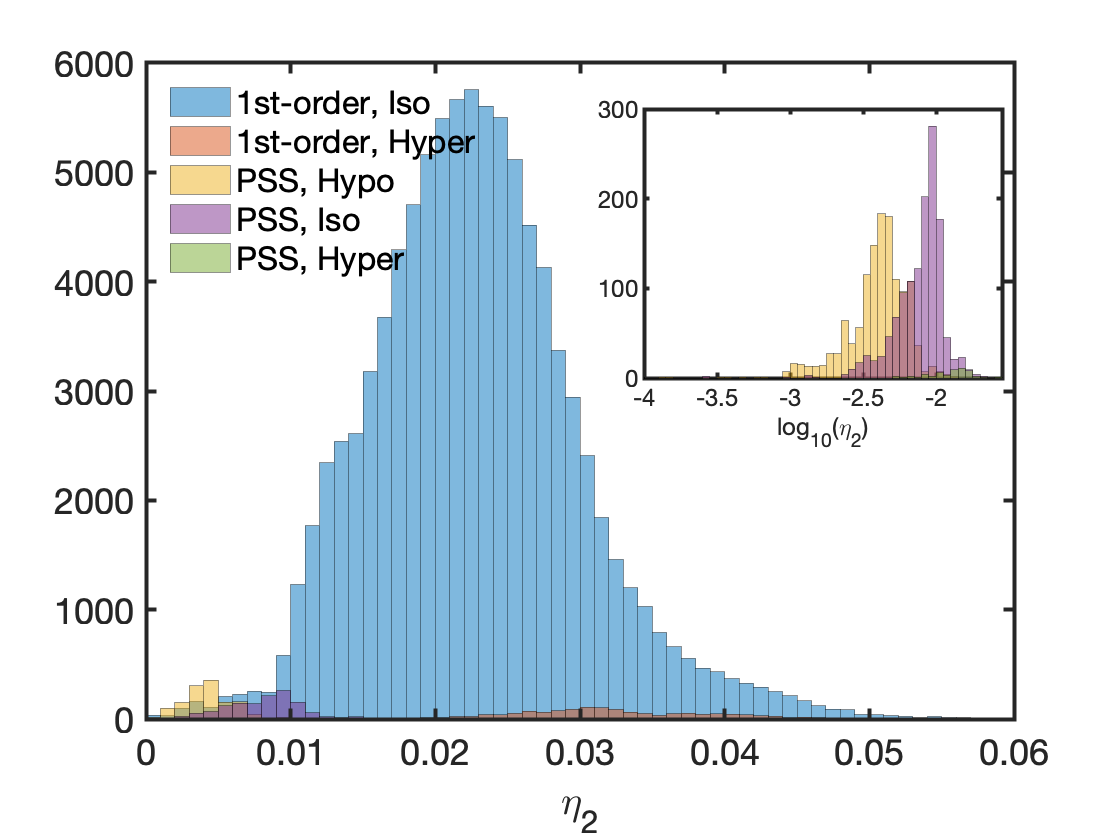}\hfill
    \includegraphics[width=0.48\linewidth]{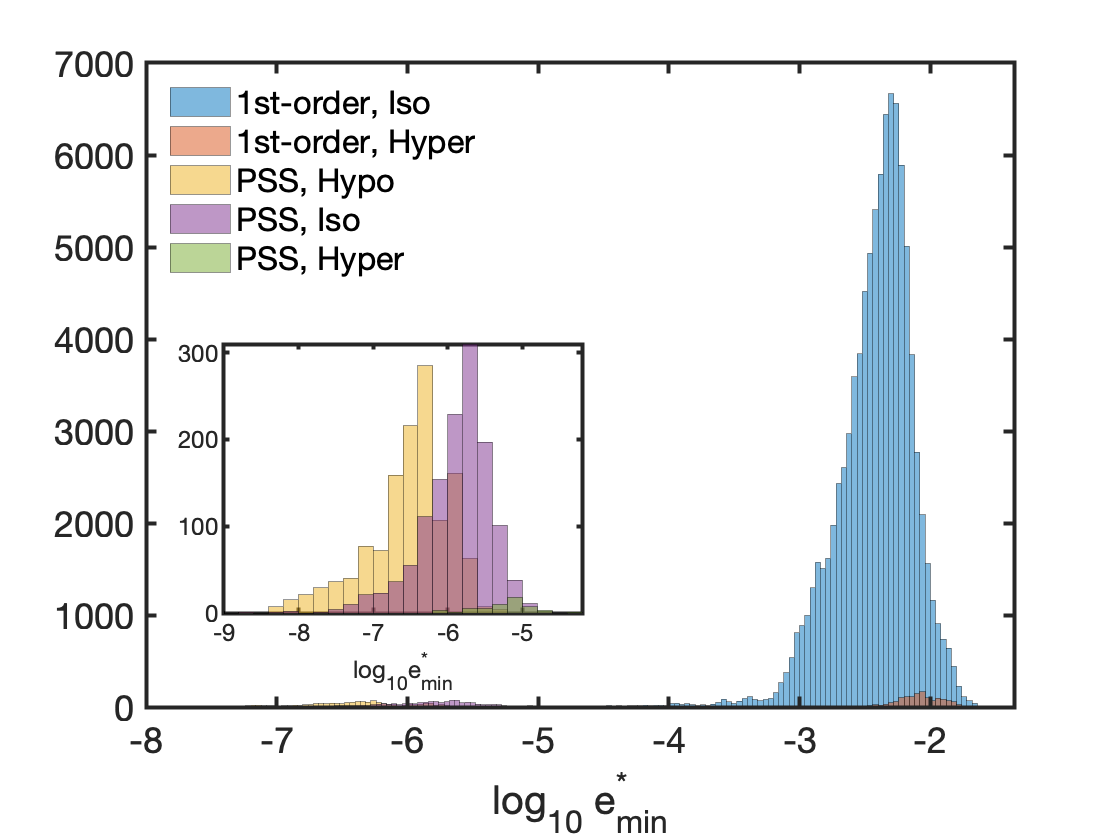}
    \caption{Lengthscales calculated for frameworks made from rigid clusters of $N=13$ identical spheres. The histograms are calculated for each type of cluster, indicated in the legends: ``1st-order'' (first-order rigid), ``PSS'' (prestress stable but not first-order rigid), ``Iso'' (isostatic, $3N-6$ edges), ``Hyper'' (hypostatic/overconstrained, $>3N-6$ edges), ``Hypo'' (hypostatic/underconstrained, $<3N-6$ contacts.) 
    }
    \label{fig:n13}
\end{figure}

For a more global picture of how the lengthscales behave for different types of frameworks, we turn to the set of rigid frameworks made from $N=13$ unit spheres in contact. A list of 98,529 such frameworks was presented in \cite{HolmesCerfon:2016wa}; this list is thought to be nearly complete. The clusters were found by solving \eqref{approxedgeeqn} numerically for different edge sets $\mathcal E$, using a tolerance of $\delta  \approx 9\times 10^{-16}$. 

Figure \ref{fig:n13} shows a histogram of the smallest singular value $\sigma_0$ for all first-order rigid clusters in the dataset. The average was $0.22$ with standard deviation $0.066$, so most singular values were close to the mean. However, the largest observed value was $\sigma_0 = 0.626$, and the smallest was $\sigma_0 = 0.00106$, so even within the first-order rigid clusters, the lengthscales $\eta_2, e_{\rm min}^*$ can vary by more than two orders of magnitude. 

We calculated $\eta_1,\eta_2,e_{\rm \min}^*$ for all clusters, using a cutoff of $10^{-7}$ for the singular values to build $\mathcal V,\mathcal W$, and choosing $\lambda = 0.8\lambda_0$.  Figure \ref{fig:n13} shows the histograms of values, grouped by type of cluster: first-order rigid or not, and number of edges: \emph{hypostatic} ($<3N-6$ edges), \emph{isostatic} ($3N-6$ edges), \emph{hyperstatic} ($>3N-6$ edges.) 
We see that hyperstatic  first-order rigid  clusters on average have slightly higher lengthscales than isostatic first-order rigid clusters. 
Non-first-order rigid clusters  have values of $\eta_2$ that are about one to two orders of magnitude smaller than first-order rigid clusters, and values of $e_{\rm min}^*$ that are about 3-4 orders of magnitude smaller. Within the non-first-order rigid clusters, these values are on average slightly smaller for hypostatic clusters, and slightly larger for hyperstatic clusters. Figure \ref{fig:n13} also shows the values of $\eta_1$ for the non-first-order rigid clusters, which are nonzero presumably because of numerical tolerance in solving \eqref{approxedgeeqn} and because we include small singular vectors in our tests. We expect the distribution of $\eta_1$-values reflects the distribution of tolerances with which we solve \eqref{approxedgeeqn}, since our numerical solver does not produce exactly the same error for each framework.


To verify that the theorems do indeed apply, we calculated $D$, $D_{\rm pss}$ from \eqref{D}, \eqref{Dpss}. 
The 20 largest values of $D$ were: 
\begin{gather*}
0.61,\;
    0.27,\;
    0.042,\;
    0.027,\;
    0.016,\;
    0.0035,\;
    0.0020,\;
    0.0018,\;
    0.0014,\;
    0.0009,\;
    0.0008,\;
    0.0007,\;\\
    0.0006,\;
    0.0004,\;
    0.0003,\;
    0.0003,\;
    0.0003,\;
    0.0003,\;
    0.0003,\;
    0.0003,\ldots
\end{gather*}
All clusters but one had values of $D$ less than $0.5$, and therefore Theorems \ref{thm:framework}, \ref{thm:outerbound} apply. The cluster with the largest value of $D$ (cluster 45601 in the dataset), for which the theorems don't apply, was isostatic, with one small flex $v$ and one small stress $w$ with singular value $\sigma_0=1.37\times10^{-10}$, and the corresponding inner product was $v^T\Omega(w)v = 2.03\times 10^{-4}$. Since this inner product is so tiny, the cluster had a small outer radius $\eta_2=1.7\times 10^{-6}$, and a large eigenvalue factor, $\bar\mu = 14569$, which causes $D$ to be large. We computed $D$ with different values of $\lambda$, but it never fell under the threshhold for the theorems. 
The cluster with the second-largest $D$-value (cluster 10386) was structurally similar to the one with the largest $D$. It had a similarly large value of $\bar \mu$, but nevertheless its $D$-value fell under the threshhold.
The authors built these clusters out of magnetic rods and ball-bearings, and they both felt very floppy. 

The 20 largest values of $D_{\rm pss}$  were
\begin{gather*}
23.7,\;
   10.2,\;
    4.5,\;
    1.4,\;
    0.48,\;
    0.15,\;
    0.11,\;
    0.11,\;
    0.092,\;
    0.055,\;
    0.050,\;
    0.020,\;
    0.018,\;\\
    0.015,\;
    0.014,\;
    0.013,\;
    0.013,\;
    0.013,\;
    0.011,\;
    0.011,\ldots
\end{gather*}
All but four clusters had $D_{\rm pss} < 0.5$. Therefore, Theorem \ref{thm:pss} guarantees that there is a nearby prestress stable cluster for all but these four clusters. 


\subsection{$L$ vs $\lambda/\lambda_0$}\label{sec:LvsLam}

\begin{figure}
    \centering
    \includegraphics[width=0.48\linewidth]{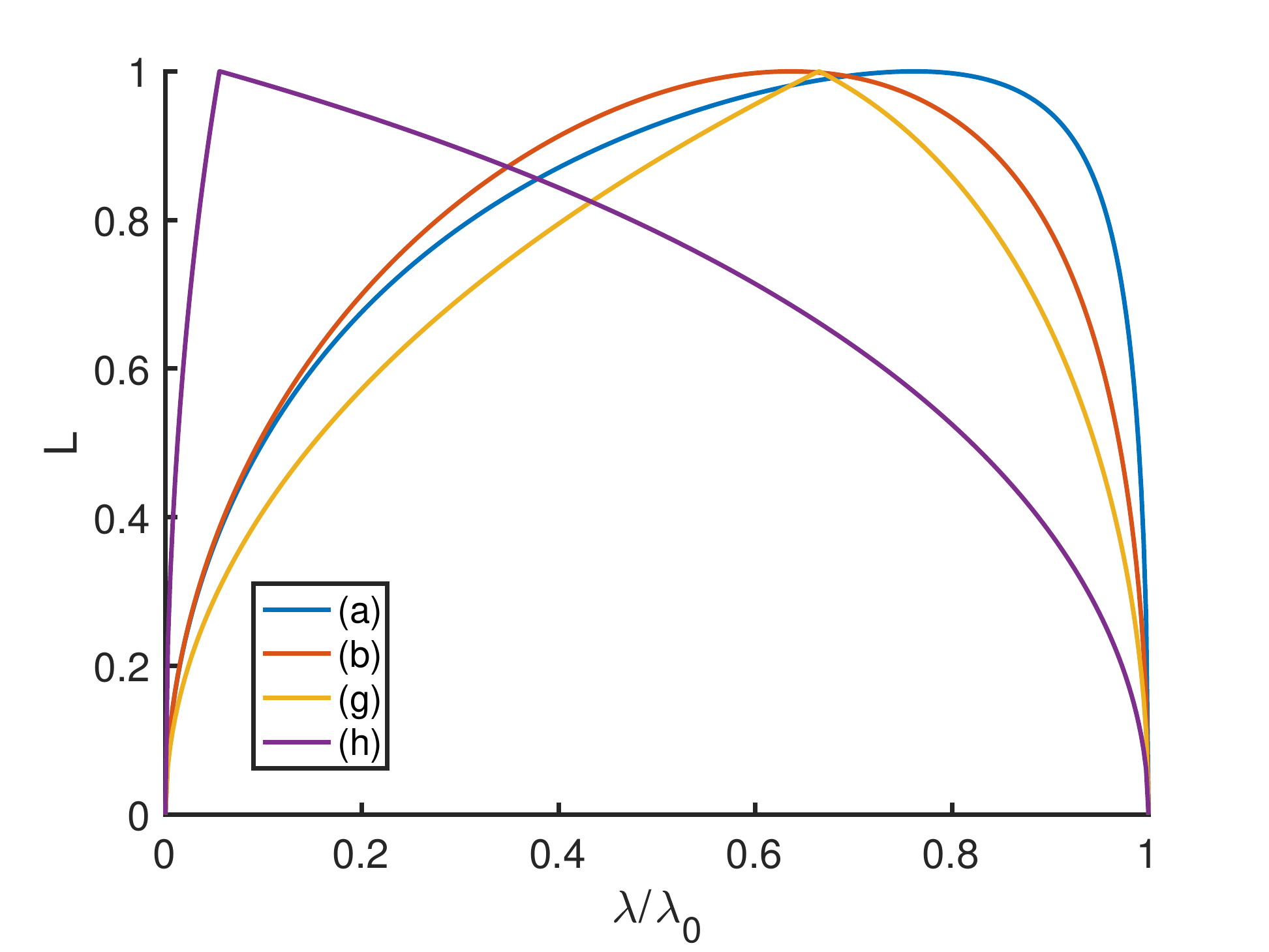}
    \caption{Lengthscale $L$ from \eqref{Lmu} for examples (a),(c),(g),(h) in Table \ref{tbl:example}, as a function of $\lambda$. All curves have been scaled so their maximum value is 1. }
    \label{fig:L}
\end{figure}


Implementing the theory requires choosing a value $\lambda\in (0,\lambda_0)$. 
Larger values of $\lambda$ lead to smaller $\eta_1$ (see \eqref{eta1}), which is usually desirable. 
However, $\lambda$ also changes the lengthscale $L$ (see \eqref{Lmu}) which sets $\eta_2,e_{\rm min}^*$, and it does so in a nontrivial way, since $\kappa$ also depends on $\lambda$. 

We computed $L(\lambda)$ for different values of $\lambda$ for examples (a),(b),(g),(h) in Tables \ref{tbl:intro}, \ref{tbl:example}. Figure \ref{fig:L} shows that there is no common features to the dependence. Although we must have $L(\lambda)=0$ for $\lambda=0,1$, 
the maximum of $L(\lambda)$ could be anywhere: for the first three examples it is somewhere in the interval $(0.6,0.8)$, but for the last example it is near $\lambda\approx 0.055$. However, $L(\lambda)$ does not change by more than about a factor of 2, within the range $\lambda\in (0.2,0.8)$. 
Interestingly, $L(\lambda)$ does not appear to always be a differentiable function of $\lambda$.

\section{Properties of a function that is ``almost'' at a critical point}\label{sec:energypropositions}

This section collects general results about a three-times continuously differentiable function that almost but doesn't quite pass the second derivative test. These results will be used to prove the theorems in Section \ref{sec:mainresults}, but they may additionally be of independent interest. 

Consider a function $H\in C^3(\R^n)$, which we think of as an energy. 
The number $n$ in this section can be any dimension and does not have to be related to frameworks.
Let $\Sball{n} = \{v\in \R^{n} : |v|=1\}$ 
be the the unit sphere  in $\R^n$. 
When we write $|v|=1$, we always mean $v\in \Sball{n}$. 
Let the $k$th directional derivative in direction $v\in \Sball{n}$  evaluated at point $q$ be $H^{(k)}|_q(v)$. 
This derivative can be evaluated using the chain rule as 
\begin{equation}\label{directionalderiv}
H^{(k)}|_{q}(v) = \frac{d^k}{dt^k}\Bigg|_{t=0} H(q + tv)\,.
\end{equation}
We sometimes write the first and second derivatives as $H'|_q(v)=\grad H(q)\cdot v$, $H''|_q(v)=v^T\grad\grad H(q)v$. 
Let $B_r(p)$ be the open ball of radius $r$ centered at $p$, and let $\bar B_r(p)$ be its closure.




We are interested in the setting where $|\grad H(p)|$ is sufficiently small,  $\grad \grad H(p)$ is sufficiently positive definite, and  $|H'''|_{p+rv}(v)|$ is not too large for $r$ within a certain range, in a way to be defined precisely later. 
These assumptions imply that $p$ is ``almost'' a positive definite critical point, and that the third derivative doesn't alter the positive definiteness of the Hessian in a sufficiently large ball. 
Under appropriate assumptions we will show three main  propositions, which are stated informally as:

\begin{enumerate}[(I)]
    \item There is a radius $\eta_1$ such that any continuous path starting at $p$ on which the energy doesn't increase, cannot leave the ball of radius $\eta_1$. 
    \item There is a radius $\eta_2>\eta_1$ such that the energy is positive on the annulus $(\eta_1,\eta_2)$. Hence, any continuous path on which the energy is not larger than $H(p)$, and such that the path contains no points in the ball of radius $\eta_1$, must lie entirely outside  the ball of radius $\eta_2$.
    \item A continuous path from $p$ to a point $q$ with $|q-p|=\eta$ must cross an energy barrier of at least $\Delta H_{\rm min}(\eta)$, a function of $\eta$ given explicitly from  the properties of the energy. 
\end{enumerate}\medskip

The numbering of the propositions parallels the numbering of the theorems they will be used to prove. 

Proving these propositions is technical so we do this in steps. 
We start by studying the behaviour of cubic univariate
polynomials as the cubic coefficient varies, in Section \ref{sec:lemmas}. We show that for a polynomial $at^3+bt^2+ct$, where $b,c$ are fixed with $b>0$, and $a$ varies within limits, 
there is a fixed interval in $|t|$ over which the polynomial remains positive. The argument is detailed but not deep; the main theoretical tool is the quadratic formula. 

Next, in Section \ref{sec:energypropositionscubic}, we study multivariate cubic functions.
We use the results about cubic univariate polynomials to study the behaviour of the function along lines in different directions, to obtain an interval that is independent of direction over which the function is positive on each line. 
The main ideas behind the more general propositions are contained in this section, so these propositions should be possible to understand from reading this section only.

A reader undaunted by technical detail can then continue to Section \ref{sec:energypropositionsgeneral}, which considers arbitrary functions that are three-times continuously  differentiable. By Taylor-expanding a function along each direction we obtain a cubic univariate polynomial with a varying cubic term, to which we may apply results from Section \ref{sec:lemmas}. The additional challenge here is that the bounds on the cubic term, and hence the estimates of the interval over which the function is positive, depend on the distance to which the Taylor expansion is valid, which in turn depend on the interval of positivity that we calculate. Therefore, we have to be careful to avoid circular definitions.


\subsection{Cubic univariate function} \label{sec:lemmas}

We start with some lemmas regarding the behaviour of cubic univariate
functions as the cubic term varies. 

\begin{lemma}\label{lem:cubic}
Consider the function
\begin{equation}\label{gt}
g(t,a) = at^3+bt^2+ct\,.
\end{equation}
Suppose $b>0$ and $c\neq 0$, and let $t_* = \frac{2|c|}{b}$.  Let $\bar a$ be a number such that $0<\bar a < \frac{b^2}{4|c|}$, and let $t_1^+(a) = \frac{b}{2|a|}(1+\sqrt{1-4|ac|/b^2})$. 
Then 
\begin{equation}\label{gta}
g(t,a) > 0 \qquad \forall\: (t,a) \; \text{ such that } \; |t|\in [t_*, t_1^+(\bar a)), \quad  |a| \leq \bar a\,,
\end{equation}
and furthermore this interval is non-empty. 
\end{lemma}

The consequences of this Lemma are illustrated in Figure \ref{fig:polynomials}. 

\begin{figure}
\includegraphics[width=0.48\linewidth]{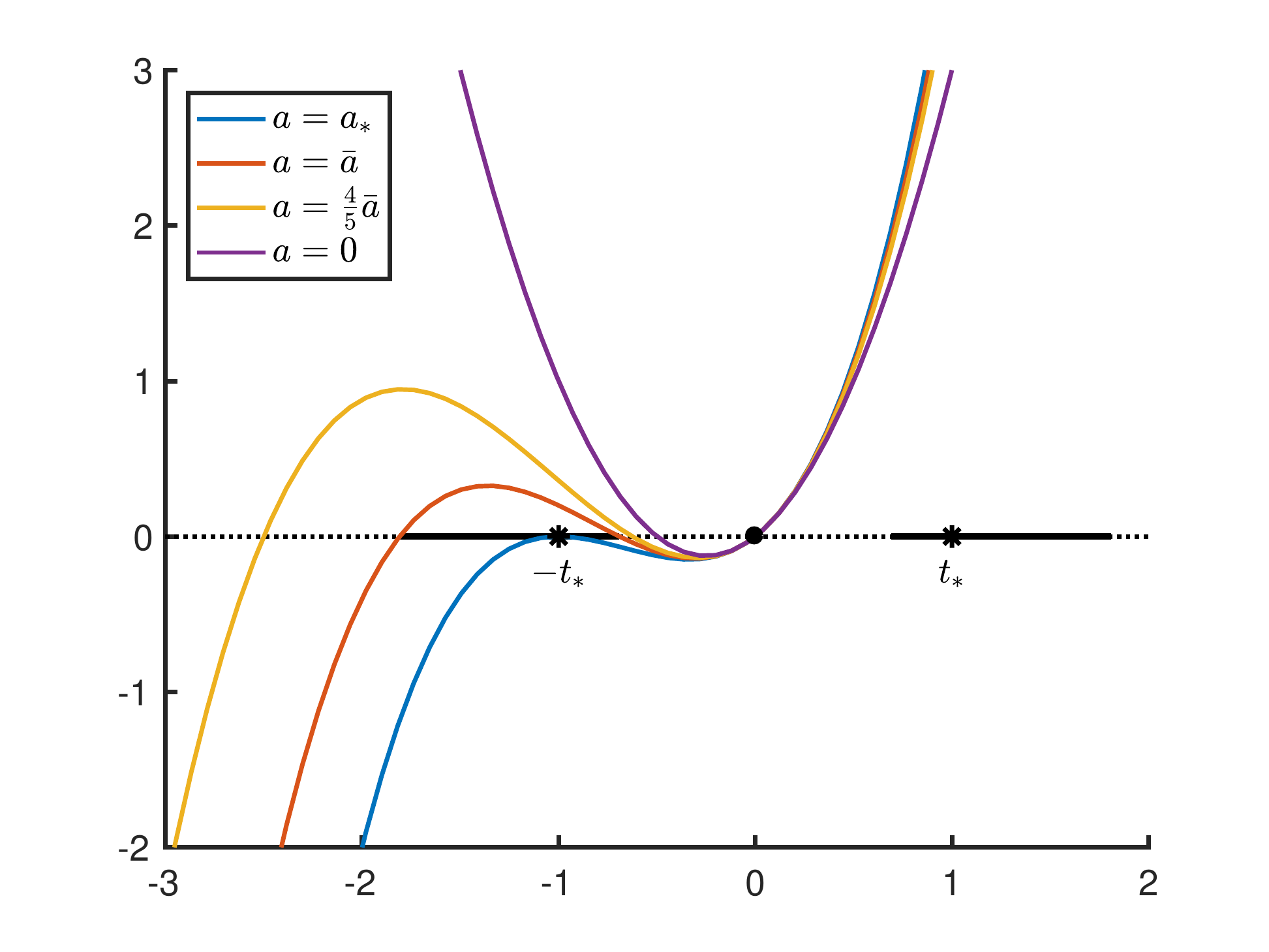}
\includegraphics[width=0.48\linewidth]{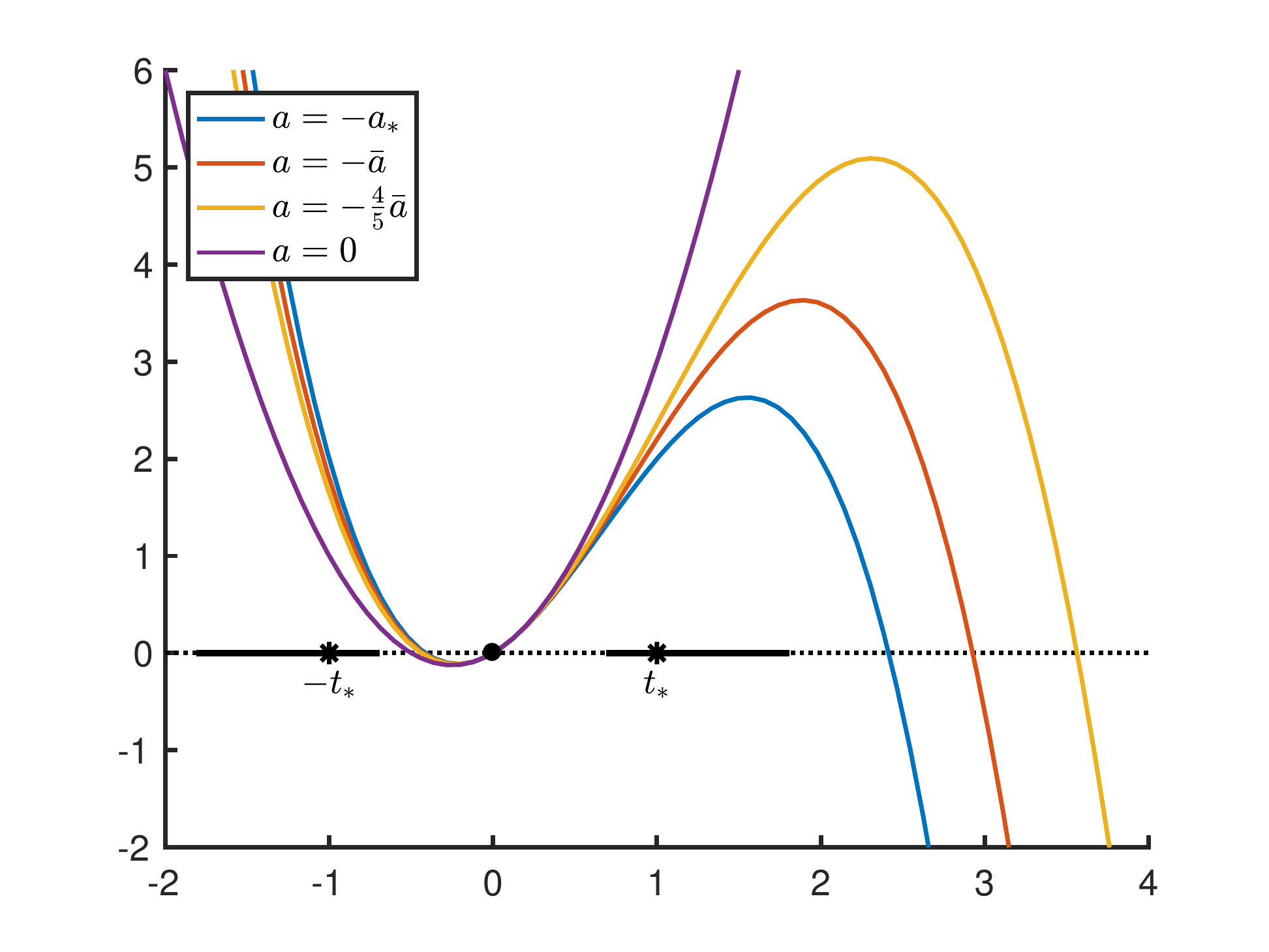}
\caption{Example polynomials $g(t,a)=at^3+bt^2+ct$ from Lemma \ref{lem:cubic}, with $b=2,c=1$, and varying $a$ as indicated in the legend. Here $a_* = b^2/4c$, $t_* = 2c/b$, and we chose $\bar a = 0.8a_*$. The intervals for $t$ from \eqref{interval} where $g(t,a)$ is positive for all $-\bar{a}\leq a\leq \bar{a}$ are shown in solid black lines. Left: $0\leq a\leq a_*$. For $0<a<a_*$, both roots $t_1(a)$, $t_2(a)$ (see \eqref{t12a}) are negative and $-t_*$ lies between them, so the polynomials are positive in some interval surrounding $-t_*$, as well as for all $t>0$. Right: $-a_*\leq a\leq 0$. For $a<0$, the smallest root is negative and satisfies $-t_*<t_1(a)<0$, so the polynomials are positive in an interval surrounding $-t_*$. The largest root is positive, $t_2(a)>0$, and decreases as $a$ decreases but does not hit $t_*$ until $a$ is well below $-\bar a$ (even $-a_*$), so there is also an interval surrounding $t_*$ where the polynomials are positive.}
\label{fig:polynomials}
\end{figure}

\begin{proof}
We can assume without loss of generality that $c>0$, since $g(t,a;c) = g(-t,-a;-c)$ (including here the dependence of $g$ on parameter $c$) and the conditions and statement of the lemma only depend on $|t|$, $|a|$, $|c|$. 

For fixed $a$, the solutions to the cubic polynomial equation $g(t,a)=0$ are $t=\{0,t_1(a),t_2(a)\}$, where
\begin{equation}\label{t12a}
t_1(a) = \frac{-b-\sqrt{b^2-4ac}}{2a}, \qquad t_2(a) = \frac{-b+\sqrt{b^2-4ac}}{2a}\,.
\end{equation} 
Notice that, for $a>0$ (and $c>0$), $t_1^+(a) = |t_1(a)|$, so our argument will proceed by analyzing the behaviour of $t_1(a)$. 
Let $a_* = b^2/4|c|$ be the critical value of $a$ where the discriminant of the roots above vanishes. 
The two varying roots above are imaginary for $a>a_*$, real for $a\leq a_*$, and at the critical value they are equal, $t_1(a_*) = t_2(a_*) = -t_*$. 
In addition we have that for $0 < a \leq a_*$ the varying roots are both negative, $t_1(a),t_2(a)<0$. 
 
We will establish the following result (which is actually stronger than \eqref{gta}):
\begin{equation}\label{interval}
g(t,a) > 0 \quad \text{ for } (t,a) \text{ such that }|t| \in (|t_2(\bar a)|,|t_1(\bar a)|), \;\; |a| \leq \bar a\,.
\end{equation}
This will imply \eqref{gta} after two short calculations. First, we must verify that $t_*\in (|t_2(\bar a)|,|t_1(\bar a)|)$. We do this by checking that $t_1(\bar a)<-t_*<t_2(\bar a)$. 
Differentiating the quadratic equation $at^2+bt+c=0$ implicitly with respect to $a$ gives $\partial_at_i = \frac{-t^2_i}{2at_i + b}$. Substituting the expressions for each root gives 
\begin{equation*}
\partial_a t_2 = \frac{-t_2^2}{\sqrt{b^2-4ac}}\quad < 0,\qquad 
\partial_a t_1 = \frac{-t_1^2}{-\sqrt{b^2-4ac}}\quad > 0\,.
\end{equation*}
Since $\bar a < a_*$ we have $t_2(\bar a) > t_2(a_*) = -t_*$, $t_1(\bar a) < t_1(a_*) = -t_*$. 

Second, we compute directly that $|t_1(\bar a)| - t_* = \frac{\delta+\sqrt{\delta}}{2\bar a/b}$, where $\delta = 1-4|\bar ac|/b^2 > 0$ by assumption. 
Therefore $\left[t_*, t_*+\frac{\delta+\sqrt{\delta}}{2\bar a/b}\right) \subset (|t_2(\bar a)|,|t_1(\bar a)|)$ so result \eqref{interval} implies result \eqref{gta}. 

Now we turn to showing \eqref{interval}. 
Our argument follows the sketch in Figure \ref{fig:polynomials}. We consider the two cases $0\leq a\leq \bar a$ and $-\bar a \leq a < 0$ in turn, and for each case, consider $t>0$ and $t<0$ separately. 

First consider $a \geq 0$ (Figure \ref{fig:polynomials}, left.)  When $t>0$ the cubic function $g(t,a)$ is a sum of nonnegative monomials, at least one of which is strictly positive (namely $bt^2>0$), so the cubic is positive, and \eqref{interval} holds for $a\geq 0$, $t>0$. It remains to consider $t<0$. For $0<a< a_*$,  the roots $t_1(a),t_2(a)$ are both negative and the cubic function is positive in between them, i.e. $g(t,a) > 0$ for $t \in (t_1(a),t_2(a))$. 
Therefore since $0<\bar a < a_*$, it is true that $g(t,\bar a)>0$ for $t \in (t_1(\bar a),t_2(\bar a))$. 
To extend this fixed interval to other values of $a$, we calculate $\partial_a g(t,a) = t^3 < 0$. 
Therefore  $g(t,a)$ strictly increases as $a$ decreases, so $g(t,a)>0$ for $t \in (t_1(\bar a),t_2(\bar a))$ whenever  $a\leq \bar a$, establishing \eqref{interval} for $a\leq\bar a$ and $t<0$. Note the latter argument includes the case $a<0$. 

Next consider $a<0$ (Figure \ref{fig:polynomials}, right). We have already established \eqref{interval} for $t<0$, so we need only consider $t>0$. 
Since $t_2(a)>0$, we see that $g(t,a)>0$ for $t\in (0,t_2(a))$, and $g(t,a)<0$ for $t>t_2(a)$. By direct calculation we find 
$g(|t_1(\bar a)|,-\bar a) = 2c >0$, 
 so $g(t,-\bar a) >0$ for $t\in(0,|t_1(\bar a)|]$. 
 Again we consider how $g$ varies with $a$, to find $\partial_a g(t,a) = t^3 > 0$ for $t>0$. Therefore $g(t,a)>0$ for all $t\in(0,|t_1(\bar a)|]$, $a\geq -\bar a$. Therefore \eqref{interval} holds for $a<0,t>0$ also. 
\end{proof}

 We need a similar lemma to deal with the case when $c=0$. 

 \begin{lemma}\label{lem:cubic0}
 Consider again the cubic function \eqref{gt} but now suppose $c=0$, in addition to $b>0$. 
 Let $0<\bar a < \infty$. 
 Then 
 \begin{equation}\label{gta0}
 g(t,a) > 0 \qquad \forall\: (t,a) \text{ such that }  |t| \in (0,\frac{b}{\bar a} )\,,\quad |a| \leq \bar a\,.
 \end{equation}
 \end{lemma}
 
 \begin{proof}
 We factor $g(t,a) = at^2(t+\frac{b}{a})$. As a function of $t$, this has a single root at $t=-\frac{b}{a}$ and a double root at $t=0$. If $a>0$ then $g(t,a)$ is negative for $t<-\frac{b}{a}$ and positive otherwise. If $a<0$ then $g(t,a)$ is negative for $t>-\frac{b}{a}$ and positive otherwise. If $a=0$ then $g(t,a)>0$ for all $t\neq 0$. Therefore $g(t,a) >0$ for all $|t| \in (0,b/a)$, and taking the intersection of the intervals over all $|a|\leq \bar a$ gives \eqref{gta0}. 
 \end{proof}

We need one more lemma which will be used to construct an energy barrier. 
\begin{lemma}\label{lem:cubicbounds}
Consider the cubic function \eqref{gt} with $b>0$ and where $\bar a\geq0$ is any nonnegative real number. Then for $t>0$, 
\begin{equation}\label{gbounds}
   g(\pm t,a)
   \quad\geq\quad f(t) \equiv -\bar a t^3 + bt^2 - |c|t  \qquad \text{for }\; |a|\leq \bar a\,.
\end{equation}
\end{lemma}

\begin{proof}
For $a\geq -\bar a$, 
\[
g(t,a) - f(t) = (a+\bar a)t^3 + (c+|c|)t \;\;\geq\;\; 0\,.
\]
For $a\leq \bar a$,
\[
g(-t,a) - f(t) = (\bar a - a)t^3 + (|c|-c)t \;\;\geq\;\; 0\,.
\]
Therefore for $|a|\leq \bar a$, both $g(t,a)\geq f(t)$, and $g(-t,a)\geq f(t)$, hold. 
\end{proof}

\subsection{Cubic energy function}\label{sec:energypropositionscubic}

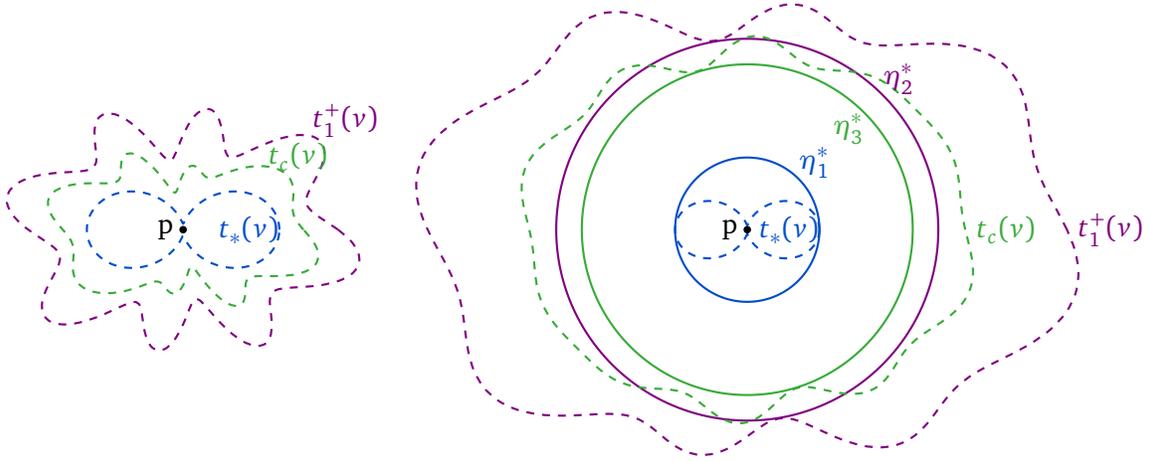
\begin{figure}
    \centering
    \pgfmathsetmacro{\etaA}{0.96}
    \pgfmathsetmacro{\etaB}{2.54}
    \pgfmathsetmacro{\etaC}{2.2}
    \pgfmathsetmacro{\xs}{7.5}
    \pgfmathsetmacro{\xscm}{\xs cm}
    \definecolor{myblue2}{rgb}{0, 0.3 0.8}
    \definecolor{mygreen2}{rgb}{0.22, 0.67, 0.22}
\begin{tikzpicture}
[declare function = {r1(\x) = 0.68 + 0.6*cos(2*\x);
                     r2(\x) = 1.66 + 0.4*cos(2*\x) + 0.35*cos(8*\x+180/2);
                     r3(\x) = 0.5*r1(\x) + 0.5*r2(\x) + 0.07*cos(10*\x+180);
                     r6(\x) = 0.51 + 0.45*cos(2*\x);
                     r4(\x) = 2.5*(1.4 + 0.3*cos(2*\x) + 0.1*cos(8*\x+180/2);
                     r5(\x) = 0.6*(2.1 + 0.1*cos(4*\x+180/2) - 0.07*cos(6*\x) + 0.06*cos(4*\x-180/3) - 0.1*cos(8*\x+5*180/6)+ 0.15*cos(10*\x+180)) + 0.4*r4(\x);}]
    \draw [myblue2,dashed,thick,domain=0:360,samples=100] plot [smooth] ({r1(\x)*cos(\x)},{r1(\x)*sin(\x)});
    \draw [violet,dashed,thick,domain=0:360,samples=100] plot [smooth] ({r2(\x)*cos(\x)},{r2(\x)*sin(\x)});
    \draw [mygreen2,dashed,thick,domain=0:360,samples=100] plot [smooth] ({r3(\x)*cos(\x)},{r3(\x)*sin(\x)});
    \fill [black] (0,0) circle (0.05) node[anchor=east] {p}; 
    \draw (1.28,0) node[myblue2,anchor=east, inner sep=0pt] {$t_*(v)$};
    \draw (2.15,1.45) node[violet] {$t_1^+(v)$};
    \draw (1.52,0.97) node[mygreen2] {$t_c(v)$};

    \draw[myblue2,thick,domain=0:360,samples=100,xshift=\xscm] plot [smooth] ({\etaA*cos(\x)},{\etaA*sin(\x)});
    \draw[violet,thick,domain=0:360,samples=100,xshift=\xscm] plot [smooth] ({\etaB*cos(\x)},{\etaB*sin(\x)});
    \draw[mygreen2,thick,domain=0:360,samples=100,xshift=\xscm] plot [smooth] ({\etaC*cos(\x)},{\etaC*sin(\x)});
    \draw [myblue2,dashed,thick,domain=0:360,samples=100,xshift=\xscm] plot [smooth] ({r6(\x)*cos(\x)},{r6(\x)*sin(\x)});
    \draw [violet,dashed,thick,domain=0:360,samples=100,xshift=\xscm] plot [smooth] ({r4(\x)*cos(\x)},{r4(\x)*sin(\x)});
    \draw [mygreen2,dashed,thick,domain=0:360,samples=100,xshift=\xscm] plot [smooth] ({r5(\x)*cos(\x)},{r5(\x)*sin(\x)});
    \fill [black] (\xs,0) circle (0.05) node[anchor=east] {p}; 
    \draw (\xs+0.96,0) node[myblue2,anchor=east,inner sep=0pt] {$t_*(v)$};
    \draw (\xs+4.25,0) node[violet,anchor=west] {$t_1^+(v)$};
    \draw (\xs+2.9,0) node[mygreen2,anchor=west] {$t_c(v)$};
    \draw ({\xs+\etaA*cos(45)},{\etaA*sin(45)}) node[myblue2,anchor=south west, inner sep = 0pt] {$\eta_1^*$};
    \draw ({\xs+\etaB*cos(45)},{\etaB*sin(45)}) node[violet,anchor=south west, inner sep = 0pt] {$\eta_2^*$};
    \draw ({\xs+\etaC*cos(45)},{\etaC*sin(45)}) node[mygreen2,anchor=north east, inner sep = 0pt] {$\eta_3^*$};
\end{tikzpicture}
\caption{Schematic for a cubic energy function illustrating the relation between $t_*(v)$, $t_c(v)$, $t_1^+(v)$, defined in \eqref{t1plus_cubic}, \eqref{eta3star_cubic}, where $v$ is a unit vector indicating the direction. 
Left: If a weak discriminant condition holds (see \eqref{disc0}), then the energy $H(p+tv)$ is always positive for $t\in (t_*(v),t_1^+(v))$ (Lemmas \ref{lem:cubic}, \ref{lem:cubic0}), with an interior maximum at approximately $t_c(v)$. 
Right: If a strong discriminant condition holds (see \eqref{disc1}), then $t_*(v)\leq \eta_1^* < \eta_2^* \leq t_1^+(v)$, so the energy is positive over the annulus $t\in (\eta_1^*,\eta_2^*)$. Under a slightly stronger discriminant condition (see \eqref{disc3}), we obtain an explicit lower bound on $H(q)-H(p)$ for $|q-p|$ up to some distance $\eta_3^* > (3/2)\eta_1^*$; see \eqref{eta3star_cubic}, \eqref{Hmin_cubic}; furthermore this stronger condition implies that $\eta_3^* < \eta_2^*$. 
}
    \label{fig:squiggly_lines}
\end{figure}

Consider a function $H\in C^3(\R^n)$ that is cubic, with $H(p) = 0$. 
The function can be written, for $v\in\Sball{n}$, $t\in \R$, as
\begin{equation}\label{Hcubic}
H(p+tv) = a(v)t^3 + b(v)t^2 + c(v)t\,,
\end{equation}
where
\begin{equation}\label{abc_cubic}
   a(v) = \oneover{3!}H'''|_{p}(v)\,,
    \quad
   b(v) = \half H''|_p(v)\,,
   \quad 
   c(v) = H'|_p(v)\,.
\end{equation}
Assume that $b(v) \geq \lambda$ for all $v$, and some $\lambda >0$. 
Let's define, for $|v|=1$, 
\begin{equation}\label{t1plus_cubic}
t_*(v) = \frac{2|c(v)|}{b(v)}\,,\qquad\quad 
t_1^+(v) = \frac{b(v)}{2|a(v)|}\left( 1+\sqrt{1-\frac{4 |a(v)c(v)|}{b^2(v)}}\right)\,.
\end{equation}
These are direction-dependent versions of $t_*,t_1^+$ from Lemma \ref{lem:cubic}. Notice that when $c(v)=0$ they also correspond to the endpoints of the interval in \eqref{gta0} in Lemma \ref{lem:cubic0}. 
We wish to use Lemmas \ref{lem:cubic}, \ref{lem:cubic0} to say something about the region where $H(q)>0$. 
Let's make the following assumption: 
\begin{equation}\label{disc0}
    \max_{|v|=1} \frac{|a(v)c(v)|}{b^2(v)} \;\;<\;\; \frac{1}{4}\,.
\end{equation}
We call this the \emph{weak discriminant condition with constant $1/4$}. 
It implies that $|a(v)| < b^2(v)/4|c(v)|$, and $t_1^+(v)\in \R$. 
Then we may apply Lemmas \ref{lem:cubic}, \ref{lem:cubic0} to say that 
\begin{equation}
    H(p+tv) > 0 \qquad \text{for } |t|\in I(v) \equiv (t_*(v),t_1^+(v))\,,
\end{equation}
and furthermore, the interval $I(v)$ is nonempty. See the schematic in Figure \ref{fig:squiggly_lines}.

Because $t_*(v),t_1^+(v)$ are continuous functions of $v$, they each sweep out a barrier in configuration space, that bounds configurations with the same energy as $p$. For example, if we have a continuous path $q(t)$ such that $H(q(t))\leq H(p)$, then we must either have that $|q(t)-p| \leq t_*(v_q)$ for all $t$, where $v_q=(q(t)-p)/|q(t)-p|$ is the unit vector in the direction $q(t)-p$. That is, the path lies inside the barrier formed by $t_*(v)$ over all directions $v$. Or, we must have that $|q(t)-p| \geq t_1^+(v_q)$, i.e. the path lies outside the surface formed by $t_1^+(v)$. 
With some care, (which we apply in proving the more general Proposition \ref{prop:energy}), one can obtain that a path within the inside barrier must satisfy $|q(t)-p| \leq \eta_1^*$, where $\eta_1^*$ is defined by 
\begin{equation}\label{eta1_cubic}
   \eta_1^* = \max_{|v|=1}\frac{2|c(v)|}{b(v)}\,.
\end{equation}

It is sometimes the case that these barriers contain an annulus 
between them. When is this the case? We need that
\[
\max_{|v|=1} t_*(v) < \min_{|v|=1} t_1^+(v)\,.
\]
A sufficient (but not necessary) condition for this to hold (cf. \eqref{t1plus_cubic}) is
\[
\max_{|v|=1}\frac{2|c(v)|}{b(v)} < \left(\min_{|v|=1}\frac{b(v)}{2|a(v)|}\right)\left( 1 + \sqrt{1-4\left(\max_{|v|=1}\frac{|a(v)|}{b(v)}\right)\left(\max_{|v|=1}\frac{|c(v)|}{b(v)}\right)}\right)\,.
\]
Rearranging gives a sufficient condition for an annulus,
\begin{equation}\label{disc1}
    \left(\max_{|v|=1}\frac{|a(v)|}{b(v)}\right)\left(\max_{|v|=1}\frac{|c(v)|}{b(v)}\right) \;\;<\;\; \frac{1}{4}\,.
\end{equation}
We call this the \emph{strong discriminant condition with constant $1/4$}. 
When it holds, we obtain a radius $\eta_2^*$, defined as
\begin{equation}\label{eta2_cubic}
    \eta_2^* = \eta_0^*\Big(1+\sqrt{1-\frac{\eta_1^*}{\eta_0^*}}\Big)\,,
    \end{equation}
with 
\begin{equation}\label{eta0_cubic}
\eta_0^* = \min_{|v|=1}\frac{b(v)}{2|a(v)|}\,,
\end{equation}
such that $H(q) > 0$ for $|q|\in (\eta_1^*,\eta_2^*)$. 
In this case, continuous paths $q(t)$ with energy no greater than $H(p)$, must either lie entirely in the $\eta_1^*$-ball, or entirely outside the $\eta_2^*$-ball. 
For this reason we may refer to $\eta_1^*$, $\eta_2^*$ as the \emph{inner} and \emph{outer} radius, respectively. See Figure \ref{fig:squiggly_lines} for a schematic. The quantity $\eta_0^*$ can also be interpreted as a radius.

\medskip

Now we return to the weak discriminant condition \eqref{disc0}, and ask what this says about the minimum energy barrier to move from $p$ to some point outside a ball of some radius $\eta$. 
We show that for small enough $\eta$, if $|q-p|=\eta$, then $H(q) \geq \Delta H_{\rm min}(\eta)$, where $\Delta H_{\rm min}(\eta)$ is the energy barrier, a function of $\eta$. 

For $t>0$, we have that
\[
H(p+tv) \geq f(v,t)\equiv -|a(v)|t^3 + b(v)t^2 - |c(v)|t\,,
\]
since to get the lower bound, we made negative some terms that may have been positive. This also follows from Lemma \ref{lem:cubicbounds} with fixed $a$. 
A useful way to write $f$ is to factor it as 
\begin{equation}\label{fcubic}
f(v,t) = b(v)t^2\left(1 - \frac{|a(v)|}{b(v)}t - \frac{|c(v)|}{b(v)}\frac{1}{t} \right)\,.
\end{equation}
Since $b(v)\geq \lambda$, and $-|c(v)|/b(v)\geq -\eta_1^*/2$, evaluating \eqref{fcubic} at $t=\eta$ and substituting the known bounds gives 
\[
f(v,\eta) \geq \lambda \eta^2\left(1 - \frac{|a(v)|}{b(v)}\eta - \frac{1}{2}\frac{\eta_1^*}{\eta} \right) \,.
\]
To bound the remaining term, suppose that 
\begin{equation}\label{etabound_cubic}
    \eta \leq \eta_3^*\,,
\end{equation}
where 
\begin{equation}\label{eta3star_cubic}
   \eta_3^* =\minv t_c(v)\,, \qquad\; \text{with} \quad  t_c(v) = \frac{2b(v)}{3|a(v)|}\,.
\end{equation}
We call $\eta_3^*$ the \emph{energy barrier} radius. 
The function $t_c(v)$ is an approximation to the local maximum of  $f(v,t)$ as a function of $t$. It comes from  simplifying the true expression by setting $c(v)=0$.\footnote{The true local maximum occurs at $\frac{b(v)}{3|a(v)|}\Bigg(1 + \sqrt{1-3\frac{|a(v)c(v)|}{b^2(v)}}\Bigg)$.} The motivation for defining this function is that we can't obtain a lower bound that is larger than the local maximum of $f$
in any direction. 

Under condition \eqref{etabound_cubic}, $\eta|a(v)|/b(v) \leq t_c(v)|a(v)|/b(v) = 2/3$, so
\begin{equation}\label{fbound}
f(v,\eta) \geq \frac{1}{3}\lambda\:\eta^2\left(1 - \frac{3}{2}\frac{\eta_1^*}{\eta} \right) \,.
\end{equation}
The right-hand side is a lower bound for the energy barrier when $\eta \leq \eta_3^*$. It strictly increases with $\eta$ and is largest when $\eta=\eta_3^*$. 
Therefore, we have shown that if $|q-p|=\eta\leq\eta_3^*)$, then
\begin{equation}\label{Hmin_cubic}
    H(q)\geq \Delta H_{\rm min}(\eta) = \frac{1}{3}\lambda \eta^2\left( 1 - \frac{3}{2}\frac{\eta_1^*}{\eta}\right)\,.
\end{equation}

So far to get an energy barrier we have only assumed that $b(v)>0$. 
Extra assumptions will now be used  to establish that 
$\Delta H_{\rm min}(\eta)$ is positive. Clearly, $\Delta H_{\rm min}(\eta)>0\Leftrightarrow\eta > \frac{3}{2}\eta_1^*$, which in turn requires that the set $((3/2)\eta_1^*,\eta_3^*]$ be nonempty, i.e.  $\eta_3^* > \frac{3}{2}\eta_1^*$. Direct calculation shows that 
\begin{equation}\label{disc3}
    \eta_3^* > \frac{3}{2}\eta_1^* \qquad \Leftrightarrow \qquad
    \left(\max_{|v|=1}\frac{|a(v)|}{b(v)}\right)\left(\max_{|v|=1}\frac{|c(v)|}{b(v)}\right) \;\;<\;\; \frac{2}{9}\,.
\end{equation}
Therefore, a strong discriminant condition with constant $2/9$ must hold exactly when the energy barrier in \eqref{Hmin_cubic} is positive. 

This is a slightly stronger than condition \eqref{disc1} required for $\eta_1^*<\eta_2^*$, which had a constant of $2/8$. It is stronger because we have approximated the local maximum of $f$, instead of using its true value, and also because we have bounded the terms in $f$ separately, instead of considering a more global optimization over the ball. 

We may also wish to know that $\eta_3^* < \eta_2^*$. This is not necessary for the energy barrier to exist and to be positive, however when it holds, it implies that to continuously move from $p$ to another point with the same energy but that lies outside the ball of radius $(3/2)\eta_1^*$, a path must cross an energy barrier of at least $\Delta H_{\rm min}(\eta_3^*)$. A sufficient condition for $\eta_3^* < \eta_2^*$ is that $t_c(v) < t_1^+(v)$ for each $v$. A necessary and sufficient condition for the latter is a weak discriminant condition with constant $2/9$: 
\begin{equation}\label{disc2}
\max_{|v|=1} \frac{|a(v)c(v)|}{b^2(v)} \;\;<\;\; \frac{2}{9}\,.
\end{equation}
This will clearly be satisfied if the stronger condition \eqref{disc3} holds. However, for a general, non-cubic, function $H$ we won't have such a simple implication. 

\medskip

Let's summarize what we have informally shown 
by some ``temporary'' propositions, which will each have an analogue in the propositions we will prove for a general energy function in the next section.

\paragraph{Proposition I$'$}
Given a cubic energy $H$ of the form \eqref{Hcubic}, with $b(v)>0$ for all $v$. Suppose the weak discriminant condition with constant $1/4$ holds, \eqref{disc0}. Suppose we are given a continuous path $q(t)$ with $q(0)=p$, and such that $H(q(t)) \leq H(p)$. Then $|q(t)-p| \leq \eta_1^*$, where $\eta_1^*$ is defined in \eqref{eta1_cubic}. 

\paragraph{Proposition II$'$}
Given a cubic energy $H$ of the form \eqref{Hcubic}, with $b(v)>0$ for all $v$. Suppose the strong discriminant condition with constant $1/4$ holds, \eqref{disc1}. Then $H(q) > 0$ for $|q|\in (\eta_1^*, \eta_2^*)$, where $\eta_1^*, \eta_2^*$ are defined in \eqref{eta1_cubic},\eqref{eta2_cubic}, and furthermore this annulus is nonempty. 
\bigskip

As a corollary, we obtain that if we have a continuous path $q(t)$ such that $H(q(t)) \leq H(p)$ and $|q(0)-p| > \eta_1^*$, then  $|q(t)-p| \geq \eta_2^*$, i.e. this path is sufficiently far away from $p$. 

\paragraph{Proposition III$'$}
Given a cubic energy $H$ of the form \eqref{Hcubic}, with $b(v)>0$ for all $v$. 
Suppose the strong discriminant condition with constant $2/9$ holds, \eqref{disc2}. Then if $|q-p| = \eta\in ((3/2)\eta_1^*,\eta_3^*]$, we must have $H(q) \geq \Delta H_{\rm min}(\eta)$, where $\Delta H_{\rm min}(\eta)$ is given in \eqref{Hmin_cubic}. Furthermore, $\Delta H_{\rm min}(\eta) > 0$, and the set  $((3/2)\eta_1^*,\eta_3^*]$ is nonempty.

\subsection{General energy function}\label{sec:energypropositionsgeneral}

Now we consider the general forms of the propositions, where we make no assumptions on the specific functional form of $H$ except that it is three times continuously differentiable. Our reasoning will follow that of the cubic case above, but in order to use the Taylor expansion, we must bound the third derivative everywhere in the
ball in which we wish to control the energy.

Consider a function $H\in C^3(\R^n)$ with $H(p)=0$. 
Assume that 
\begin{equation}\label{Hpp}
    \min_{|v|=1}H''|_p(v) \geq 2\lambda
\end{equation}
for some $\lambda >0$.  
Let us Taylor-expand the function along some particular direction $v\in \Sball{n}$ using Taylor's theorem for one-dimensional functions, as
\begin{equation}\label{Hcubic_expand}
H(p+tv) = t^3 a_t(v) + t^2 b(v) + tc(v)\,.
\end{equation}
Here
\begin{equation}\label{abc}
b(v) = \half H''|_p(v)\,,\quad c(v) = H'|_p(v)\,,\quad\text{and }\; a_t = \oneover{3!}H'''|_{p+sv}(v) \;\;\text{for some }s\in [0,t]\,.
\end{equation}
This expansion is reminiscent of \eqref{Hcubic}, but now the coefficient of the cubic term can vary with $t$. To handle this variation, it is convenient to define the function 
\begin{equation}\label{abar}
    \bar a(v,r) = \max_{|s|\leq r}\frac{1}{3!}\Big| H'''|_{p+sv}(v)\Big|\,.
\end{equation}
We have $\bar a(v,r)< \infty$ since the maximum is over a compact set, and, by construction, $|a_t(v)|\leq \bar a(v,r)$ for $t\leq r$. 
This function will play the role of $\bar a$ in Lemmas \ref{lem:cubic}, \ref{lem:cubic0}. It replaces $|a(v)|$ from Section \ref{sec:energypropositionscubic}, and introduces a radius dependence in many of the quantities we studied earlier. 
Notice that, by construction, $\bar a(v,r)$ is an increasing function of $r$, and it is also continuous in $v,r$. 

Let us define 
\begin{equation}\label{tstar}
    t_*(v) = \frac{2|c(v)|}{b(v)}\,.
\end{equation}
This is the same definition given in \eqref{t1plus_cubic} for a cubic energy. 
We also define, for $r$ such that $\max_{|v|=1}\frac{4\bar a(v,r)|c(v)|}{b^2(v)}<1$, 
\begin{equation}\label{t1plus}
t_1^+(v,r) = \frac{b(v)}{2\bar a(v,r)}\left( 1+\sqrt{1-\frac{4\bar a(v,r)|c(v)|}{b^2(v)}} \right)\,.
\end{equation}
This is the generalization of $t_1^+(v)$ in \eqref{t1plus_cubic}. 
In its domain of definition, it is a continuous function, and furthermore it is a decreasing function of $r$, since $\bar a(v,r)$ is an increasing function of $r$.

Continuing, we define
\begin{equation}\label{tc}
    t_c(v,r) = \frac{2b(v)}{3\bar a(v,r)}\,.
\end{equation}
This generalizes the approximate critical point $t_c(v)$ from \eqref{eta3star_cubic}. It is also a continuous in $v,r$, and a decreasing function of $r$. 

We now turn to generalizations of the inner, outer, and energy barrier radii, from \eqref{eta1_cubic}, \eqref{eta2_cubic}, \eqref{eta3star_cubic}. 
Finding these radii is no longer as straightforward, since they were previously computed from the cubic coefficient $|a(v)|$, but now this coefficient depends on the radius $r$ over which we are looking. Therefore, we must be careful to ensure that the values of these radii estimated using $\bar a(v,r)$, also lie within the radius $r$ over which we control the cubic coefficient. 

Let
\begin{equation}\label{eta1star}
    \eta_1^* = \max_{|v|=1} t_*(v)\,.
\end{equation}
This will be the inner radius. It has not changed from \eqref{eta1_cubic}. 

Let 
\begin{equation}\label{eta2star}
  \eta_2^* = \min_{t_1^+(r) \leq r} r\,, 
    \qquad \text{with } \quad 
    t_1^+(r) = \min_{|v|=1}t_1^+(v,r)\,.
\end{equation}
(Set $\eta_2^* = \infty$ if the minimum doesn't exist within the domain of $t_1^+(r)$.) 
The number $\eta_2^*$ will be the outer radius. It is no longer as simple to compute as in  \eqref{eta2_cubic}, because we must ensure that it is not bigger than the radius used to estimate it. 

Let 
\begin{equation}\label{eta3star}
    \eta_3^* = \min_{|v|=1} \eta_3^*(v)\,, 
    \qquad \text{with } \quad 
    \eta_3^*(v) = \min_{t_c(v,r)\leq r}r\,.
\end{equation}
(Set $\eta_3^*,\eta_3^*(v) = \infty$ whenever the minimum defining these quantities doesn't exist.)
The number $\eta_3^*$ will be the energy barrier radius, a generalization of \eqref{eta3star_cubic}. 
For each direction,  $\eta_3^*(v)$ is the maximum radius 
computed using the formula in \eqref{tc}, but such that the radius is still no larger than the radius $r$ over which we control the cubic coefficient. A useful way to control $\eta_3^*(v)$ is via the following lemma.

\begin{lemma}\label{lem:eta3starv}
If $t_c(v,\eta)< \eta$ for some $\eta>0$, then $t_c(v,\eta)\leq\eta_3^*(v)<\eta$. 
Furthermore $t_c(v,\eta_3^*(v)) = \eta_3^*(v)$.
\end{lemma}

\begin{proof}
The function $g(r)\equiv t_c(v,r)-r$ is continuous and decreasing in $r$. 
We are given that $g(\eta) < 0$, and we also have that $g(0) = 2b(v) / (3\bar a(v,0)) > 2\lambda / (3\bar a(v,0)) > 0$, Therefore, $g(r)$ must first cross zero at some point $\eta_3^*(v) \in (0,\eta)$ at which the desired inequalities and equality hold.
\end{proof}

We will now study how the radii are related. There are two types of conditions needed. 




\begin{definition*}
A \emph{weak discriminant condition with constant $C$ in a ball of radius $r$}, call this $\weak{C,r}$, is a condition of the form 
\begin{equation}\label{disc_strong_general}
    \max_{|v|=1} \frac{\bar a(v,r)|c(v)|}{b^2(v)} \;\; < \;\; C\,.
\end{equation}
\end{definition*}

\begin{definition*}
A \emph{strong discriminant condition with constant $C$ in a ball of radius $r$}, call this $\strong{C,r}$, is a condition of the form 
\begin{equation}\label{disc_weak_general}
  \left(  \max_{|v|=1} \frac{\bar a(v,r)}{b(v)}\right)  \
  \left( \max_{|v|=1} \frac{|c(v)|}{b(v)} \right)\;\; < \;\; C\,.
\end{equation}
\end{definition*}

\begin{lemma}\label{lem:strongweakrelations}
We have the following implications:
\begin{enumerate}[(i),nosep]
    \item $\strong{C,r} \Rightarrow \weak{C,r}$;
    \item $\strong{C,r} \Rightarrow \strong{C,r'}$ for all $r' \leq r$;
    \item $\weak{C,r} \Rightarrow \weak{C,r'}$ for all $r' \leq r$;
    \item $\strong{C,r} \Rightarrow \strong{C',r}$ for all $C' \geq r$;
    \item $\weak{C,r} \Rightarrow \weak{C',r}$ for all $C' \geq r$;
    \item $\strong{C,r} \Rightarrow \strong{C,r^+}$ for some $r^+ > r$. 
    \item $\weak{C,r} \Rightarrow \weak{C,r^+}$ for some $r^+ > r$. 
\end{enumerate}
\end{lemma}

\begin{proof}
Most of these are obvious. The latter two implications follow because the weak and strong conditions are open, so we can always expand the radius a little bit. 
\end{proof}

A useful fact is that if $\weak{\frac{1}{4},r}$ holds for some $r$, then $t_1^+(v,r)$ is defined for that value of $r$. This follows directly from the definition of $t_1^+$, \eqref{t1plus}. 





\begin{lemma}\label{lem:eta1eta2}
Suppose that $\strong{\frac{1}{4},\eta_1^*}$ holds. Then $\eta_1^* < \eta_2^*$. 
\end{lemma}

\begin{proof}
To show this, we must show that $t_1^+(v,\eta_1^*) > \eta_1^*$ for all $|v|=1$. But we have
\begin{align*}
t_1^+(v,\eta_1^*) &\geq \left(\min_{|v|=1}\frac{b(v)}{2\bar a(v,\eta_1^*)} \right) \left( 1+\sqrt{1-\max_{|v|=1}\frac{4\bar a(v,\eta_1^*)|c(v)|}{b^2(v)}}\right)\\
&\geq \left(\max_{|v|=1}\frac{2\bar a(v,\eta_1^*)}{b(v)} \right)^{-1} \left( 1+\sqrt{1-4\left(\max_{|v|=1}\frac{\bar a(v,\eta_1^*)}{b(v)}\right)\left(\max_{|v|=1}\frac{|c(v)|}{b(v)}\right)}\right)\,.
\end{align*}
Therefore, a sufficient condition for $\eta_2^*>\eta_1^*$ is that the final expression be $>\eta_1^*$ for all $v$. Rearranging to isolate $\left(\max_{|v|=1} \frac{\bar a(v,\eta_1^*)}{b(v)}\right)\left(\max_{|v|=1} \frac{|c(v)|}{b(v)}\right)$ gives the sufficient condition 
\[
    \left(\max_{|v|=1} \frac{\bar a(v,\eta_1^*)}{b(v)}\right)\left(\max_{|v|=1} \frac{|c(v)|}{b(v)}\right) < \frac{1}{4}\,.
\]
\end{proof}

\begin{lemma}\label{lem:eta1eta3}
Suppose that $\strong{\frac{2}{9},\frac{3}{2}\eta_1^*}$ holds. Then $\eta_3^* > \frac{3}{2}\eta_1^*$. 
\end{lemma}

\begin{proof}
A necessary and sufficient condition for the conclusion to hold is $t_c(v,\frac{3}{2}\eta_1^*) > \frac{3}{2}\eta_1^*$ for all $|v|=1$. This is equivalent to asking that 
\[
\left(\maxv \frac{\bar a(v,\eta_1^*)}{b(v)}\right)\left(\maxv \frac{|c(v)|}{b(v)}\right) < \frac{2}{9}\,,
\]
which is the condition given in the lemma. 
\end{proof}

\begin{lemma}\label{lem:eta3eta2}
Suppose that for some number $\eta_3$, we have 
$\eta_3\leq\eta_3^*$ and that $\weak{\frac{2}{9},\eta_3}$ holds. Then $\eta_2^* > \eta_3$. 
\end{lemma}

\begin{proof}
We must show that $t_1^+(v,\eta_3)> \eta_3$ for all $v\in\Sball{n}$. Since $\eta_3\leq \eta_3^*$, we have that $t_c(v,\eta_3)\geq \eta_3$ for all $v$. 
Therefore, it is sufficient to show that $t_1^+(v,\eta_3) > t_c(v,\eta_3)$ for all $v$. Straightforward algebra gives that
\[
t_1^+(v,\eta_3) > t_c(v,\eta_3) \qquad \Longleftrightarrow \qquad 
   \frac{\bar a(v,\eta_3)|c(v)|}{b^2(v)} < \frac{2}{9}\,.
\]
\end{proof}



We are now in a position to state the energy propositions. The first proposition tells us when a point $p$ cannot be continuously deformed by very much, without increasing its energy. 

\begin{mainprop}\label{prop:energy}
Given an energy function $H(q)\in C^3(\R^{n})$ such that \eqref{Hpp} holds, and define $\eta_1^*$ as in \eqref{eta1star}. Suppose that $\weak{\frac{1}{4},\eta_1^*}$ holds:
\begin{equation}\label{weak1}
\max_{|v|=1,|s|\leq \eta_1^*}
\frac{|H'_p(v)H'''_{p+sv}(v)|}{(H''_p(v))^2} \quad <\;\;  \frac{3}{8}\,.
\end{equation}
Suppose we are given a continuous path $q(t)$ with $q(0)=p$, and such that $H(q(t)) \leq H(p)$. Then $|q(t)-p| \leq \eta_1^*$. 
\end{mainprop}

\begin{proof}
Taylor-expand the energy $H(p+vt)$ in a particular direction $v\in\Sball{n}$ with $|t|\leq \eta_1^*$ as in \eqref{Hcubic_expand}. Since, by the weak discriminant assumption, $|a_t(v)|\leq |a(v,t)|\leq\bar a(v,\eta_1^*) < b^2(v) / 4|c(v)|$ for $|t|\leq \eta_1^*$, we may apply Lemmas \ref{lem:cubic}, \ref{lem:cubic0} to obtain intervals in which the energy is positive. 

Lemma \ref{lem:cubic} says that, for $c(v)\neq 0$, 
\[
H(p+vt) > 0 \quad\text{for}\quad |t|\in \big[t_*(v),\min(t_1^+(v,\eta_1^*), \eta_1^*)\big)\,,
\]
where $t_1^+(v,r)$ is defined in \eqref{t1plus}. We take a minimum because the lemma only applies in regions where we control $a_t$. 
When $c(v)=0$, we have  $t_1^+(v,\eta_1^*) = b/\bar a(v,\eta_1^*)$, so Lemma \ref{lem:cubic0} gives that 
\[
H(p+vt) > 0 \quad\text{for}\quad |t|\in \big(t_*(v),\min(t_1^+(v,\eta_1^*), \eta_1^*)\big)\,. 
\]

We now wish to construct a continuous topological sphere $S$, contained in $\bar B_{\eta_1^*}(p)$ but containing $p$ in its interior, such that $H(q) > 0$ for $q\in S$. It will follow from the Jordan-Brouwer separation theorem, a generalization of the Jordan curve theorem to dimensions greater than 2, that any continuous curve which starts at $p$ and reaches the exterior of the surface must cross $S$ at some point, where $H(q)>0$.

It is tempting to construct our spherical surface $S^{\rm wrong}$ as the set of points $q = p+t_*(v)v$ for all $|v|=1$. 
Unfortunately this construction won't give us a manifold due to a small technicality:
in the directions where $c(v)=0$, we have $t_*(v)=0$, so $S^{\rm wrong}$
collapses to $p$, where $\Delta H(v,t)=0$. 
We deal with this technicality by expanding out a little bit in the directions where $t_*(v)=0$, which won't affect the resulting bounding radius. 
In directions where $t^*(v)\neq 0$, we don't need to expand the surface, but we do anyways, expanding by less as $t^*(v)$ increases, to ensure that the surface $S$ is continuous. 


We therefore construct a surface as
\begin{equation}\label{Ssurface}
S = \{q : q = p+(t_*(v)+\epsilon(v))v \}\,, 
\end{equation}
where
\[
\epsilon(v) = \min\left( \frac{t_1^+(v,\eta_1^*)-t_*(v)}{2}, \frac{\eta_1^*-t_*(v)}{2} \right)\,. 
\]
The function $\epsilon(v)$ is continuous, and hence the surface $S$ is continuous, since it is constructed parameterically from continuous functions.
For $c(v)\neq 0$, the function is such that $t_*(v)+\epsilon(v)\in [t_*(v),\min(t_1^+(v,\eta_1^*), \eta_1^*))$  and for $c(v)=0$, it is such that $t_*(v)+\epsilon(v)\in (t_*(v),\min(t_1^+(v,\eta_1^*),\eta_1^*))$. Therefore, $H(q) > 0$ on $S$. In addition, $S\subset \bar B_{\eta_1^*}(p)$, and $p$ is contained in the interior of $S$. 

Then by the Jordan-Brouwer separation theorem, any continuous path which starts at $p$ and hits the set $R^n-\bar B_{\eta_1^*}(p)$, must cross surface $S$ at least once, at which point the energy is strictly positive. 
Therefore a path whose energy does not increase must be contained in $\bar B_{\eta_1^*}(p)$, as claimed. 
\end{proof}

A second proposition tells us when the point and the low-energy ball surrounding it, are sufficiently isolated from other low-energy parts of the landscape. 

\begin{mainprop}\label{prop:energy2}
Given the setup in Proposition \ref{prop:energy}, and define $\eta_2^*$ as in \eqref{eta2star}. Suppose that $\strong{\frac{1}{4},\eta_1^*}$ holds:
\begin{equation}\label{strong1}
    \left( \max_{|v|=1,|s|\leq \eta_1^*} \frac{\big|H'|_p(v)\big|}{H''|_p(v)}\right) \left(\max_{|v|=1,|s|\leq \eta_1^*} \frac{\big|H'''|_{p+sv}(v)\big|}{H''|_p(v)}\right) \quad <\;\;  \frac{3}{8}\,.
\end{equation}
Then, $\eta_2^* > \eta_1^*$, and furthermore, there exists a number $\eta_2'$  such that $\eta_1 < \eta_2' \leq \eta_2^*$ and such that $\weak{\frac{1}{4},\eta_2'}$ holds. For any such number $\eta_2'$, $H(q) > H(p)$ for $|q|\in (\eta_1^*,\eta_2')$. 
\end{mainprop}

This proposition is weaker than that for a cubic function, because it may be the case that we can't compute $\eta_2^*$, or, even if we can, that $\weak{\frac{1}{4},\eta_2^*}$ does not hold. However, the proposition still gives an outer radius $\eta_2'$. 
We consider the issue of actually computing $\eta_2^*$, from bounds on the energy function and its derivatives, in Section \ref{sec:approxbounds}.



This proposition implies an easy corollary, that if there is some configuration $q$ such that $H(q) = H(p)$, and $|q|>\eta_1^*$, then $|q|>\eta_2$. Therefore, it gives a radius separating our current configuration $p$ and its continuous deformations, from any other embedding that is not within a radius $\eta_1^*$ of $p$. 

\begin{proof}
That $\eta_1^* < \eta_2^*$ follows directly from Lemma \ref{lem:eta1eta2}. 

Now we show that the desired number $\eta_2$ exists. By Lemma \ref{lem:strongweakrelations}, $\strong{\frac{1}{4},\eta_1^*}\Rightarrow \weak{\frac{1}{4},\eta_1^*}$, and $\weak{\frac{1}{4},\eta_1^*}\Rightarrow \weak{\frac{1}{4},\eta_2}$ for some $\eta_2 > \eta_1^*$. Therefore such a number exists. 

Now we assume we have a number $\eta_2\leq \eta_2^*$ that satisfies $\weak{\frac{1}{4},\eta_2}$, and we show the statement of the proposition. As in the proof of Proposition \ref{prop:energy}, Lemmas \ref{lem:cubic},\ref{lem:cubic0} imply that 
\[
H(p+vt) > 0 \quad\text{for}\quad |t|\in (t_*(v),\min(t_1^+(v,\eta_2), \eta_2))\,.
\]
We also must have that $t_1^+(v,\eta_2)\geq \eta_2$, since $t_1^+(v,\eta_2)\geq t_1^+(\eta_2) \geq t_1^+(\eta_2^*)=\eta_2^*\geq \eta_2$ (recall $t_1^+(r)$ is decreasing.) 
Therefore $\min(t_1^+(v,\eta_2), \eta_2)=\eta_2$, and 
$(\eta_1^*,\eta_2) \subset (t_*(v),\min(t_1^+(v,\eta_2), \eta_2))$ for all $v\in\Sball{n}$, so $H(p+vt) > 0$ for $|t|\in (\eta_1^*,\eta_2)$, as claimed.
\end{proof}

We end with a proposition that explicitly computes a lower bound for the energy barrier when we move a given distance $\eta$ away from $p$. 

\begin{mainprop}\label{prop:energy3}
Given the setup in Proposition \ref{prop:energy}, and
suppose that $\strong{\frac{2}{9},\frac{3}{2}\eta_1^*}$ holds: 
\begin{equation}\label{strong2}
    \left( \max_{|v|=1,|s|\leq \frac{3}{2}\eta_1^*} \frac{\big|H'|_p(v)\big|}{H''|_p(v)}\right) 
    \left(\max_{|v|=1,|s|\leq \frac{3}{2}\eta_1^*} \frac{\big|H'''|_{p+sv}(v)\big|}{H''|_p(v)}\right) \quad <\;\;  \frac{3}{8}\,.
\end{equation}
Then $\eta_3^* > \frac{3}{2}\eta_1^*$, where $\eta_3^*$ is defined in \eqref{eta3star}. 
Furthermore, if $|q-p|=\eta\in(\frac{3}{2}\eta_1^*,\eta_3^*]$, then 
\begin{equation}\label{DeltaH}
H(q)\geq \Delta H_{\rm min}(\eta) = \frac{1}{3}\lambda \: \eta^2\:\Big( 1-\frac{3}{2}\frac{\eta_1^*}{\eta}\Big)\,.
\end{equation}
For such values of $\eta$ the function $\Delta H_{\rm min}(\eta) >0$. 
\end{mainprop}

\begin{proof}
Assumption \eqref{strong2} and Lemma \ref{lem:eta1eta3} imply that $\eta_3^* > \frac{3}{2}\eta_1^*$. 
Clearly $\Delta H_{\rm min}(\eta) > 0\Leftrightarrow \min > (3/2)\eta_1^*$, which holds by assumption on $\eta$. 

Suppose now that an appropriate value of $\eta$ is given. 
We wish to show that $H(q) \geq \Delta H_{\rm min}(\eta)$ whenever $|q-p|=\eta$, for sufficiently small $\eta$. 
Taylor-expand the energy as in \eqref{Hcubic_expand}. Since $|a_t(v)| \leq \bar a(v,t)$, 
we may apply  Lemma \ref{lem:cubicbounds} to say that, for $t>0$, 
\[
H(p+tv) \geq f(v,t) \equiv -t^3\bar a(v,t) + t^2b(v) - t|c(v)|\,,
\]
where we defined an auxiliary function $f(v,t)$ above. 
As in \eqref{fcubic} we factor $f$ as
\[
f(v,t) = b(v)t^2\left(1 - \frac{\bar a(v,t)}{b(v)}t - \frac{|c(v)|}{b(v)}\frac{1}{t} \right)\,.
\]
We now evaluate $f(v,t)$ at $t=\eta$ and bound each term. 
Using $b(v) \geq \lambda$, $-|c(v)|/b(v) \geq -\eta_1^*/2$, gives
\[
f(v,\eta) \geq \lambda\eta^2\left(1 - \frac{\bar a(v,t)}{b(v)}\eta - \frac{\eta_1^*}{2\eta} \right)\,.
\]
To bound the remaining term, note that $\eta\leq \eta_3^*$, so, by definition \eqref{eta3star},  $\eta\leq \eta_3^*(v)$. Lemma \ref{lem:eta3starv} says that $\eta_3^*(v) =t_c^*(v,\eta_3^*(v)) \leq t_c^*(v,\eta)$. Therefore $\eta \leq t_c^*(v,\eta)$, so 
 $-\frac{\bar a(v,t)}{b(v)}\eta\leq -\frac{\bar a(v,t)}{b(v)}t_c^*(v,\eta) = 2/3$. 
 Substituting this bound gives 
 \[
 f(v,\eta) \geq \lambda \eta^2\left( \frac{1}{3} - \frac{\eta_1^*}{2}\frac{1}{\eta}\right)\,.
 \]
 The lower bound for $f(v,\eta)$ above is exactly $\Delta H_{\rm min}(\eta)$. 
\end{proof}

\begin{remark}
Under the setup to Proposition \ref{prop:energy3}, we may wish to know when the energy radius is such that $\eta_3^* < \eta_2^*$, i.e. it is not larger than the outer radius. 
If $\weak{\frac{2}{9},\eta_3^*}$ holds, then Lemma \ref{lem:eta3eta2} implies that $\eta_3^* < \eta_2^*$. 
However, condition \eqref{strong2} of Proposition \ref{prop:energy3} does not guarantee that $\weak{\frac{2}{9},\eta_3^*}$ holds. Therefore, whether or not $\eta_3^* > \eta_2^*$ must be checked separately.
%
\end{remark}

\subsubsection{Approximate bounds}\label{sec:approxbounds}

In practice we may not be able to compute $\eta_1^*,\eta_2^*,\eta_3^*$ analytically, since we may only have upper or lower bounds for the energy derivatives and their ratios. 
How can we still apply Propositions \ref{prop:energy}, \ref{prop:energy2}, \ref{prop:energy3}?
Suppose we have numbers $\eta_1,\eta_2,\eta_3$ such that
\begin{equation}\label{eta123}
  \eta_1\geq \eta_1^*\,,\qquad  \eta_2 \leq \eta_2^*\,,\qquad \eta_3\leq \eta_3^*\,.
\end{equation}
Suppose in addition we have a function $\eta_0(r)$ such that 
\begin{equation}\label{eta0r}
    \minv \frac{b(v)}{2\bar a(v,r)} \geq \eta_0(r)\,.
\end{equation}
The notation parallels that used in \eqref{eta0_cubic}, but our nonoptimal radii here are unstarred. We assume that  $\eta_0(r)$ is a continuous, positive, strictly decreasing function on $[0,\infty)$. 

The strong discriminant condition may be formulated in terms of $\eta_1,\eta_0(r)$. Since, by \eqref{eta123},\eqref{eta0r},
\[
\left(\maxv\frac{|c(v)|}{b(v)}\right)\left(\maxv \frac{\bar a(v,r)}{b(v)}\right) < \frac{\eta_1}{4\eta_0(r)}\,,
\]
we have that, using also Lemma \ref{lem:strongweakrelations}, 
\begin{equation}\label{disc_eta01}
    \frac{\eta_1}{\eta_0(r)} < 4C \quad \Rightarrow\quad \strong{C,r} 
    \quad \Rightarrow\quad 
    \weak{C,r}\,.
\end{equation}

We have the following corollaries to the energy propositions, in terms of the bounds $\eta_0(r),\eta_1,\eta_2,\eta_3$. 

\begin{corollary}\label{cor1}
Suppose that 
$\eta_1 / \eta_0(\eta_1) < 1$, where $\eta_1$ satisfies \eqref{eta123} and $\eta_0(r)$ satisfies \eqref{eta0r}. Then the conclusion of Proposition \ref{prop:energy} holds with $\eta_1$ replacing $\eta_1^*$.
\end{corollary}

\begin{proof}
By \eqref{disc_eta01}, $\eta_1 / \eta_0(\eta_1) < 1 \Rightarrow \strong{\frac{1}{4},\eta_1}\Rightarrow \weak{\frac{1}{4},\eta_1}$, and since $\eta_1\geq \eta_1^*$, Lemma \ref{lem:strongweakrelations} implies $\weak{\frac{1}{4},\eta_1^*}$. 
Therefore Proposition \ref{prop:energy} holds. But, since $B_{\eta_1^*}(p)\subset B_{\eta_1}(p)$, if a continuous path is constrained to stay in $\bar B_{\eta_1^*}(p)$, then it is also constrained to stay in $\bar B_{\eta_1}(p)$. 
\end{proof}

We see from the proof that Corollary \ref{cor1} is a weakening of Proposition \ref{prop:energy}, since it requires a strong discriminant condition to hold, not a weak one. However, with no bounds other than $\eta_0,\eta_1(r)$, we cannot prove a weak condition that is distinct from a strong condition. 

\begin{corollary}\label{cor2}
Suppose that $\eta_2 > \eta_1$ and that
$\eta_1 / \eta_0(\eta_2) < 1$, where $\eta_1,\eta_2$ satisfy \eqref{eta123} and $\eta_0(r)$ satisfies \eqref{eta0r}. Then $H(q) > H(p)$ for $|q| \in (\eta_1,\eta_2)$. 
\end{corollary}

\begin{proof}
The condition $\eta_1 / \eta_0(\eta_2) < 1$ implies  $\strong{\frac{1}{4},\eta_2}$, which in turn implies $\strong{\frac{1}{4},\eta_1^*}$, by Lemma \ref{lem:strongweakrelations}. 
By definition we have $\eta_2\leq \eta_2^*$. 
Therefore, $\eta_2$ satisfies the conditions of Proposition \ref{prop:energy2}, so we may apply this proposition directly to say that $H(q) > H(p)$ for $|q| \in (\eta_1^*,\eta_2)$. Since $(\eta_1,\eta_2)\subset(\eta_1^*,\eta_2)$, it follows that $H(q) > H(p)$ for $|q| \in (\eta_1,\eta_2)$.
\end{proof}

\begin{corollary}\label{cor3}
Suppose that $\eta_3 > \frac{3}{2}\eta_1$ and that $\eta_1 / \eta_0(\frac{3}{2}\eta_1) < \frac{8}{9}$, where $\eta_1,\eta_3$ satisfy \eqref{eta123} and $\eta_0(r)$ satisfies \eqref{eta0r}.
Suppose that $|q-p|=\eta\in(\frac{3}{2}\eta_1,\eta_3]$. Then 
\begin{equation}\label{DeltaHb}
H(q) \geq \Delta H_{\rm min}^{\rm est}(\eta) = \frac{1}{3}\lambda \: \eta^2\:\Big( 1-\frac{3}{2}\frac{\eta_1}{\eta}\Big)\,.
\end{equation}
Furthermore $\Delta H_{\rm min}^{\rm est}(\eta)>0$ over the range of appropriate values of $\eta$. 
\end{corollary}

The difference between $\Delta H_{\rm min}^{\rm est}(\eta)$ and $\Delta H_{\rm min}(\eta)$, is that in the former we have substituted $\eta_1^*\to\eta_1$, $\eta_3^*\to\eta_3$. 

\begin{proof}
If $\eta_3 > (3/2)\eta_1$ then $\Delta H_{\rm min}^{\rm est}(\eta)>0$ for $\eta \in (\frac{3}{2}\eta_1,\eta_3)$. 
The condition $\eta_1 / \eta_0(\frac{3}{2}\eta_1) < \frac{8}{9}$ implies $\strong{\frac{2}{9},\frac{3}{2}\eta_1}$, which in turn implies $\strong{\frac{2}{9},\frac{3}{2}\eta_1^*}$, by Lemma \ref{lem:strongweakrelations}. 
Therefore Proposition \ref{prop:energy3} holds. Since $\Delta H_{\rm min}^{\rm est}(\eta)\leq\Delta H_{\rm min}(\eta)$, the corollary holds too. 
\end{proof}

\begin{remark} The relationships we obtained before, that $\eta_1^* < \eta_3^* < \eta_2^*$ under certain discriminant conditions, won't necessarily hold for the approximate radii. We will need to check empirically that $\eta_1 < \eta_3 < \eta_2$, or else derive conditions that are specific to the bounds $\eta_0(r), \eta_1$. 
\end{remark}

Now we consider how to find the bounds $\eta_2,\eta_3$ in terms of $\eta_1,\eta_0(r)$. 
From the definitions \eqref{t1plus}, \eqref{tc}, we have that 
\begin{align}
    \minv t_1^+(v,r) &\geq l_{t_1^+}(r) \equiv \eta_0(r)\left( 1+ \sqrt{1-\frac{\eta_1}{\eta_0(r)}}\right)\,,\qquad (r \in [0, r_{\rm max}])\nonumber\\
    \minv t_c(v,r) &\geq l_{t_c}(r) \equiv \frac{4}{3}\eta_0(r)\,. 
    \label{lowerbounds}
\end{align}
We have defined particular lower bounds $l_{t_1^+}(r)$, $l_{t_c}(r)$ in the above equation. Each of these is a function of $\eta_1,\eta_0(r)$. In the above, $r_{\rm max}$ is the solution to $\eta_1 = \eta_0(r_{\rm max})$; when $r>r_{\rm max}$, $t_1^+(v,r)$ is not defined. 

Now, suppose that $\eta_2,\eta_3<r_{\rm max}$ are numbers that satisfy
\begin{equation}\label{eta23}
    \eta_2 \leq l_{t_1^+}(\eta_2)\,,\qquad \eta_3 \leq l_{t_c}(\eta_3)\,.
\end{equation}
It could be the case that either $\eta_2,\eta_3=\infty$. They are related to $\eta_2^*$, $\eta_3^*$ by the following lemma. 

\begin{lemma}\label{lem:eta23}
Given $\eta_2,\eta_3$ satisfying \eqref{eta23}. Then $\eta_2 \leq \eta_2^*$, $\eta_3\leq \eta_3^*$. 
\end{lemma}

\begin{proof}
We prove the lemma for $\eta_2$; the proof for $\eta_3$ is identical by changing the symbols in the natural way. 
We have that $\eta_2 \leq l_{t_1^+}(\eta_2) \leq t_1^+(v,\eta_2)$, so, by the definition of $\eta_2^*$ (see \eqref{eta2star}), we have $\eta_2^* \geq \eta_2$.  
\end{proof}


\begin{remark}
It should be clear from the proof that Lemma \ref{lem:eta23} holds for \emph{any} lower bound $l(r)$ such that $t_1^+(v,r) \geq l(r)$, not just the particular bound $l_{t_1^+}(r)$ considered in the proof. That is, if $t_1^+(v,r) \geq l(r)$, and if $\eta_2 \leq l(\eta_2)$, then $\eta_2 \leq \eta_2^*$. A similar comment holds for $\eta_3$. 
\end{remark}

Sometimes we will need a simpler bound than $l_{t_1^+}(r)$. The following lemma gives a way to construct other bounds. We will use it later in Section \ref{sec:theoremproofs}, to prove Theorem \ref{thm:outerbound}.

\begin{lemma}\label{lem:lw}
Let $w\in [0,1]$, and let $l_w(r) = (1+w)\eta_0(r) - w\eta_1$. Then $l_w(r) \leq l_{t_1^+}(r)$, and hence $l_w(r) \leq t_1^+(v,r)$, for $r\in [0,r_{\rm max}]$. 
\end{lemma}

\begin{proof}
When $r\in [0,r_{\rm max}]$, $\eta_1\leq \eta_0(r)$, so  $\sqrt{1-\eta_1/\eta_0(r)} \geq 1-\eta_1/\eta_0(r)$. 
Therefore, for $w\in [0,1]$, 
\begin{align*}
    l_{t_1^+}(r) &\geq \eta_0(r)\left(1 + 1-\frac{\eta_1}{\eta_0(r)} \right)\\
    &= \eta_0(r)\left(1 + w\left(1-\frac{\eta_1}{\eta_0(r)}\right) + (1-w)\left(1-\frac{\eta_1}{\eta_0(r)}\right) \right)\\
    &\geq \eta_0(r)\left(1 + w\left(1-\frac{\eta_1}{\eta_0(r)}\right) \right)\\
    &= (1+w)\eta_0(r) - w\eta_1\,.
\end{align*}
\end{proof}


\section{An energy function on frameworks}\label{sec:energyfunction}

In this section we construct an energy function associated with a given framework and study its derivatives. The energy function will be used in the next section to prove the theorems presented in Section \ref{sec:mainresults}.

We start with a framework $(p,\mathcal E)$, a  space $\mathcal C$ as in \eqref{Cdef}, a linear space $\mathcal V\subset \mathcal C$ with $\mathcal V^0\subset \mathcal V$, a stress $\omega$ and corresponding stress matrix $\Omega$ satisfying \eqref{OmegaV}, and scalars $\lambda,\mu_0,\kappa$ satisfying \eqref{lam0mu0}, \eqref{kappa}. Recall that all these quantities are defined in Section \ref{sec:mainresults}. 

Let $q\in \R^{dn}$, let $q_{ij} := q_i-q_j\in\R^{d}$ be the vector between vertices $(i,j)$, and let $q_{ij}^2 := q_{ij}\cdot q_{ij}$ be the squared length of that vector. Define an energy on edge $(i,j)\in \mathcal E$ with squared length $x$ to be 
\begin{equation}\label{hedge}
h_{ij}(x) = \half\kappa(x-p_{ij}^2)^2 + \omega_{ij}x\,.
\end{equation}

This energy is chosen so that its first and second derivatives evaluated at the squared edge lengths of $p$ are $(h_{ij})'(p_{ij}^2) = \omega_{ij}$, $(h_{ij})''(p_{ij}^2) = \kappa$. 
Define the total energy to be 
a sum of energies of each individual edge, as 
\begin{equation}\label{H}
H(q) = \sum_{(i,j)\in\mathcal E} h_{ij}(q_{ij}^2)\,.
\end{equation} 

The energy as constructed is artificial, designed to prove a geometrical statement, however it has a physical interpretation. 
The energy of each edge is like that of a spring (with squared rest length $p^2_{ij} - \omega /\kappa$) under tension or compression.\footnote{This energy is not a usual harmonic spring energy because it is quadratic in squared length. A harmonic spring would have an energy that is quadratic in length. Our choice is purely for mathematical convenience. } 
Its first derivative, $\omega_{ij}$, represents the tension in the edges, which is positive if they are under compression and negative if they are under extension. The second derivative, $\kappa$, plays the role of a spring constant. 

To make a more concrete physical connection,  consider the dimensions of each of the variables and parameters. Let us use SI units of m for lengths, s for times, kg for masses, and derived units such as N=kg$\cdot$m/s$^2$ for force, J = N$\cdot$m=kg$\cdot$m$^2$/s$^2$ for energy or work. Then the dimensions are\footnote{For a harmonic spring, the spring constant has dimensions of N/m and the tension has dimensions of N.}
\begin{gather}
[p]=[\eta_1]=\text{m}\,, \quad [v]=\frac{\text{m}}{\text{s}}\,,\quad [\omega]=[\Omega]=[\lambda_0]=[\lambda] = \frac{\text{N}}{\text{m}}\,,\nonumber\\ 
[\kappa] = \frac{\rm N}{\rm m^3}\,,\quad [\eta] = \rm N\,,\quad [R(u)] = [u]\,.
\label{dimensions}
\end{gather}
We see that $[h_{ij}]=[\kappa p^4] = \rm N\cdot \rm m = \rm J$ has dimensions of energy, as required. The derived quantity $L$ in \eqref{Lmu} has dimensions m, or length, so it is a lengthscale, as suggested earlier, and $\mu_0$ is dimensionless. One can similarly verify that $\eta_1,\eta_2,\eta_3$ all have dimensions of length.

\begin{remark}
If instead of proving a geometrical property, one wishes to make a statement about the energy basins of a system whose energy has the form \eqref{H} but with different $h_{ij}$, one could verify that the necessary conditions hold for this particular energy function and apply our theorems to it. 
For example, a common energy considered in molecular dynamics is
\[
H(q) = \sum_{(i,j)\in \mathcal E} U(|q_{ij}|)\,,
\]
where $U(r)$ is an interaction potential depending on the distance $r$ between a given pair of particles, such as a spring, Lennard Jones, Morse, Coulomb, or other potential. By defining $h_{ij}(q_{ij}^2) = U\left(\sqrt{q_{ij}^2}\right)$, one can compute derivatives of $h_{ij}$ at $p$ to obtain the stresses $\omega_{ij}$ and spring constants $\kappa_{ij}$, which could depend on the edge $(i,j)$. Depending on $U$, there may be additional higher-order terms, which would only affect the bounds for $H'''|_q(v)$. 
\end{remark}

To make a link to Section \ref{sec:energypropositions}, let us compute 
 $a,b,c$ as \eqref{abc}. We start by computing the $k$th directional derivative in direction $v\in \Sball{dn}$ at $q$, using \eqref{directionalderiv}:
\begin{align}
H'|_q(v) &= \sum_{(i,j)\in\mathcal E} 2h_{ij}'(q_{ij}^2)q_{ij}\cdot v_{ij}
   &&=   2\omega_q^T R(q)v  \label{H1} \\
H''|_q(v) &= \sum_{(i,j)\in\mathcal E} 4h_{ij}''(q_{ij}^2)(q_{ij}\cdot v_{ij})^2 +  2h_{ij}'(q_{ij}^2)(v_{ij}\cdot v_{ij})
   &&=  4\kappa v^TR(q)^TR(q)v + 2v^T\Omega_qv   \label{H2} \\
H'''|_q(v) &=\sum_{(i,j)\in\mathcal E} 12h_{ij}''(q_{ij}^2)(q_{ij}\cdot v_{ij})(v_{ij}\cdot v_{ij})
   &&= 12\kappa v^TR(q)^TR(v)v \label{H3}\\
 H^{(4)}|_q(v) &=\sum_{(i,j)\in\mathcal E} 12h_{ij}''(q_{ij}^2)(v_{ij}\cdot v_{ij})^4
   &&= 12\kappa |R(v)v|^2
\end{align}
We defined $\omega_q$, the stress vector at $q$, to have components $(\omega_q)_{ij}=\kappa(q_{ij}^2-p_{ij}^2) + \omega_{ij}$, with corresponding stress matrix  $\Omega_q=\Omega(\omega_q)$. 

 
From these calculations, we have
\begin{align}
b(v) &= \half H''|_p(v) = 2\kappa |R(p)v|^2 + v^T\Omega v\,, \nonumber\\
c(v) &= H'(v)|_p = 2\omega^TR(p)v\,, \nonumber\\
a(v,r) &= \frac{1}{3!}H'''|_{p+rv}(v) = 2\kappa v^T(R(p)+R(rv))^TR(v)v\,.\label{abc_R}
\end{align}
We defined $a(v,r)$ in the above to be the third derivative in direction $v$, at point $p+rv$. 
Because $a(v,r)$ is an increasing function of $r$ for this particular energy, 
\[
\bar a(v,r) = \max_{|s|\leq r}a(v,s)=a(v,r)\,.
\]
We will establish bounds on the derivatives in a series of lemmas.

\begin{lemma}\label{lem:Hb}
Given $H$ defined by \eqref{H}, then $\grad \grad H \geq 2\lambda$ on $\mathcal C\cap \Sball{dn}$. Equivalently, $b(v) \geq \lambda$ for all $v\in \mathcal C$, where $b(v)$ is defined in \eqref{abc_R}. 
\end{lemma}

\begin{proof}
We have $\half \grad\grad H = 2\kappa R(p)^TR(p) + \Omega \geq \lambda$ on $\mathcal C$ by \eqref{kappa}.
\end{proof}

A physical interpretation of this lemma is as follows:  imagine putting the edges under tension or compression using the stress $\omega$. 
If the framework undergoes a displacement that lies in $\mathcal V^0$, where the energy gradient is exactly zero, its energy increases, since \eqref{OmegaV} ensures the Hessian is positive. 
That is, if a particular velocity $v\in\mathcal V^0$ lengthens an edge for which $\omega_{ij} > 0$, the energy for this edge will increase, and similarly if it shortens an edge for which $\omega_{ij} < 0$, the energy will also increase. It could be that the velocity lengthens an edge with $\omega_{ij} < 0$ or shortens an edge with $\omega_{ij} > 0$, but summing over all edges gives that the total energy increases. 

The energy could still decrease in directions not in $\mathcal V^0$, because its gradient is not necessarily zero in these directions. To handle these directions, tighten the springs, by increasing the spring constant $\kappa$ just enough to satisfy \eqref{kappa}. Then Lemma \eqref{lem:Hb} says the Hessian of the energy is sufficiently large, so even if the energy decreases a little bit initially during a deformation, it should rapidly start increasing again. 
Including some almost flexes in 
$\mathcal V$ can help with the last step, since the Hessian in these directions is guaranteed to be large by $\Omega$, and not by increasing $\kappa$. 

Here is a technical lemma that will be needed to bound the derivatives of $H$. 

\begin{lemma}\label{lem:Rbounds}
Given an edge set $\mathcal E$, let $z$ be the maximum adjacency of the graph, i.e.   the maximum number of edges coming out of any one node.  
Let $p, \Omega,\kappa$ be as in the setup to Theorem \ref{thm:framework}. 
Then 
\begin{enumerate}
\item  For any vectors $q,u\in \R^{dn}$, 
\begin{equation}\label{Rubound}
|R(q)u| \leq 2z^{1/2}|q||u|\,.
\end{equation}
A special case is when $|u|=1$, in which case $|R(u)u| = \sum_{(i,j)\in \mathcal E} u_{ij}^2\leq 2z^{1/2}$. 

\item For $|v|\in \mathcal C\cap \Sball{dn}$, 
\begin{align}
\frac{\kappa|R(p)v|}{2\kappa |R(p)v|^2 + v^T\Omega v}
&\leq 
\frac{1}{\sqrt{2}}\frac{\kappa^{1/2}(\lambda-\mu_0)^{1/2}}{\lambda}\,.
\label{Rpvbound}
\end{align}
\end{enumerate}
\end{lemma}

\begin{proof}
\begin{enumerate}
\item We have
\begin{align*}
|R(q)u|^2 &= \sum_{(i,j)\in\mathcal E} \left( (q_i-q_j)\cdot (u_i-u_j) \right)^2\\
&\leq \sum_{(i,j)\in\mathcal E} |q_i-q_j|^2 |u_i-u_j|^2 & \text{apply Cauchy-Schwartz to each term}\\
&\leq 4\sum_{(i,j)\in\mathcal E}(|q_i|^2 + |q_j|^2)(|u_i|^2+|u_j|^2) & \text{using $|q_i-q_j|^2\leq 2(|q_i|^2+|q_j|^2)$, etc}\\
&\leq 4|u|^2\sum_{(i,j)\in\mathcal E}(|q_i|^2 + |q_j|^2) & \text{since } |u_i|^2+|u_j|^2\leq |u|^2\\
&\leq 4z|u|^2|q|^2\,.
\end{align*}
The last step comes because the maximum number of times that any term $|q_i|^2$ can appear in the sum above is $z$, so an upper bound for the sum is is $z\sum_{i=1}^n |q_i|^2 = z|q|^2$. 

\item 
Let
\[
g(x;\mu) = \frac{\kappa x}{2\kappa x^2 + \mu}\,.
\]
For any $\mu$, 
on the level sets where $v^T\Omega v=\mu$, $|v|=1$, $v\in \mathcal C$, we have 
\[
\frac{\kappa|R(p)v|}{2\kappa |R(p)v|^2 + v^T\Omega v}= g(|R(p)v|;\mu)\,.
\]
We proceed by treating $x$ as an independent variable and seek the maximum value of $g(x;\mu)$ for each fixed $\mu$, and then consider the maximum of these upper bounds over all $\mu$. The domain over which we seek to maximize is
\[
\mu\in[\mu_0,\lambda_{max}],\qquad 
x \geq  \left\{\begin{array}{ll}
\sqrt{\frac{\lambda-\mu}{2\kappa}} & \text{if } \mu\leq \lambda \\
0 & \text{if } \mu> \lambda
\end{array}\right.\,.
\] 
Let $\mathcal R(x;\mu)$ denote the range of $x$ for a fixed $\mu$. 
Here $\lambda_{max}$ is the maximum eigenvalue of $\Omega$, and the lower bound for $x$ comes from the observation that $2\kappa |R(p)v|^2+\mu \geq \lambda$, by \eqref{kappa}. 

We start by calculating the critical points of $g$ as a function of $x$. 
The derivative is
\[
g'(x;\mu) = \frac{-\kappa(2\kappa x^2-\mu)}{(2\kappa x^2 + \mu)^2} \,.
\]
The critical points are $\pm(\mu/2\kappa)^{1/2}$. 
If $\mu < 0$, then neither of these is inside the domain of optimization $\mathcal R(x;\mu)$. 
If $\mu\geq 0$, then the critical point  $x_c = (\mu/2\kappa)^{1/2}$ lies inside the domain of optimization if either (a) $\mu > \lambda$, or (b) $\mu\leq \lambda$ and $x_c \geq \sqrt{(\lambda-\mu)/2\kappa}$, which, by substituting for $x_c$, is equivalent to  $\mu \geq \half\lambda$. 
One can check that $x_c$ is a local maximum. 
Therefore, there is an interior maximum of $g(x;\mu)$ only if $\mu\geq \lambda/2$, and the maximum in this case satisfies
\[
 \max_{x\in \mathcal R(x;\mu)} g(x;\mu) =g(x_c;\mu) 
= \frac{1}{2\sqrt{2}}\frac{\kappa^{1/2}}{\mu^{1/2}} 
\;\;\leq\;\; \frac{1}{2}\frac{\kappa^{1/2}}{\lambda^{1/2}}\,,
\qquad \text{ if } \mu\geq \frac{1}{2}\lambda\,.
\]
An upper bound for $g(x;\mu)$ over all $\mu\geq \lambda/2$ is therefore $\frac{1}{2}\frac{\kappa^{1/2}}{\lambda^{1/2}}$. 

For $\mu<\lambda/2$ the maximum will be at the boundary of the domain. 
Since $g(x;\mu)\to 0$ as $x\to \infty$, the maximum will occur at the left-hand endpoints of the domain, $x_l=((\lambda-\mu)/2\kappa)^{1/2}$ (we don't need to consider $x_l=0$ since if $\mu <\frac{1}{2}\lambda$ then $\mu < \lambda$.) 
The maximum in this case is 
\begin{align*}
 \max_{x\in \mathcal R(x)} g(x;\mu) &=g(x_l;\mu) =  \frac{1}{\sqrt{2}}\frac{\kappa^{1/2}(\lambda-\mu)^{1/2}}{\lambda} &\text{ if } \mu < \frac{1}{2}\lambda\,.
\end{align*}
This is a decreasing function of $\mu$, which is smallest at the largest values of $\mu$ (where there is continuity with the critical point case,  $g(x_l;\frac{1}{2}\lambda)=g(x_c;\frac{1}{2}\lambda)$), and has its maximum at the smallest value, $\mu=\mu_0$. 
Therefore the global upper bounds is  $g(x_l;\mu_0)$, giving \eqref{Rpvbound}. 
\end{enumerate}
\end{proof}

Applying Lemma \ref{lem:Rbounds} to the specific energy function $H$ in \eqref{H} gives a bound on the ratio of the second and third derivatives. This is the main lemma we will use going forward.

\begin{lemma}\label{lem:Hbounds}
Given $H$ in \eqref{H}, and $v\in\mathcal C\cap \Sball{dn}$, we have
\begin{equation}
    \frac{\big|H'''|_{p+rv}(v)\big|}{H''|_p(v)} \leq 3\left(\frac{8z\kappa}{\lambda}\right)^{1/2}\left( \Big(1-\frac{\mu_0}{\lambda}\Big)^{1/2} + \Big(\frac{8z\kappa}{\lambda}\Big)^{1/2}r\right)\,.
\end{equation}
Equivalently, 
\begin{equation}\label{eta0rH}
    \frac{b(v)}{2a(v,r)} \;\;\geq\;\; \eta_0(r) = \frac{L}{2}\left(\bar\mu^{1/2} + \frac{r}{L} \right)^{-1}\,,
\end{equation}
where $L$, $\bar\mu$ are defined in \eqref{Lmu}. 
\end{lemma}


\begin{proof}
That \eqref{eta0rH} is equivalent to the first equation follows by direct substitution, after noting that $b(v)/2a(v,r) = (3/2) H''|_p(v) / |H'''|_{p+rv}(v)|$. 
We compute, for $v\in\mathcal C\cap \Sball{dn}$,
\begin{align*}
\frac{\big|H'''|_{p+rv}(v)\big|}{H''|_p(v)} &= \frac{12\kappa \big|v^TR(p+rv)^TR(v)v\big|}{4\kappa |R(p)v|^2 + 2v^T\Omega v}\\
&\leq 3|R(v)v|\frac{\kappa( |R(p)v| + |R(rv)v|)}{2\kappa |R(p)v|^2 + v^T\Omega v}\\
&\leq 6z^{1/2}\left(\frac{\kappa |R(p)v| }{(2\kappa  |R(p)v|^2 + v^T\Omega v)^2} + \frac{ 2z^{1/2}\kappa r} {\lambda} \right)\\
&\leq  6z^{1/2}\left(\frac{1}{\sqrt{2}}\frac{\kappa^{1/2}(\lambda-\mu_0)^{1/2}}{\lambda}+ \frac{ 2z^{1/2}\kappa r} {\lambda} \right)\\
&= \frac{3}{2}\left(\frac{8z\kappa}{\lambda}\right)^{1/2}\left( \Big(1-\frac{\mu_0}{\lambda}\Big)^{1/2} + \Big(\frac{8z\kappa}{\lambda}\Big)^{1/2}r\right)\,.
\end{align*}
In the third step we  applied \eqref{Rubound} to $|R(v)v|$ and $|R(rv)v|$, and we used that the lower bound of the denominator is $\lambda$, from \eqref{lem:Hb}. 
In the fourth step we used \eqref{Rpvbound} to bound the remaining term.
\end{proof}

It was a lot of work to derive the bound \eqref{eta0rH} in Lemma \ref{lem:Hbounds}, via \eqref{Rpvbound} in Lemma \ref{lem:Rbounds}. What would happen if we bounded the second and third derivatives separately? 

This  can be done using only \eqref{Rubound} from Lemma \ref{lem:Rbounds}, not the more complicated \eqref{Rpvbound}.  
We have, from \eqref{lem:Hb}, that $b(v) \geq \lambda$. To find an upper bound for $a(v,r)$, calculate, using \eqref{H3},  \eqref{Rubound}, and the triangle inequality $|p+rv| \leq |p|+r$, 
\[
|H'''|_q(v)| \;\;\leq\;\; 12\kappa |R(p+rv)v||R(v)v| \;\;\leq\;\; 48z\kappa (|p|+r)\,.
\]
Therefore
\begin{equation}\label{eta0bad}
\frac{b(v)}{2\bar a(v,r)} \geq \tilde \eta_0(r) = \frac{\lambda}{8z\kappa(|p|+r)} = L\left(\frac{|p|}{L}+\frac{r}{L}\right)^{-1} \,,
\end{equation}
where $L$ is defined in \eqref{Lmu} as before. 
Compared to \eqref{eta0rH}, this bound has term $|p|/L$ instead of $\bar\mu^{1/2}$, which ended up being the largest of all the terms in the examples we considered.

\section{Theorem proofs}\label{sec:theoremproofs}

This section presents the proofs of the theorems presented in Section \ref{sec:mainresults}. The main idea of the proofs, is to apply the propositions from Section \ref{sec:energypropositions}, to the particular energy function constructed in Section \ref{sec:energyfunction}.

We start by proving a lemma needed for the existence of $\kappa$ in \eqref{kappa}. 
\begin{lemma}\label{lem:AB}
Given symmetric matrices $A,B\in\R^{m\times m}$ and a real number $\lambda_0>0$ such that $B\geq 0$, and $A\geq\lambda_0$ on $\mbox{Null}(B)$. Then, for any $\lambda\in (0,\lambda_0)$, 
there exists a constant $\kappa \geq 0$ such that 
\begin{equation}
A + \kappa B \geq \lambda\,.
\end{equation}
\end{lemma}

\begin{remark}
Solving for the smallest such $\kappa$ is a convex problem, hence may be done efficiently, since 
the problem of minimizing $\kappa$ subject to $A + \kappa B \succcurlyeq \lambda I$ is in the standard ``linear matrix inequality'' form for semidefinite programs \cite{boydsdp}. 
\end{remark}



\begin{proof}
Let $X=\mbox{Null}(B)\cap \mathbb S^{m-1}$ be the space of unit vectors in the null space of $B$.  Then 
 $N=\{ v\in \mathbb S^{m-1} | v^TAv > \lambda\}$ is an open subset of the sphere that contains $X$, since 
 $A \ge \lambda_0 > \lambda$ on $X$.  Hence $\mathbb{S}^{m-1}\setminus X$ is a compact subset of the 
sphere on which $B > 0$.  It now follows that there are  constants $\mu,\alpha$ with $\alpha>0$ such that
\[
A \geq \mu, \quad 
B \geq \alpha \qquad \text{on }\; \mathbb S^{m-1} \backslash N\,.
\]
Define $\kappa = \frac{\lambda + |\mu|}{\alpha}$. Clearly $\kappa \geq 0$. On $N$, we calculate
\[
A + \kappa B \;\geq\; A \;>\; \lambda\,.
\]
On $\mathbb S^{m-1}\backslash N$, we calculate
\[
A + \kappa B \;\geq\; \mu +  \Big(\frac{\lambda + |\mu|}{\alpha}\Big)\alpha \;\geq\;  \lambda\,.
\]
Therefore $A+\kappa B \geq \lambda$ on all of $\mathbb S^{m-1}$.
\end{proof}

\begin{remark}
The lemma and its proof are similar to Lemma 3.4.1 of \cite{Connelly:1996vja}, except that instead of simply requiring $A+\kappa B>0$, we ask for a particular lower bound. This lower bound can become arbitrarily close to $\lambda$, however, we expect the constant $\kappa$ to diverge in general as $\lambda\nearrow\lambda_0$. 
From the proof, we see that 
as $\lambda \nearrow \lambda_0$, the set $\mathbb S^{m-1}\backslash N$ approaches $X$, so we expect $\alpha$, the minimum of $B$ on this set, to approach 0, since $B$ is continuous. Since the best constant we obtain in the proof is proportional to $\alpha^{-1}$, we expect $\kappa \to \infty$ as $\lambda \nearrow \lambda_0$. 
\end{remark}

Next we turn to the proofs of our main theorems. 
They are proved using the propositions in Section \ref{sec:energypropositionsgeneral}. These propositions apply to a function such that $H(p)=0$, which was purely for notational compactness. Therefore we apply these propositions to the function $\tilde H(q) = H(q)-H(p)$, where $H$ is defined in \eqref{H}. Henceforth we will remove the tilde on $H$ and assume it has already been shifted so that $H(p)=0$; this shift makes no difference to the calculations that follow.

\begin{proof}[Proof of Theorem \ref{thm:framework}]
Consider the energy function $H(q)$ defined in \eqref{H}. 
This energy is a function only of the edge lengths, and so if the edge lengths are preserved along the path $q(t)$, then $H(q(t)) = H(p)$. 

We restrict $H$ to $\mathcal C^p$, and consider only velocities $v\in \mathcal C$. Since $\mathcal C^p$ is a linear subspace of $\R^{dn}$, we may apply Proposition \ref{prop:energy} and its corresponding Corollary \ref{cor1}, treating the ambient space in this proposition as $\mathcal C^p$. 

Now we verify the conditions of this proposition and corollary. We have by \eqref{abc_R} that $|H'|_p(v)| \leq 2|\omega^TR(p)|$, and by Lemma \ref{lem:Hb} that $H''|_p(v) \geq 2\lambda$ for $v\in \mathcal C\cap \Sball{dn}$, so $\eta_1^* = \maxv 4|H'_p(v)| / H''|_p(v) \leq \eta_1$. 
We also have that $\eta_0(r)$ defined in \eqref{eta0rH} satisfies \eqref{eta0r}. 
Since condition \eqref{Framework1} is equivalent to the condition $\eta_1 / \eta_0(\eta_1) < 1$, we may apply Corollary \ref{cor1} to say that since $H(q(t))\leq H(p)$, we must have  $|q(t)-p|\leq \eta_1$.
\end{proof}

\begin{proof}[Proof of Theorem \ref{thm:outerbound}]
Consider the energy function $H(q)$ defined in \eqref{H}.
This energy is a function only of the edge lengths, so if $q$ has the same edge lengths as $p$, then $H(q) = H(p)$. As in the proof of Theorem \ref{thm:framework}, we restrict $H$ to $\mathcal C^p$, and consider only velocities $v\in \mathcal C$, and aim to apply Proposition \ref{prop:energy2} and its corresponding Corollary \ref{cor2} to show that  $H(p') > 0$ for $|p'-p|\in(\eta_1,\eta_2)$. 

We start by showing that $\eta_2\leq \eta_2^*$. 
Let $w\in [0,1]$ and let
\[
l_w(r) = (1+w)\eta_0(r) - w\eta_1\,
\]
where $\eta_0(r)$ is given in \eqref{eta0rH}. 
By Lemma \ref{lem:lw}, $l_w(r)\leq l_{t_1^+}(r)$ for $r\in [0,r_{\rm max}]$, where $l_{t_1^+}(r)$ is defined in \eqref{lowerbounds}.
Define $r_w$ to be the solution to
\[
r_w = (1+w)\eta_0(r_w) \qquad \Longrightarrow \qquad r_w = L\frac{(\bar\mu+2(1+w))^{1/2}-\bar\mu^{1/2}}{2}\,.
\]
Set
\[
\eta_2(w) = r_w - w\eta_1\,.
\]
Then, 
\[
l_{t_1^+}(\eta_2(w)) \geq l_{w}(\eta_2(w)) \geq l_w(r_w) = (1+w)\eta_0(r_w)-w\eta_1 = r_w-w\eta_1 = \eta_2(w)\,.
\]
We used that $l_{w}(\eta_2(w)) \geq l_w(r_w)$ since $l_w(r)$ is decreasing, and $\eta_2(w) \leq r_w$. 
Therefore, by Lemma \ref{lem:eta23}, $\eta_2(w)\leq \eta_2^*$. 

Now, to apply Corollary \ref{cor2} for a given $\eta_2(w)$, we must additionally verify that (a) $\eta_2(w)>\eta_1$, and (b) $\eta_1/\eta_0(\eta_2(w)) < 1$. 
Suppose (a) holds. We show this implies (b) holds also. For, 
\begin{align*}
    \frac{\eta_1}{\eta_2(w)} < 1 &\Leftrightarrow \frac{\eta_1}{r_w-w\eta_1} < 1\\
    &\Leftrightarrow \frac{\eta_1}{(1+w)\eta_0(r_w)-w\eta_1} < 1\\
    &\Leftrightarrow \frac{\eta_1}{\eta_0(r_w)} < 1\\
    &\Rightarrow \frac{\eta_1}{\eta_0(\eta_2(w))} < 1\,.
\end{align*}
The last step follows because $\eta_0(r)$ is decreasing, and $\eta_2(w) \leq r_w$.

Therefore, it remains to determine when $\eta_2(w) > \eta_1$. An equivalent condition is $(1+w)\eta_1 < r_w$, leading to 
\[
\eta_2(w) > \eta_1 \quad\Longleftrightarrow\quad \eta_1 < \underbrace{L\frac{(\bar\mu+2(1+w))^{1/2}-\bar\mu^{1/2}}{2(1+w)}}_{\equiv y}\,.
\]
For $\eta_1\geq 0$, this is equivalent to $\frac{\eta_1}{L}\Big(\bar \mu^{1/2} + \frac{\eta_1}{L} \Big) < \frac{y}{L}\Big(\bar \mu^{1/2} + \frac{y}{L} \Big)$, which, after some algebra, gives 
\begin{equation}
\eta_2(w) > \eta_1 \quad\Longleftrightarrow\quad 
\frac{\eta_1}{L}\Big(\bar \mu^{1/2} + \frac{\eta_1}{L} \Big) < \frac{1}{4(1+w)^2} \left[2(1+w) - 2w\bar \mu + 2w\bar\mu^{1/2}(\bar\mu+2(1+w))^{1/2} \right]\,.
\end{equation}
Substituting $w=1$ shows that if $\frac{\eta_1}{L}\Big(\bar \mu^{1/2} + \frac{\eta_1}{L} \Big) < \frac{1}{4} + \frac{1}{8}\left( \bar\mu^{1/2}(\bar\mu+4)^{1/2} - \bar\mu\right)$, then $\eta_2(1) > \eta_1$. This is the first case in \eqref{eta2}. 

Otherwise, substitute $w=0$, to obtain that if $\frac{\eta_1}{L}\Big(\bar \mu^{1/2} + \frac{\eta_1}{L} \Big) < 1/2$, then $\eta_2(0) > \eta_1$. This is the second case in \eqref{eta2}. 

Therefore, in both cases in \eqref{eta2}, we have $\eta_2 > \eta_1$. Therefore, we may apply Corollary \ref{cor2}, to say that $H(p')>H(p)$ for $|p'-p|\in(\eta_1,\eta_2)$, 
so such a $p'$ cannot have the same edge lengths as $p$. 
\end{proof}


We introduce a proposition to handle the energy barrier for deforming a framework, given the specific energy function in Section \ref{sec:energyfunction}. 

\begin{proposition}\label{prop:energybarrier}
Given the setup in Theorem \ref{thm:framework}, and suppose that 
\begin{equation}\label{Framework3p}
    D \quad < \quad
    \frac{1}{9}\left(\frac{8}{3} + \bar \mu^{1/2} \Big(\bar\mu+\frac{8}{3}\Big)^{1/2} - \bar \mu \right)\,,
\end{equation}
where $D$ is defined in \eqref{Framework1}. 
Then $\eta_3 > (3/2)\eta_1$, where $\eta_3$ is defined in \eqref{eta3}. 
Suppose $|q-p|=\eta\in((3/2)\eta_1,\eta_3]$ for some $q\in \mathcal C^p$. 
Let the energy function be defined as in \eqref{H}. 
Then
\begin{equation}\label{DeltaHprop}
H(q) \geq \frac{1}{3}\lambda \: \eta^2\:\Big(1-\frac{3}{2}\frac{\eta_1}{\eta}\Big) > 0\,.
\end{equation}
\end{proposition}


\begin{proof}
We verify that the conditions of Corollary \ref{cor3} are satisfied. 
By construction, $\eta_1\geq \eta_1^*$, and it is trivial to verify that $\eta_3 = (4/3)\eta_0(\eta_3)$, where $\eta_0(r)$ is defined in \eqref{eta0rH}. Therefore, by Lemma \ref{lem:eta23}, $\eta_3 \leq \eta_3^*$. 
It remains to verify that (a) $\eta_3 > (3/2)\eta_1$, and (b) $\eta_1 / \eta_0((3/2)\eta_1)) < 8/9$. 

Inequality (a) is true, if and only if
\[
\frac{\eta_1}{\eta_3} < \frac{2}{3} \quad \Leftrightarrow\quad  \eta_1 < \underbrace{\frac{L}{3}\Big(\Big(\bar\mu+\frac{8}{3}\Big)^{1/2}-\bar\mu^{1/2}\Big)}_{=y}\,,
\]
by direct substitution. 
Since the last expression (call it $y$) is strictly positive for $\bar\mu\geq 0$, the second inequality is equivalent (for $\eta_1\geq 0$) to 
\[
\frac{\eta_1}{L} \big(\bar\mu^{1/2}+\frac{\eta_1}{L} \big) \quad < \quad
\frac{y}{L} \big(\bar\mu^{1/2}+\frac{y}{L} \big)\,,
\]
which is equivalent to \eqref{Framework3}. 

Next we must show (b). Since $\eta_0(r)$ is a decreasing function, we have by (a) that 
\[
\frac{\eta_1}{\eta_0(\frac{3}{2}\eta_1)} < \frac{\eta_1}{\eta_0(\eta_3)} = \frac{\eta_1}{\frac{3}{4}\eta_3} < \frac{2}{3}\cdot\frac{4}{3} = \frac{8}{9}\,.
\]
Therefore the conditions of Corollary \ref{cor3} are satisfied, and so we apply this corollary to claim the statement of the proposition. 
\end{proof}

\begin{proof}[Proof of Theorem \ref{thm:emin}]
Consider the energy function defined in \eqref{H}, again restricted to $\mathcal C^p$. 
The energy change at $q$ is 
\begin{align*}
H(q)-H(p) &= \sum_{(i,j)\in\mathcal E}\frac{1}{2}\kappa(e_{ij}(q)-e_{ij}(p))^2 + \omega_{ij}(e_{ij}(q)-e_{ij}(p)) \\
&=\frac{1}{2}\kappa|e(q)-e(p)|^2 + \omega\cdot (e(q)-e(p))\\
&\leq \frac{1}{2}\kappa|e(q)-e(p)|^2 + |\omega||e(q)-e(p)|\,,
\end{align*}
where the last step follows by the Cauchy-Schwartz inequality. 
By Proposition \ref{prop:energybarrier},  $H(q)-H(p) \geq \frac{1}{3}\lambda \: \eta^2\:\Big(1-\frac{3}{2}\frac{\eta_1}{\eta}\Big)$, so
\[
\frac{1}{2}\kappa|e(q)-e(p)|^2 + |\omega||e(q)-e(p)| \geq \frac{1}{3}\lambda \: \eta^2\:\Big(1-\frac{3}{2}\frac{\eta_1}{\eta}\Big)\,.
\]
Solving for the minimum value of $|e(q)-e(p)|$ such that this inequality holds gives \eqref{emin}. 

For the remaining statement, notice that if $e(q)=e(p)$, then $H(q)=H(p)$. But this can only happen if $|q-p|\leq (3/2)\eta_1$ or if $|q-p|  > \eta_3$, since otherwise $H(q)-H(p) > 0$, by Proposition \ref{prop:energybarrier}. Therefore there is some point $q'$on the deformation path such that $|q'-p| = \eta_3$, at which point $|e(q') - e(p)| \geq e_{\rm min}(\eta_3)$. 
\end{proof}

\begin{proof}[Proof of Theorem \ref{thm:pss}]
Construct the energy function $H(q)\equiv H(q)$ as in \eqref{H}, 
and restrict it to $\mathcal C^p$. 
The proof of Theorem \ref{thm:framework} showed that $H$ satisfies
condition \eqref{weak1} of Proposition \ref{prop:energy}, and 
therefore the result of the Proposition applies. Let $S$ be the surface
\eqref{Ssurface} constructed in the proof of Proposition \ref{prop:energy},
which surrounds $p$ and is contained in $\bar B_{\eta_1}(p)$. Recall that
$\Delta H(q) \equiv H(q)-H(p) > 0$ for $q\in S$. 

Let $\mbox{int}(S)$ denote the (path) connected component of the 
complement of $S$ that contains $p$
and let $\overline{\mbox{int}(S)}$ be its closure.
We showed in the proof of Proposition 1 that $\overline{\mbox{int}(S)}\subset
\bar B_{\eta_1}(p)$. Since $\overline{\mbox{int}(S)}$ is a compact set,
$\Delta H$(q) achieves its minimum somewhere in the set, say at point
$q=p_{\rm pss}$. 

We know that $p_{\rm pss}\in \mbox{int}(S)$, i.e. the minimum is not on the boundary $S$, because $\Delta H(p) = 0$ so the minimum is at least as small as $0$, but $\Delta H(q)>0$ on the boundary.  
Therefore $\grad H(p_{\rm pss})=0$. Clearly $p_{\rm pss} \in \bar B_{\eta_1}(p)$. 
To show that $p_{\rm pss}$ is prestress stable, we must show that this local minimum is isolated, i.e. that $\grad\grad H(p_{pss}) >0$. Definition 3.3.1 of \cite{Connelly:1996vja} then implies that $p_{\rm pss}$ is prestress stable, since the energy function is quadratic. 


We show $\grad\grad H(p_{pss}) > 0$ by considering how the directional second derivative $H''|_q(v)$ varies with $q$ for fixed $v$, and showing that it cannot drop below zero. Let $v\in\mathcal C\cap \Sball{dn}$ be fixed, and let $q=p+tu$ with $u\in\mathcal C\cap \Sball{dn}$, $|t|\leq \eta_1$. 
Define $f(t) = H''|_{p+tu}(v)$. 
By Taylor's theorem, we have 
\[
H''|_q(v) = H''|_p(v) + t f'(s)
\]
for some $s\in[0,t]$. To show that $H''|_q(v) > 0$ it is sufficient to show that 
\begin{equation}\label{pssSufficient}
\max_{s\in[0,\eta_1]}\frac{\eta_1\big|f'(s)\big|} { H''|_p(v)} < 1\,.
\end{equation}
We calculate the required derivative explicitly: 
\begin{align*}
f'(t) &= \sum_{(i,j)\in\mathcal E}8h'_{ij}(q_{ij}^2)(q_{ij}\cdot v_{ij}) (v_{ij}\cdot v_{ij})
+ 4h''(q_{ij}^2)(q_{ij}\cdot u_{ij})(v_{ij}\cdot v_{ij})\\
&= 4\kappa \left( 2(R(q)v)^TR(u)v + (R(q)u)^TR(v)v \right)\,.
\end{align*}
Therefore, for $s\in[0,\eta_1]$,
\begin{align*}
\frac{\eta_1\big|f'(s)\big|} { H''|_p(v)} &= 
\frac{\eta_1 2\kappa \Big| 2(R(q)v)^TR(u)v + (R(q)u)^TR(v)v\Big|}{2\kappa|R(p)v|^2+v^T\Omega v}    \\
&\leq 2\eta_1\cdot 2z^{1/2} \cdot \left( \frac{2\kappa(|R(p)v| + |R(su)v|)}{2\kappa|R(p)v|^2+v^T\Omega v}
+ \frac{\kappa(|R(p)u|+|R(su)u|)}{2\kappa|R(p)v|^2+v^T\Omega v} \right)   \\
&\leq 4z^{1/2}\eta_1 \cdot \left( 2\cdot\frac{1}{\sqrt{2}}\frac{\kappa^{1/2}}{\lambda^{1/2}}\big(1-\frac{\mu_0}{\lambda}\big)^{1/2} + \frac{2z^{1/2}\cdot 2\kappa s}{\lambda} + \frac{2z^{1/2}\cdot \kappa (|p|+s)}{\lambda}   \right)  \\
&\leq \frac{\eta_1}{L} \cdot \left( 2\bar\mu^{1/2} + \frac{3\eta_1}{L} +  \frac{|p|}{L}\right)  \,.
\end{align*}
In the second step we used \eqref{Rubound} to bound $|R(u)v|$, $|R(v)v|$. 
In the third step we used \eqref{Rpvbound} to bound the first term, and \eqref{Rubound}, Lemma \ref{lem:Hb} to bound the numerators and denominators of the remaining terms, respectively.
In the final step we used $s\leq \eta_1$, and the definitions of $L,\bar\mu$ from \eqref{Lmu}. 
\end{proof}

\begin{remark}
When condition \eqref{PssBound} fails, we still obtain a local minimum value of the energy inside $\bar B_{\eta_1}(p)$, however this local minimum may not be isolated -- there could be a locus of dimension 1 or higher where the energy is constant, hence, where the framework could be deformed with constant edge lengths. 
\end{remark}

The proof of Theorem \ref{thm:pss} shows why we are unable to obtain a weaker condition than \eqref{PssBound}. The problematic term is $\frac{\kappa|R(p)u|}{2\kappa|R(p)v|^2 + v^T\Omega v}$, which comes from controlling
\[
\frac{\dd{}{t}H''|_{p+tu}(v)}{H''|_p(v)} \,.
\]
That is, we must control how much the second derivative in direction $v$ can vary, as we move along a line in direction $u$. When $u=v$, we have good control, because even when $v$ is an eigenvector corresponding to a small eigenvalue $\lambda_1$ of $\grad\grad H$, the change of this eigenvalue in its own eigendirection is controlled by $\lambda_1$ itself, hence will also be small. However, if $v$ is an eigenvector corresponding to a small eigenvalue $\lambda_1$, and $u$ is an eigenvalue corresponding to a large eigenvalue $\lambda_2$, then the rate of change of $\lambda_1$ may be determined by $\lambda_2$, hence could be large enough to cause $\lambda_1$ to cross zero. 


\section{Extension to tensegrities}\label{sec:tens}

We remark briefly on how to extend the theorems to tensegrities, leaving the details for future work. A tensegrity is a framework such that each edge is labelled either a cable, strut, or bar.  Cables cannot extend, struts cannot compress, and bars cannot change lengths (they cannot extend or compress.) A tensegrity can be represented by a triplet $(p,\mathcal E, l)$ such that $(p,\mathcal E)$ is a framework, and $l\in \{-1,0,1\}^m$ is a set of edge labels such that $l_{ij} = 1$ if edge $(i,j)$ is a cable, $l_{ij}=-1$ if edge $(i,j)$ is a strut, and $l_{ij}=0$ if edge $(i,j)$ is a bar.  
The frameworks studied in the rest of this paper are tensegrities where all edges are bars.

A tensegrity is rigid if it cannot be continuously deformed without violating one of the edge length conditions, except by rigid body motions. 
Specifically, it is rigid for edge lengths $\{d_{ij}\}_{(i,j)\in\mathcal E}$ if $p$ is an isolated solution (modulo rigid body motions) to the following equations (recall \eqref{fij}):
\begin{align}\label{fijtens}
    |p_i-p_j|^2 &\leq d_{ij}^2 && \text{if } l_{ij} = 1 \text{ (cable) }\nonumber \\
    |p_i-p_j|^2 &\geq d_{ij}^2 && \text{if } l_{ij} = -1 \text{ (strut) }\\
    |p_i-p_j|^2 &= d_{ij}^2 && \text{if } l_{ij} = 0 \text{ (bar) }\,. \nonumber
\end{align}

Connelly \& Whiteley \cite{Connelly:1996vja} define the concept of prestress stability for a tensegrity, and show that it implies rigidity of the tensegrity. They additionally give a necessary and sufficient test for the prestress stability of a tensegrity, which works as follows:  
suppose there exists a self-stress $\omega$ solving \eqref{pss} (with $\mathcal V_0$, $\mathcal W_0$ computed assuming that all edges are bars.) 
Suppose furthermore that 
\begin{equation}
\omega_{ij} > 0 \text { if edge $(i,j)$ is a cable}, \quad \omega_{ij} < 0 \text{ if edge $(i,j)$ is a strut.}
\end{equation}
Such a self-stress stress is called a \emph{strict, proper self-stress}. Then 
$(p,\mathcal E, l)$ is prestress stable, and hence rigid. 

For the converse, if $(p,\mathcal E, l)$ is prestress stable, then there exists a strict proper self-stress solving \eqref{pss} for a tensegrity $(p,\mathcal E', l')$ with some edges removed, where $\mathcal E'\subset \mathcal E$, and $l'$ is the corresponding set of labels. (To determine which edges to remove, one solves for a self-stress and removes cables or struts on which the stress is zero.) 

Testing for prestress stability of a tensegrity is similar to testing for prestress stability of a framework; an algorithm is given at the end of the section. With an $\omega$ in hand that demonstrates prestress stability (or an almost- version of it, using $\mathcal V, \mathcal W$ in place of $\mathcal V_0, \mathcal W_0$), one can then apply our setup and Theorems \ref{thm:framework}-\ref{thm:pss} nearly verbatim to the tensegrity. 
The only caveat, is that the theorems only hold for configurations  $q\in \mathcal C$ that  additionally satisfy
\begin{equation}\label{tensconstraint}
    \max_{(i,j)} |q_{ij}^2 - p_{ij}^2|  < 
    \min_{(i,j)\in \mathcal E} \frac{2|\omega_{ij}|}{\kappa}\,.
\end{equation}
That is, each edge length of $q$ cannot be too different from the corresponding edge length of $p$, in order for the theorems to make a statement about how $q$ may be related to $p$. 

Here is a brief explanation for where \eqref{tensconstraint} arises. 
The key step in proving our theorems is constructing an energy function $H(q)$ such that if $H(q) > H(p)$, then $q$ cannot satisfy the edge equations, \eqref{fij}. 

We use the same energy function for a tensegrity, see \eqref{H}. This energy is constructed as a sum of energies $h_{ij}(x)$ of each squared edge length, \eqref{hedge}. When edge $(i,j)$ is a cable, then $\omega_{ij}>0$, so $h_{ij}(x) $ increases beyond $h_{ij}(p_{ij}^2)$ if $x$ increases beyond $x=p_{ij}^2$, as desired. It also decreases initially if $x$ decreases below $x=p_{ij}^2$. 
However, it eventually increases when $x$ is small enough, and this no
longer models a tensegrity, since cables may decrease their length unrestrictedly. We have that $h_{ij}(x) > h_{ij}(p_{ij}^2)$ for $x < p_{ij}^2 - 2\omega_{ij} / \kappa$, and therefore we cannot use the overall energy function when $q_{ij}^2 < p_{ij}^2 - 2\omega_{ij} / \kappa$. A similar argument for struts implies that the energy argument breaks down for $q_{ij}^2 > p_{ij}^2 - 2\omega_{ij} / \kappa$. 

In practice we don't expect \eqref{tensconstraint} to be very restrictive. We calculated the right-hand side for examples (a), (b), (g), (h), and obtained 0.1, 0.1, 0.5, 1.0 respectively. This is a much greater edge stretching than is allowed by $e_{\rm min}^*$, for example, so is beyond the region in configuration space where the theorems apply.

\paragraph{An algorithm to test for prestress stability of a tensegrity}\label{sec:tensalgorithm}
Based on~\cite{Connelly:1996vja}, 
here is one algorithm to test for prestress stability of a tensegrity, using a slight modification of \eqref{pssalg}.
Define variables $s,t\in \R$, $a\in \R^{n_w}$, $X\in \R^{n_v\times n_v}$ as before, and furthermore construct a matrix $E\in \R^{n_t\times m}$ such that $(Ew)_{k} = l_{i_k}w_{i_k}$ where $i_k$ is the index of the $k$th cable or strut; $n_t$ is the total number of cables and struts. $E$ can be constructed as the matrix with $l$ on its diagonal, and then removing rows corresponding to bars.
Solve 
\begin{equation}\label{tensalg}
\begin{aligned}
& \underset{t,a,X}{\text{maximize}}
& & \min(s,t) \\
& \text{subject to}
& & X = \sum_{i=1}^{n_w}a_iM_i - t I_{n_v}, \\
&&& X \succeq 0, \\
&&& \left( \sum_{i=1}^{n_w}a_i^2\right) ^{1/2} \leq 1 ,\\
&&& EWa - s \mathbf{1}\geq 0.
\end{aligned}
\end{equation}
Here $\mathbf{1}\in \R^{n_t}$ is the vector of ones.
Once a solution is found, it must be checked that $t>0$ and $s>0$. If $s=0$, one can remove edges where the inequality is not strict, and try again. The self stress that demonstrates prestress stability is extracted as $\omega = \sum_{i=1}^{n_w}a_iw_i$. 


When $n_v=0$, one may solve a simpler problem, obtained by considering \eqref{tensalg} when $\mathcal V$ is empty. Solve
\begin{equation}\label{tensalg0}
\begin{aligned}
& \underset{a,s}{\text{maximize}}
& & s \\
& \text{subject to}
&&& \left( \sum_{i=1}^{n_w}a_i^2\right) ^{1/2} \leq 1 ,\quad
 EWa - s \mathbf{1}\geq 0.
\end{aligned}
\end{equation}
The program is successful if $s>0$. 
If $s=0$, one may remove edges where the inequality is not strict and try again, recalculating $n_v$ and running either \eqref{tensalg} or \eqref{tensalg0}. 
This algorithm is equivalent to the first-order test given in Theorem 2.3.2 of \cite{Connelly:1996vja}.

\section{Conclusion}\label{sec:conclusion}
We defined and studied an approximate notion of local rigidity 
called \textit{almost rigidity}. Our initial motivation was a 
quantitative alternative to the
linear perturbation analysis of the singular values  of the rigidity and stress
matrices, which are used in exact local rigidity tests such as infinitesimal 
rigidity and prestress stability.  

Our main results, Theorems \ref{thm:framework}--\ref{thm:emin}, provide such an alternative.
Using the spectral data associated associated with approximate 
rigidity and stress matrices of a framework, we can provide:
an upper bound on the distance from the initial configuration
reachable by a continuous flex ($\eta_1$ from Theorem \ref{thm:framework}); a lower bound on the distance 
from the initial configuration to any equivalent one that is not 
reachable by a continuous deformation that preserves edge lengths $(\eta_2$ from Theorem \ref{thm:outerbound}); 
and a lower bound on the minimum change in edge lengths required
to move continuously between distant equivalent configurations $(e_{\text{min}^*}$
from Theorem \ref{thm:emin}).  Together, these theorems provide an 
efficient way to certify that a framework does not move very much, and they begin to quantify the geometric barrier to deforming it.

Moreover, Theorem \ref{thm:pss}
gives a quantitative test for an almost rigid framework to be 
near a pre-stress stable, and hence, rigid framework.  One can 
interpret this as saying that the initial framework is a 
perturbation of a rigid framework.

Towards this goal, we have explored an extension of the classical second-derivative test to
the situation of nearly critical points using bounds on the first three derivatives. This may have other uses.

We close with some open problems, remaining issues, and other related
work.

\paragraph{Applications to tensegrity design}
Our results are focused on analyzing the rigidity properties 
of a given framework.  A related question is to design 
frameworks or tensegrities with pre-specified geometric, rigidity, and 
flexibility properties.  One design technique, known 
as ``form finding'' \cite{VB-form-finding,Demaine,Connelly-Back}, (loosely)
attempts to find configurations with the desired geometry within 
a space of configurations satisfying certain equilibrium stresses.

We speculate that the theory presented here can be applied to 
tensegrity design problems, with almost stresses and almost 
flexes replacing numerical approximations of exact equilibrium stresses
and infinitesimal motions. 
Furthermore, one challenge that arises in tensegrity design, is that a tensegrity may be designed to be rigid, but when it is built in practice with non-ideal edge lengths it may not be as rigid \cite{Pietroni:2017epa}. Our hope is that $\eta_2$ might correlate with how sloppy one can be and still retain rigidity, so that by incorporating $\eta_2$ into the optimization procedure itself, one can design tensegrities that are rigid under a wider variety of perturbations. 

\paragraph{Sharper constants}
Our results, while effective, are pessimistic.  The examples of the 
Siamese dipyramid and $K_{3,4}$ both illustrate that small $\eta_2$
may correlate with a first-order rigid structure feeling unusually flexible,
but not necessarily that there are other, nearby configurations with 
the same edge lengths.

While any effective estimates on the distance to the 
nearest configuration with the same edge lengths will, necessarily, be pessimistic, we haven't investigated in detail whether our 
methods can be improved.  Example  where our predicted $\eta_2$ is 
close to the distance to the nearest equivalent configuration would
be interesting in this context.

\paragraph{The effect of pinning}
To remove translations and rotations of the original configuration, 
we first make a choice of $\mathcal{C}^p$ and then perform 
our analysis.  The role of $\mathcal{C}^p$ is to ``pin down''
our original framework.  Near $p$, any choice of 
$\mathcal{C}^p$ approximates the space of point 
configurations modulo congruences, so every choice yields a 
correct analysis.  However, we don't have a geometric or 
quantitative understanding of how changing $\mathcal{C}^p$ 
changes the estimates we get.

\paragraph{Relationship to polynomial separation bounds}
Questions about the distance between isolated roots 
of univariate polynomials have been studied classically.
Generalizations for zero-dimensional varieties, starting 
with Canny's celebrated ``gap theorem'' \cite{canny}
and more recent results such as \cite{separation} also 
exist.

Translated to our setting, we could attempt to compute 
a radius analogous to 
$\eta_2$ from Theorem \ref{thm:outerbound}
using these general methods.
Given the state of the art, this would require us 
to make the (very restrictive) assumption that our input 
framework is locally rigid.  Moreover, the
bounds from \cite{separation} 
decay at least exponentially in the number of nodes 
and edges.


\paragraph{Algebraic approach}
An alternative approach, proposed in \cite{epsilon}, 
is to fix an $\epsilon > 0$ and formulate an algebraic
system that has as its solutions: frameworks equivalent to the 
input and 
that are at a distance of $\epsilon$ to it.  If no such solutions
exist, then this $\epsilon$ can be interpreted similarly to 
$\eta_1$ from Theorem \ref{thm:framework}.  
That paper attempts to use homotopy 
continuation methods \cite{wampler} to solve the
algebraic systems. 
Doing this, they are able to solve
the systems 
associated with some small examples.  It is not clear
that this approach can scale to larger examples.
In contrast, 
our methods are based on linear algebra and convex programming, 
and so we can apply them to  larger examples.

\subsection*{Acknowledgements}
M. H.-C. would like to thank Michael Overton for helpful conversations. 
SJG would like to thank Bob Connelly for discussions about frameworks and geometric energies.
The authors acknowledge support from NSF-FRG Grants DMS-1564487
and
DMS-1564473.
M. H.-C. acknowledges support from the US Department of Energy DE-SC0012296, and the Alfred P. Sloan Foundation. 

\bibliographystyle{plainnat}

\begin{thebibliography}{43}
    \providecommand{\natexlab}[1]{#1}
    \providecommand{\url}[1]{\texttt{#1}}
    \expandafter\ifx\csname urlstyle\endcsname\relax
      \providecommand{\doi}[1]{doi: #1}\else
      \providecommand{\doi}{doi: \begingroup \urlstyle{rm}\Url}\fi
    
    \bibitem[kur(2009)]{kurilpa}
    Kurilpa bridge.
    \newblock \emph{http://www.arup.com/projects/kurilpa\_bridge}, 2009.
    \newblock URL \url{http://www.arup.com/projects/kurilpa_bridge}.
    
    \bibitem[Abbott(2008)]{Abbott:2008tj}
    T~G Abbott.
    \newblock \emph{{Generalizations of Kempe's universality theorem}}.
    \newblock PhD thesis, Master's Thesis, 2008.
    
    \bibitem[Asimow and Roth(1978)]{Asimow:1978ena}
    L~Asimow and B~Roth.
    \newblock {The rigidity of graphs}.
    \newblock \emph{Transactions of the American Mathematical Society},
      245:\penalty0 279--289, 1978.
    
    \bibitem[Bertoldi et~al.(2017)Bertoldi, Vitelli, Christensen, and van
      Hecke]{bertholdi}
    Katia Bertoldi, Vincenzo Vitelli, Johan Christensen, and Martin van Hecke.
    \newblock Flexible mechanical metamaterials.
    \newblock \emph{Nature Reviews}, 2:\penalty0 17066, 2017.
    
    \bibitem[Borcea and Streinu(2017)]{Borcea:2017jf}
    Ciprian~S Borcea and Ileana Streinu.
    \newblock {New Principles for Auxetic Periodic Design}.
    \newblock \emph{SIAM Journal on Applied Algebra and Geometry}, 1\penalty0
      (1):\penalty0 442--458, January 2017.
    
    \bibitem[Canny(1988)]{canny}
    John Canny.
    \newblock \emph{The complexity of robot motion planning}, volume 1987 of
      \emph{ACM Doctoral Dissertation Awards}.
    \newblock MIT Press, Cambridge, MA, 1988.
    \newblock ISBN 0-262-03136-1.
    
    \bibitem[Connelly and Servatius(1994)]{Connelly:1994fz}
    R~Connelly and H~Servatius.
    \newblock {Higher-order rigidity--What is the proper definition?}
    \newblock \emph{Discrete {\&} Computational Geometry}, 11\penalty0
      (2):\penalty0 193--200, 1994.
    
    \bibitem[Connelly and Whiteley(1996)]{Connelly:1996vja}
    R~Connelly and W~Whiteley.
    \newblock {Second-order rigidity and prestress stability for tensegrity
      frameworks}.
    \newblock \emph{SIAM Journal on Discrete Mathematics}, 9\penalty0 (3):\penalty0
      453--491, August 1996.
    
    \bibitem[Connelly(1980)]{Connelly:1980joa}
    Robert Connelly.
    \newblock {The rigidity of certain cabled frameworks and the second-order
      rigidity of arbitrarily triangulated convex surfaces}.
    \newblock \emph{Advances in Mathematics}, 37\penalty0 (3):\penalty0 272--299,
      1980.
    
    \bibitem[Connelly and Back(1998)]{Connelly-Back}
    Robert Connelly and Allen Back.
    \newblock Mathematics and tensegrity: Group and representation theory make it
      possible to form a complete catalogue of "strut-cable" constructions with
      prescribed symmetries.
    \newblock \emph{American Scientist}, 86\penalty0 (2):\penalty0 142--151, 1998.
    \newblock URL \url{http://www.jstor.org/stable/27856980}.
    
    \bibitem[Connelly et~al.(2019)Connelly, Gortler, and Theran]{sticky}
    Robert Connelly, Steven~J. Gortler, and Louis Theran.
    \newblock Rigidity for sticky discs.
    \newblock \emph{Proceedings of the Royal Society A: Mathematical, Physical and
      Engineering Sciences}, 475\penalty0 (2222):\penalty0 20180773, 2019.
    \newblock \doi{10.1098/rspa.2018.0773}.
    
    \bibitem[Daro et~al.(2015)Daro, Gray, Guest, Micheletti, and
      Winslow]{tensegritree}
    Paola Daro, Don Gray, Simon~D. Guest, Andrea Micheletti, and Pete Winslow.
    \newblock The {K}ent {T}ensegritree project.
    \newblock In \emph{Proceedings of the International Association for Shell and
      Spatial Structures Symposium 2015: Future Visions (IASS'15: Future Visions)},
      August 2015.
    
    \bibitem[Demaine et~al.(2004)Demaine, Ochsendorf, Demaine, Greenwold, Kilian,
      and Cutler]{Demaine}
    Erik Demaine, John Ochsendorf, Martin Demaine, Simon Greenwold, Axel Kilian,
      and Barbara Cutler.
    \newblock Online course 4.491: Form-finding and structural optimization: Gaudi
      workshop.
    \newblock Online course, 2004.
    \newblock URL
      \url{https://ocw.mit.edu/courses/architecture/4-491-form-finding-and-structural-optimization-gaudi-workshop-fall-2004}.
    
    \bibitem[Demaine and O'Rourke(2007)]{Demaine:2007jh}
    Erik~D Demaine and Joseph O'Rourke.
    \newblock \emph{{Geometric folding algorithms}}.
    \newblock Linkages, Origami, Polyhedra. Cambridge University Press, Cambridge,
      Cambridge, 2007.
    
    \bibitem[Emiris et~al.(2020)Emiris, Mourrain, and Tsigaridas]{separation}
    Ioannis Emiris, Bernard Mourrain, and Elias Tsigaridas.
    \newblock Separation bounds for polynomial systems.
    \newblock \emph{J. Symbolic Comput.}, 101:\penalty0 128--151, 2020.
    \newblock ISSN 0747-7171.
    
    \bibitem[Frohmader and Heaton(2020)]{epsilon}
    Andrew Frohmader and Alexander Heaton.
    \newblock Epsilon local rigidity and numerical algebraic geometry, 2020.
    
    \bibitem[Goldberg(1978)]{Goldberg:1978df}
    Michael Goldberg.
    \newblock {Unstable Polyhedral Structures}.
    \newblock \emph{Mathematics Magazine}, 51\penalty0 (3):\penalty0 165--170, May
      1978.
    
    \bibitem[Gorkavyy and Fesenko(2018)]{Gorkavyy:2018gp}
    V~Gorkavyy and I~Fesenko.
    \newblock {On the model flexibility of Siamese dipyramids}.
    \newblock \emph{Journal of Geometry}, 110\penalty0 (1):\penalty0 1--19,
      December 2018.
    
    \bibitem[Grant and Boyd(2008)]{gb08}
    Michael Grant and Stephen Boyd.
    \newblock Graph implementations for nonsmooth convex programs.
    \newblock In V.~Blondel, S.~Boyd, and H.~Kimura, editors, \emph{Recent Advances
      in Learning and Control}, Lecture Notes in Control and Information Sciences,
      pages 95--110. Springer-Verlag Limited, 2008.
    \newblock \url{http://stanford.edu/~boyd/graph_dcp.html}.
    
    \bibitem[Grant and Boyd(2014)]{cvx}
    Michael Grant and Stephen Boyd.
    \newblock {CVX}: Matlab software for disciplined convex programming, version
      2.1.
    \newblock \url{http://cvxr.com/cvx}, March 2014.
    
    \bibitem[Henkes et~al.(2016)Henkes, Quint, Fily, and Schwarz]{henkes}
    Silke Henkes, David~A. Quint, Yaouen Fily, and J.~M. Schwarz.
    \newblock Rigid cluster decomposition reveals criticality in frictional
      jamming.
    \newblock \emph{Phys. Rev. Lett.}, 116:\penalty0 028301, Jan 2016.
    \newblock \doi{10.1103/PhysRevLett.116.028301}.
    \newblock URL \url{https://link.aps.org/doi/10.1103/PhysRevLett.116.028301}.
    
    \bibitem[Holmes-Cerfon(2017)]{HolmesCerfon:2017hz}
    Miranda Holmes-Cerfon.
    \newblock {Sticky-Sphere Clusters}.
    \newblock \emph{Annual Review of Condensed Matter Physics}, 8\penalty0
      (1):\penalty0 77--98, March 2017.
    
    \bibitem[Holmes-Cerfon(2016)]{HolmesCerfon:2016wa}
    Miranda~C Holmes-Cerfon.
    \newblock {Enumerating Rigid Sphere Packings}.
    \newblock \emph{SIAM Review}, 58\penalty0 (2):\penalty0 229--244, May 2016.
    
    \bibitem[Jackson and Jord\'{a}n(2005)]{jacksonjordan}
    Bill Jackson and Tibor Jord\'{a}n.
    \newblock Connected rigidity matroids and unique realizations of graphs.
    \newblock \emph{J. Combin. Theory Ser. B}, 94\penalty0 (1):\penalty0 1--29,
      2005.
    \newblock ISSN 0095-8956.
    \newblock \doi{10.1016/j.jctb.2004.11.002}.
    \newblock URL
      \url{https://doi-org.ezproxy.st-andrews.ac.uk/10.1016/j.jctb.2004.11.002}.
    
    \bibitem[Jacobs et~al.(2001)Jacobs, Rader, Kuhn, and Thorpe]{Jacobs:2001we}
    D~J Jacobs, A~J Rader, L~A Kuhn, and M~F Thorpe.
    \newblock {Protein flexibility predictions using graph theory}.
    \newblock \emph{Proteins-Structure Function and Genetics}, 44\penalty0
      (2):\penalty0 150--165, 2001.
    
    \bibitem[Kabsch(1976)]{Kabsch}
    W.~Kabsch.
    \newblock {A solution for the best rotation to relate two sets of vectors}.
    \newblock \emph{Acta Crystallographica Section A}, 32\penalty0 (5):\penalty0
      922--923, Sep 1976.
    \newblock \doi{10.1107/S0567739476001873}.
    \newblock URL \url{https://doi.org/10.1107/S0567739476001873}.
    
    \bibitem[Kallus and Holmes-Cerfon(2017)]{Kallus:2017hi}
    Yoav Kallus and Miranda Holmes-Cerfon.
    \newblock {Free energy of singular sticky-sphere clusters}.
    \newblock \emph{Physical Review E}, 95\penalty0 (2):\penalty0 2491--18,
      February 2017.
    
    \bibitem[Liu and Nagel(2010)]{Liu:2010jx}
    Andrea~J Liu and Sidney~R Nagel.
    \newblock {The Jamming Transition and the Marginally Jammed Solid}.
    \newblock \emph{Annual Review of Condensed Matter Physics}, 1\penalty0
      (1):\penalty0 347--369, 2010.
    
    \bibitem[Manoharan(2015)]{Manoharan:2015ko}
    V~N Manoharan.
    \newblock {Colloidal matter: Packing, geometry, and entropy}.
    \newblock \emph{Science}, 349\penalty0 (6251):\penalty0 1253751--1253751,
      August 2015.
    
    \bibitem[Patrick~Royall et~al.(2008)Patrick~Royall, Williams, Ohtsuka, and
      Tanaka]{PatrickRoyall:2008fz}
    C~Patrick~Royall, Stephen~R Williams, Takehiro Ohtsuka, and Hajime Tanaka.
    \newblock {Direct observation of a local structural mechanism for dynamic
      arrest}.
    \newblock \emph{Nature Materials}, 7\penalty0 (7):\penalty0 556--561, June
      2008.
    
    \bibitem[Paulose et~al.(2015)Paulose, Meeussen, and Vitelli]{Paulose:2015hd}
    Jayson Paulose, Anne~S Meeussen, and Vincenzo Vitelli.
    \newblock {Selective buckling via states of self-stress in topological
      metamaterials}.
    \newblock \emph{Proceedings of the National Academy of Sciences}, 112\penalty0
      (25):\penalty0 7639--7644, June 2015.
    
    \bibitem[Pietroni et~al.(2017)Pietroni, Tarini, Vaxman, Panozzo, and
      Cignoni]{Pietroni:2017epa}
    Nico Pietroni, Marco Tarini, Amir Vaxman, Daniele Panozzo, and Paolo Cignoni.
    \newblock {Position-based tensegrity design}.
    \newblock \emph{Acm Transactions on Graphics}, 36\penalty0 (6):\penalty0 1--14,
      November 2017.
    
    \bibitem[Rayneau-Kirkhope et~al.(2018)Rayneau-Kirkhope, Zhang, Theran, and
      Dias]{RayneauKirkhope:2017cy}
    Daniel Rayneau-Kirkhope, Chengzhao Zhang, Louis Theran, and Marcelo~A. Dias.
    \newblock Analytic analysis of auxetic metamaterials through analogy with rigid
      link systems.
    \newblock \emph{Proceedings of the Royal Society A: Mathematical, Physical and
      Engineering Sciences}, 474\penalty0 (2210):\penalty0 20170753, 2018.
    \newblock \doi{10.1098/rspa.2017.0753}.
    
    \bibitem[Robinson et~al.(2019)Robinson, Turci, Roth, and
      Royall]{Robinson:2019et}
    Joshua~F Robinson, Francesco Turci, Roland Roth, and C~Patrick Royall.
    \newblock {Morphometric Approach to Many-Body Correlations in Hard Spheres}.
    \newblock \emph{Physical Review Letters}, 122\penalty0 (6):\penalty0 068004,
      February 2019.
    
    \bibitem[Rocks et~al.(2017)Rocks, Pashine, Bischofberger, Goodrich, Liu, and
      Nagel]{Rocks:2017bu}
    J~W Rocks, N~Pashine, I~Bischofberger, C~P Goodrich, A~J Liu, and S~R Nagel.
    \newblock {Designing allostery-inspired response in mechanical networks}.
    \newblock \emph{Proceedings of the National Academy of Sciences}, 114:\penalty0
      2520--2525, 2017.
    
    \bibitem[Sartbaeva et~al.(2006)Sartbaeva, Wells, Treacy, and
      Thorpe]{Sartbaeva:2006ks}
    Asel Sartbaeva, Stephen~A Wells, M~M~J Treacy, and M~F Thorpe.
    \newblock {The flexibility window in zeolites}.
    \newblock \emph{Nature Materials}, 5\penalty0 (12):\penalty0 962--965, November
      2006.
    
    \bibitem[Sommese and Wampler(2005)]{wampler}
    Andrew~J. Sommese and Charles~W. Wampler, II.
    \newblock \emph{The numerical solution of systems of polynomials}.
    \newblock World Scientific Publishing Co. Pte. Ltd., Hackensack, NJ, 2005.
    \newblock Arising in engineering and science.
    
    \bibitem[Theran et~al.(2015)Theran, Nixon, Ross, Sadjadi, Servatius, and
      Thorpe]{anchored}
    Louis Theran, Anthony Nixon, Elissa Ross, Mahdi Sadjadi, Brigitte Servatius,
      and M.~F. Thorpe.
    \newblock Anchored boundary conditions for locally isostatic networks.
    \newblock \emph{Phys. Rev. E}, 92:\penalty0 053306, Nov 2015.
    \newblock \doi{10.1103/PhysRevE.92.053306}.
    \newblock URL \url{https://link.aps.org/doi/10.1103/PhysRevE.92.053306}.
    
    \bibitem[Vandenberghe and Boyd(1996)]{boydsdp}
    Lieven Vandenberghe and Stephen Boyd.
    \newblock Semidefinite programming.
    \newblock \emph{SIAM Rev.}, 38\penalty0 (1):\penalty0 49--95, 1996.
    \newblock ISSN 0036-1445.
    \newblock \doi{10.1137/1038003}.
    \newblock URL \url{https://doi-org.ezproxy.st-andrews.ac.uk/10.1137/1038003}.
    
    \bibitem[Veenendaal and Block(2012)]{VB-form-finding}
    D.~Veenendaal and P.~Block.
    \newblock An overview and comparison of structural form finding methods for
      general networks.
    \newblock \emph{International Journal of Solids and Structures}, 49\penalty0
      (26):\penalty0 3741--3753, 2012.
    \newblock \doi{10.1016/j.ijsolstr.2012.08.008}.
    
    \bibitem[White and Whiteley(1983)]{whitewhiteley}
    Neil~L. White and Walter Whiteley.
    \newblock The algebraic geometry of stresses in frameworks.
    \newblock \emph{SIAM J. Algebraic Discrete Methods}, 4\penalty0 (4):\penalty0
      481--511, 1983.
    \newblock ISSN 0196-5212.
    \newblock \doi{10.1137/0604049}.
    \newblock URL \url{https://doi-org.ezproxy.st-andrews.ac.uk/10.1137/0604049}.
    
    \bibitem[Whiteley(2005)]{Whiteley:2005hj}
    Walter Whiteley.
    \newblock {Counting out to the flexibility of molecules}.
    \newblock \emph{Physical Biology}, 2\penalty0 (4):\penalty0 S116--S126,
      December 2005.
    
    \bibitem[Yan et~al.(2017)Yan, Ravasio, Brito, and Wyart]{Yan:2017cu}
    L~Yan, R~Ravasio, C~Brito, and M~Wyart.
    \newblock {Architecture and coevolution of allosteric materials}.
    \newblock \emph{Proceedings of the National Academy of Sciences}, 114:\penalty0
      2526--2531, 2017.
    
    \end{thebibliography}

\end{document}